\documentclass[12pt]{article}
\usepackage{graphics,graphicx}
\usepackage{amssymb,amsmath,amsopn,amsfonts}
\usepackage{enumerate,bbm,stmaryrd,mathrsfs,latexsym,theorem}
\usepackage[dvips]{color}
\usepackage{colordvi,multicol}
\usepackage{epsf}
\newfam\msbfam
\font\tenmsb=msbm10 \textfont\msbfam=\tenmsb \font\sevenmsb=msbm7
\scriptfont\msbfam=\sevenmsb \font\fivemsb=msbm5
\scriptscriptfont\msbfam=\fivemsb

\newcommand{\tmrsup}[1]{\textsuperscript{#1}}
\renewcommand{\leq}{\leqslant}
\renewcommand{\geq}{\geqslant}
\newcommand{\assign}{:=}
\newcommand{\backassign}{=:}

\newtheorem{axiom}{Assumption}
\newtheorem{theorem}{Theorem}[section]
\newtheorem{lemma}[theorem]{Lemma}
\newtheorem{proposition}[theorem]{Proposition}

{\theorembodyfont{\rmfamily}\newtheorem{remark}[theorem]{Remark}}

\newcommand{\CC}{\mathcal{C} \hspace{.1em}}

\newcommand{\LL}{\mathcal{L}}

\newcommand{\VV}{\mathscr{V}}

\newcommand{\rmb}[1]{\textcolor{blue}{#1}}
\newcommand{\rmm}[1]{\textcolor{magenta}{#1}}
\newcommand{\rmg}[1]{\textcolor[rgb]{0.50,0.25,0.00}{#1}}

\newcommand{\R}{\mathbb{R}}

\newcommand{\N}{\mathbb{N}}

\newcommand{\Prec}{\prec\!\!\!\prec}
\newcommand{\8}{\infty}
\newcommand{\dd}{\mathrm{d}}

\numberwithin{equation}{section}

\newcommand{\tone}[1]{#1\!\tmrsup{\resizebox{.6em}{!}{\includegraphics{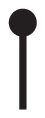}}}}
\newcommand{\tthree}[1]{#1\!\tmrsup{\resizebox{1.2em}{!}{\includegraphics{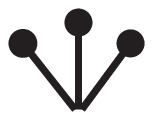}}}}
\newcommand{\ttwo}[1]{#1\!\tmrsup{\resizebox{1.2em}{!}{\includegraphics{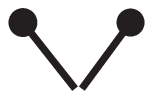}}}}
\newcommand{\ttthree}[1]{#1\!\tmrsup{\resizebox{1.2em}{!}{\includegraphics{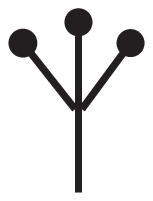}}}}
\newcommand{\tttwo}[1]{#1\!\tmrsup{\resizebox{1.2em}{!}{\includegraphics{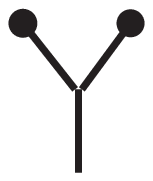}}}}
\newcommand{\ttthreetwo}[1]{#1\!\tmrsup{\resizebox{1.2em}{!}{\includegraphics{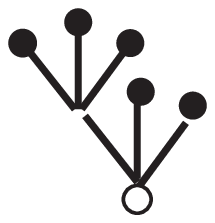}}}}
\newcommand{\tttwotwo}[1]{#1\!\tmrsup{\resizebox{1.2em}{!}{\includegraphics{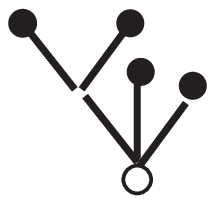}}}}
\newcommand{\ttthreeone}[1]{#1\!\tmrsup{\resizebox{1.2em}{!}{\includegraphics{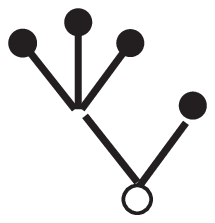}}}}

\def\R{{\mathbb R}}

\def\bc{\begin{center}}
\def\ec{\end{center}}
\def\no{\noindent}
\def\hang{\hangindent\parindent}
\def\textindent#1{\indent\llap{\qquad #1\ \ \enspace}\ignorespaces}
\def\ref{\par\hang\textindent}

\begin{document}

\title{ {\bf Weak universality of the dynamical  $\Phi_3^4$ model on the whole space
\thanks{Research supported in part  by  NSFC (No.11671035, No.11771037). Financial support by the DFG through
the CRC 1283 " Taming uncertainty and profiting from randomness and low regularity in analysis, stochastics
and their applications" is acknowledged.
}\\} }
\author{  {\bf Rongchan Zhu}$^{\mbox{a,c}}$, {\bf Xiangchan Zhu}$^{\mbox{b,c,d},}$\thanks{Corresponding author}
\date{}
\thanks{E-mail address:
zhurongchan@126.com(R. C. Zhu), zhuxiangchan@126.com(X. C. Zhu)}\\\\
\small
$^{\mbox{a}}$Department of Mathematics, Beijing Institute of Technology, Beijing 100081,  China\\
\small$^{\mbox{b}}$School of Science, Beijing Jiaotong University, Beijing 100044, China\\
\small$^{\mbox{c}}$Department of Mathematics, University of Bielefeld, D-33615 Bielefeld, Germany
\\\small $^{\mbox{d}}$ Academy of Mathematics and Systems Science,
Chinese Academy of Sciences, Beijing 100190, China}

\maketitle

\noindent {\bf Abstract} We prove the large scale convergence of a class of stochastic weakly nonlinear reaction-diffusion models on $\mathbb{R}^3$ to the dynamical $\Phi^4_3$ model by paracontrolled distributions on weighted Besov space. Our approach depends on the delicate choice of the weight, the localization operator technique and a modification version of the maximal principle from [GH18].

\vspace{1mm}
\no{\footnotesize{\bf 2000 Mathematics Subject Classification AMS}:\hspace{2mm} 60H15, 82C28}
 \vspace{2mm}

\no{\footnotesize{\bf Keywords}:\hspace{2mm}  $\Phi_3^4$ model,  paracontrolled distributions, weak universality, space-time white noise, renormalisation}

\section{Introduction}

Recall that the usual continuum Euclidean $\Phi^4_d$-quantum field  is heuristically described by the following
probability measure:
$$N^{-1}\Pi_{x\in\mathbb{R}^d}d\phi(x)\exp\bigg(-\int_{\mathbb{R}^d}(|\nabla \phi(x)|^2+\phi^2(x)+\phi^4(x))dx\bigg),$$
where $N$
is a normalization constant and $\phi$ is the real-valued field. There have been many approaches to the problem of
giving a meaning to the above heuristic measure for $d=2$ and $d=3$ (see  [GRS75, GlJ86] and references
therein).
In [PW81]  Parisi and Wu proposed a program for Euclidean quantum field
theory of getting Gibbs states of classical
statistical mechanics as limiting distributions of stochastic processes,  especially as solutions to non-linear stochastic differential
equations. Then one can use the stochastic differential equations to study the properties of the Gibbs states. This
procedure is called stochastic field quantization (see [JLM85]). The $\Phi_d^4$ model is the simplest non-trivial Euclidean quantum field (see [GLJ86] and the reference therein). The issue of the stochastic quantization of the $\Phi^4_d$ model is to solve the following equation:
\begin{equation}\aligned d \Phi= &( \Delta \Phi-\Phi^3)dt +dW(t) \quad\Phi(0)=\Phi_0.\endaligned\end{equation}
where $W$ is a cylindrical Wiener process on $L^2(\mathbb{R}^d)$. The solution $\Phi$ is also called  dynamical $\Phi^4_d$ model. (1.1) is ill-posed in both two and three dimensions.

In
two spatial dimensions, the dynamical $\Phi_2^4$ model was previously treated in [AR91], [DD03] and [MW15]. In three
spatial dimensions this equation (1.1) is ill-posed and the main difficulty in this case is that W
and hence the solutions are so singular that the non-linear term is not well-defined in the classical sense.
It was a long-standing open problem to give a meaning to equation (1.1) in the three dimensional
case. A breakthrough result was achieved recently by Martin Hairer in [Hai14], where he introduced
a theory of regularity structures and gave a meaning to  equation (1.1) successfully. He also proved existence and uniqueness of a local (in time) solution. By using the paracontrolled distributions proposed by Gubinelli, Imkeller and Perkowski
in [GIP15]  existence and uniqueness of local solutions to (1.1) has also been obtained in [CC13].
Recently, these two approaches have been successful in giving a meaning to several other ill-posed
stochastic PDEs like the Kardar-Parisi-Zhang (KPZ) equation ([KPZ86], [BG97], [Hai13]), the Navier-Stokes equation driven by space-time white
noise ([ZZ14, ZZ15]), the dynamical sine-Gordon equation ([HS14]) and so on (see [HP14]
for further interesting examples).

 Recently, a lot of results
concerning the dynamical $\Phi^4_3$ model have been obtained. In particular  the global well-posedness of the dynamical
$\Phi^4_2$ model on infinite volume  and the dynamical
$\Phi^4_3$ model on finite volume  have been obtained by Mourrat and
Weber in [MW15, MW17]. The relation between the solutions constructed in [MW15, MW17] and the associated Dirchlet form have been also considered in [RZZ17, ZZ18a].   More recently,
the global well-posedness have been proved in infinite volume case by Gubinelli and Hofmanov\'{a} in [GH18] by using paracontrolled distribution method on weighted Besov space. Lattice approximation to the dynamical $\Phi^4_3$ model on torus has been obtained in [ZZ18] and has been extended to infinite volume case to construct the Euclidean $\Phi^4_3$ field in [GH18a].

An interesting problem in the singular SPDE is called weak universality, which proves that the same limiting object describes the large scale behaviour of the solutions of more general equations. Weak universality  has been
first studied  by Hairer and Quastel in [HQ15] for the
Kardar--Parisi--Zhang equation. Later Hairer and
Xu in [HX16] proved a weak universality result with the limit given by the dynamical $\Phi^4_3$ model for three
dimensional reaction--diffusion equations driven by  Gaussian noise and a
polynomial non-linearity by the theory of regularity structures. Later, it has been extended to non-Gaussian noise by Shen and Xu in [SX16]. Moreover, it has been extended  to a large class
of non-linearity by Furlan and Gubinelli in [FG17] by paracontrolled distribution method. Most work in the literature consider the weak universality on finite volume case.   In [MP17] J. Martin and Perkowski also consider the weak universality for 2d PAM on infinite volume.
 In this paper we would like to extend the result in [FG17] to infinite volume case. We mainly use the localization operator technique (see Lemma 2.2)  and the modification of the maximal principle (see Lemma 2.4) from [GH18] to deduce the results.

In the following we recall the framework for the weak universality:
Consider the models in a weakly nonlinear regime:
\begin{equation}\LL u(t,x)=-\varepsilon^{-1}F_\varepsilon(u(t,x))+\eta(t,x),\quad (t,x)\in [0,T/\varepsilon^2]\times (\mathbb{R}/\varepsilon)^3\end{equation}
with
$\LL=\partial_{t} +(-\Delta+\mu)$, $\varepsilon\in (0,1], T>0$, initial condition $\bar{u}_{0,\varepsilon}: (\mathbb{R}/\varepsilon)^3\rightarrow\mathbb{R}$, $F_\varepsilon\in C^9(\mathbb{R})$, the condition of which will be given below. Here $\eta$ is a centered Gaussian noise with covariance
$\mathbb{E}(\eta(t,x)\eta(s,y))=\Sigma(t-s,x-y)$ if $|x-y|\leq 1$ and $0$ otherwise, where $\Sigma:\mathbb{R}\times \mathbb{R}^3\rightarrow\mathbb{R}^+$ is a smooth compactly supported in $[0,1]\times B(0,1)$. Set $u_\varepsilon=\varepsilon^{-1/2}u(t/\varepsilon^2,x/\varepsilon)$, $u_{0,\varepsilon}=\varepsilon^{-1/2}\bar{u}_{0,\varepsilon}(x/\varepsilon)$. By (1.2) we know
\begin{equation}\LL u_\varepsilon(t,x)=-\varepsilon^{-\frac{3}{2}}F_\varepsilon(\varepsilon^{\frac{1}{2}}u_\varepsilon(t,x))+\eta_\varepsilon(t,x),\end{equation}
with $\eta_\varepsilon(t,x):=\varepsilon^{-\frac{5}{2}}\eta(t/\varepsilon^2,x/\varepsilon)$ and initial value $u_{0,\varepsilon}$.
We follow the same approach as in [FG17] to expand the nonlinear term around the stationary solution $Y_\varepsilon$ to the following
linear equation
$$\LL Y_\varepsilon=\eta_\varepsilon.$$
The Gaussian r.v. $\varepsilon^{1/2}Y_\varepsilon(t,x)$ has variance $\sigma_\varepsilon=\varepsilon \mathbb{E}[(Y_\varepsilon(0,0))^2]$.
Then we can expand the random variable $F_\varepsilon(\varepsilon^{\frac{1}{2}}Y_\varepsilon(t,x))$ according to the chaos decomposition relative to $\varepsilon^{\frac{1}{2}}Y_\varepsilon(t,x)$ and obtain
$$F_\varepsilon(\varepsilon^{1/2} Y_\varepsilon(t,x))=\sum_{n\geq0} f_{n,\varepsilon}H_n(\varepsilon^{1/2}Y_\varepsilon(t,x),\sigma_\varepsilon^2),$$
where $H_n(x,\sigma_\varepsilon^2)$ are standard Hermite polynomials with variance $\sigma_\varepsilon^2$ and highest-order term normalized to $1$ and
$$f_{n,\varepsilon}=\frac{1}{n!\sigma_\varepsilon^{2n}}\mathbb{E}[F_\varepsilon(\varepsilon^{1/2}Y_\varepsilon(t,x))H_n(\varepsilon^{1/2}Y_\varepsilon(t,x),\sigma_\varepsilon^2)].$$

Let $$\tilde{F}_\varepsilon(x):=F_\varepsilon(x)-f_{0,\varepsilon}-f_{1,\varepsilon}x-f_{2,\varepsilon}H_2(x,\sigma_\varepsilon^2)=\sum_{n\geq3}f_{n,\varepsilon}H_n(x,\sigma_\varepsilon^2).$$
Let $\tilde{F}_{\varepsilon}^{ (k)}$ be the $k$-th derivative of the
function $\tilde{F}_{\varepsilon}$ for $k\in\mathbb{N}$ and define
the following $\varepsilon$--dependent constants:
\begin{equation}
  \begin{array}{lll}
    \tttwotwo{d_{\varepsilon}} & \assign & \frac{\varepsilon^{- 2}}{9}
    \int_{s, x} P_s (x) \mathbb{E} [\tilde{F}_{\varepsilon}^{ (1)}
    (\varepsilon^{1 / 2} Y_{\varepsilon} (s, x)) \tilde{F}_{\varepsilon}^{
    (1)} (\varepsilon^{1 / 2} Y_{\varepsilon} (0, 0))],\\
    \tttwotwo{\bar{d}_{\varepsilon}} & \assign & 2 \varepsilon^{- 1 / 2}
    f_{3, \varepsilon} f_{2, \varepsilon}  \int_{s, x} P_s (x)
    (C_{\varepsilon} (s, x))^2,\\
    \ttthreeone{d_{\varepsilon}} & \assign & \frac{\varepsilon^{- 2}}{6}
    \int_{s, x} P_s (x) \mathbb{E} [\tilde{F}_{\varepsilon}^{ (0)}
    (\varepsilon^{1 / 2} Y_{\varepsilon} (s, x))_{} \tilde{F}_{\varepsilon}^{
    (2)} (\varepsilon^{1 / 2} Y_{\varepsilon} (0, 0))],\\
    \ttthreetwo{d_{\varepsilon}} & \assign & \frac{\varepsilon^{- 5 / 2}}{3}
    \int_{s, x} P_s (x) \mathbb{E} [\tilde{F}_{\varepsilon}^{ (0)}
    (\varepsilon^{1 / 2} Y_{\varepsilon} (s, x))_{} \tilde{F}_{\varepsilon}^{
    (1)} (\varepsilon^{1 / 2} Y_{\varepsilon} (0, 0))],
  \end{array} \label{eq:d-constants}
\end{equation}
where $P_s (x)$ is the heat kernel corresponding to $\LL$ and $\int_{s, x}$ denotes integration on
$\mathbb{R}^+ \times \mathbb{R}^3$.

\begin{axiom}
  \label{a:main}Suppose the following assumptions:
\begin{itemize}

    \item $(u_{0, \varepsilon})_{\varepsilon}\subset \CC^{2+\gamma}(\rho^{\frac{3}{2}+\gamma_1})\cap L^\infty(\rho^{\frac{3+6\alpha}{2m}})$ converges in law to a limit
    $u_0$ in $\CC^{1+\alpha}(\rho^{\frac{3}{2}+\gamma'})\cap \CC^{\frac{1}{2}+\alpha}(\rho^{\frac{3+6\alpha}{2m}})$ with $\rho$ a polynomial weight and is independent of $\eta$, where $\alpha, \gamma_1,\gamma'$ will be given in Section 3;

    \item $(F_{\varepsilon})_{\varepsilon} \subseteq C^9 (\mathbb{R})$ and
    there exists constants $c, C > 0$ such that
    $$
      \sup_{\varepsilon, x} \sum_{k = 0}^9 | \partial_x^k F_{\varepsilon} (x)
      | \leq C e^{c | x |},
    $$

    \item \begin{equation}(F_{\varepsilon})_{\varepsilon} =C_{0,\varepsilon}x^{m}+G_\varepsilon,\end{equation}
    with $C_{0,\varepsilon}>0, |G_\varepsilon^{(m_1)}|\leq C_1$ for some $4\leq m_1\leq m$ and $m$  odd;
    \item the family of vectors $\lambda_{\varepsilon} = (\lambda_{0,
    \varepsilon}, \lambda_{1, \varepsilon}, \lambda_{2, \varepsilon},
    \lambda_{3, \varepsilon}) \in \mathbb{R}^4$ given by
    \begin{equation}
      \begin{array}{lllllll}
        \lambda_{3, \varepsilon} & \assign & f_{3, \varepsilon} &  &
        \lambda_{1, \varepsilon} & \assign & \varepsilon^{- 1} f_{1,
        \varepsilon} - 9 \tttwotwo{d_{\varepsilon}} - 6
        \ttthreeone{d_{\varepsilon}}\\
        \lambda_{2, \varepsilon} & \assign & \varepsilon^{- 1 / 2} f_{2,
        \varepsilon} &  & \lambda_{0, \varepsilon} & \assign & \varepsilon^{-
        3 / 2} f_{0, \varepsilon} - \varepsilon^{- 1 / 2} f_{2, \varepsilon}
        \ttthreeone{d_{\varepsilon}} - 3 \ttthreetwo{d_{\varepsilon}} - 3
        \tttwotwo{\bar{d}_{\varepsilon}}
      \end{array} \label{e:def-lambda}
    \end{equation}
    has a finite limit $\lambda = (\lambda_0, \lambda_1, \lambda_2, \lambda_3)
    \in \mathbb{R}^4$ as $\varepsilon \rightarrow 0$ and $\lambda_3>0$.
    \item There exists $\delta>0$ such that for every $x,y\geq0$
    \begin{equation}\aligned&(C_{0,\varepsilon}-\delta)x+(\lambda_{3}-\delta)y
    \\\geq&2\sum_{l=4, l \textrm{ even}}^{m_1-1} |c_{l,\varepsilon}|x^{\frac{l-3}{m-3}}y^{\frac{m-l}{m-3}}+2\sum_{l=4, l \textrm{ odd}}^{m_1-1} |a_{l,\varepsilon}\wedge0|x^{\frac{l-3}{m-3}}y^{\frac{m-l}{m-3}}+\frac{ 2C_1}{m_1!}x^{\frac{m_1-3}{m-3}}y^{\frac{m-m_1}{m-3}},\endaligned\end{equation}
    with  $b_{l,\varepsilon}=\frac{C_{0,\varepsilon}m!}{(m-l)!l!}$, $c_{l,\varepsilon}=\frac{1}{l!}\mathbb{E}[G_\varepsilon^{(l)}(\varepsilon^{1/2}Y_\varepsilon)]$, $a_{l,\varepsilon}=b_{l,\varepsilon}\mathbb{E}[(\varepsilon^{1/2}Y_\varepsilon)^{m-l}]+c_{l,\varepsilon}$.
  \end{itemize}
\end{axiom}

Under Assumption 1, existence and uniqueness of classical solution to  equation (1.3) have been obtained in Proposition B.2.

\begin{theorem}
  \label{t:maintheorem}Under Assumption 1 the family of random
  fields $(u_{\varepsilon})_{\varepsilon}$ given by the solution to
  eq. converges in probability and globally in time to a limiting random
  field $u (\lambda)$ in the space $C_T \CC^{- \kappa} (\rho^{\frac{3}{2}+\gamma_1})$ for
  every $1 / 2 < \kappa < 2 / 3$ and every $T>0$. The law of $u (\lambda)$ is the same as solution to the dynamic $\Phi^4_3$
  model with parameter vector $\lambda \in \mathbb{R}^4$ obtained in [GH18].
\end{theorem}

\begin{remark} i) In [GH18] the authors only consider  the dynamic $\Phi^4_3$ model with $\lambda_i=0$ for $i=1,2,4$ and $\lambda_3=1$.  By similar arguments the results also hold for general $\lambda$ with $\lambda_3>0$.

ii) Compared to the conditions in [FG17], which are mainly used to obtain the convergence of the renormalized term,   we add conditions (1.5) and (1.7), which  might be not necessary in finite volume case.  However, in infinite volume case we have to consider in weighted Besov space. (1.5) is a usual assumption to obtain a uniform estimate and helps us to get better weight for the solution (see Lemma 2.4).
The condition (1.7) is used to  obtain the uniform  estimate for $\psi^l$ in (4.6) on the weighted space. (1.7) can be easily deduced by Young's inequality for $\lambda_3$ and $C_{0,\varepsilon}$ large enough to control the other coefficient.
\end{remark}

 The structure of this paper is as follows. In Section 2 we recall some useful results from paracontrolled distribution method and prove a modification of the maximal principle in Lemma 2.4. In Section 3 we decompose the equation by paracontrolled distribution method. In Section 4 we obtain uniform estimates for the approximation.

\section{Preliminaries}

The space of Schwartz functions on $\R^{d}$ is denoted by $\mathcal{S}(\R^{d})$ and its dual, the space of tempered distributions is $\mathcal{S}'(\R^{d})$. The Fourier transform of $u\in \mathcal{S}'(\R^{d})$ is given by
$$
\mathcal{F}u(z)=\int_{\R^{d}}u(x)e^{-i z\cdot x}\,\dd x.
$$
We use $(\Delta_{i})_{i\geq -1}$ to denote the Littlewood--Paley blocks for a dyadic partition of unity.
 Set $\langle x\rangle=(1+|x|^2)^{1/2}$. In this paper we only consider $\rho=\langle x\rangle^{-\nu}$ for $\nu>0$.  For $\alpha\in\mathbb{R}$,  define the weighted Besov space $B^\alpha_{\infty,\infty}(\rho)\backassign \CC^\alpha(\rho)$ as the collection of all $f\in\mathcal{S}'(\R^{d})$ with finite norm
\[ \| f \|_{\CC^{\alpha} (\rho)} = \sup_{i\geq -1} 2^{i \alpha} \| \Delta_i f
   \|_{L^{\infty} (\rho)} = \sup_{i\geq -1} 2^{i \alpha} \| \rho \Delta_i f
   \|_{L^{\infty}} . \]
Moreover by [Tri06, Theorem 6.9], for $\alpha\in(0,1)$ we have the weighted space $\CC^\alpha(\rho)$ given by
\begin{equation}\|f\|_{L^\infty(\rho)}+\sup_{0<|h|\leq1}h^{-\alpha}\|f(\cdot+h)-f(\cdot)\|_{L^\infty(\rho)}.\end{equation}
More details can be found  e.g. in [Tri06].

Let $\rho$ be a polynomial  weight. Then the following embedding holds
\begin{equation}\label{eq1}
\CC^{\beta_1}(\rho^{\gamma_1})\subset \CC^{\beta_2}(\rho^{\gamma_2}) \quad\text{provided}\quad \beta_1\geq\beta_2,\ \gamma_1\leqslant \gamma_2,
\end{equation}
and by [Tri06, Theorem 6.31], the embedding in (2.1) is compact provided  $ \beta_1>\beta_2$ and $\gamma_1< \gamma_2$.

 $C_T\CC^{\alpha}(\rho)$ is the space of space-time distributions $f$ that are continuous in time with values in $\CC^{\alpha}(\rho)$, and have finite norm
\[
\|f\|_{C_T\CC^{\alpha}(\rho)}:=\sup_{0\leq t\leq T} \|(\rho f)(t)\|_{\CC^{\alpha}}.
\]
 If a mapping $f:[0,T]\to  \CC^{\alpha}(\rho) $ is only bounded but not continuous, we write $f\in L_T^{\infty}\CC^{\alpha}(\rho)$ with the norm
$$
\|f\|_{L_T^{\infty}\CC^{\alpha}(\rho)}:=\mathrm{esssup}_{0\leq t\leq T} \|(\rho f)(t)\|_{\CC^{\alpha}}<\infty.
$$
For $\alpha\in(0,1)$ and $\beta\in\R$ we denote by $C_T^{\alpha}\CC^{\beta}(\rho)$  the space of mappings $f:[0,T]\to  \CC^{\beta}(\rho) $ with finite norm
$$
\|f\|_{C_T^{\alpha}\CC^{\beta}(\rho)}:=\sup_{0\leq t\leq T} \|(\rho f)(t)\|_{\CC^{\beta}}+\sup_{0\leq s,t\leq T,s\neq t} \frac{\|(\rho f)(t)-(\rho f)(s)\|_{\CC^{\beta}}}{|t-s|^{\alpha}}.
$$

 We collect some useful results from [GH18]. The form is different but we can use similar method to obtain it.

\begin{lemma}
  \label{lemma:interp2}
Let $\kappa\in (0,1)$ and let $\rho$ be a polynomial weight.  We have, for any $\alpha \in [0, 2 - \kappa]$
  \[ \| \psi \|_{C_T\CC^{\alpha} (\rho_1)} \lesssim \| \psi
     \|_{L_T^{\8}L^{\infty} (\rho_2)}^{1 - \alpha / (2 - \kappa)} \| \psi \|_{C_T\CC^{2 -
     \kappa} (\rho_3)}^{\alpha / (2 - \kappa)} ,\]
     with $\rho_1=\rho_2^{1 - \alpha / (2 - \kappa)}\rho_3^{\alpha / (2 - \kappa)}$.
Moreover, if  $\alpha/2\notin\N_{0}$ then
       \[ \| \psi \|_{C_T^{\alpha/2}L^{\8} (\rho_1)} \lesssim \| \psi
     \|_{L_T^{\8}L^{\infty} (\rho_2)}^{1 - \alpha / (2 -\kappa)} \| \psi \|_{C_T^{(2-\kappa)/2}L^{\8} (\rho_3)}^{\alpha / (2 - \kappa)} .\]
\end{lemma}

Let  $\sum_{k\geq-1} w_k = 1$ be a smooth dyadic  partition of unity on $\R^{3}$,
where $w_{-1}$ is supported in a ball containing zero and each $w_k$ for $k\geq0$ is supported on the annulus of size $2^k$.
 Let $(v_{\ell})_{\ell\geq -1}$ be a smooth dyadic partition of unity on $[0,\8)$ such that  $v_{-1}$ is supported in a ball containing zero and each $v_\ell$ for $\ell\geq0$ is supported on the annulus of size $2^\ell$.
 For a given sequence $(L_{k,\ell})_{k,\ell\geq -1}$  we define localization operators $\VV_{>},
\VV_{\leqslant}$  as in [GH18]
\begin{equation}
\VV_{>} f = \sum_{k,\ell} v_{\ell}w_k \Delta_{> L_{k,\ell}} f, \qquad
   \VV_{\leqslant} f = \sum_{k,\ell} v_{\ell}w_k \Delta_{\leqslant L_{k,\ell}} f,
   \end{equation}
where $\Delta_{> L_{k,\ell}}=\sum_{j;j>L_{k,\ell}}\Delta_j$ and $\Delta_{\leq L_{k,\ell}}=\sum_{j;j\leq L_{k,\ell}}\Delta_j$.

\begin{lemma}\label{lem:local2} ([GH18, Lemma 2.6])
Let $L>0, T>0$ be given and  let $\rho$ be a polynomial weight. There exists a (universal) choice of parameters $(L_{k,\ell})_{k,\ell\geq -1}$ such that for all $\alpha,\delta,\kappa>0$, $0\leq t\leq T$ and $a,b\in\R$ it holds true
      \[ \| \VV_{>} f \|_{C_t\CC^{- \alpha - \delta} (\rho^{- a})} \lesssim 2^{- \delta L} \| f \|_{C_t\CC^{- \alpha}
     (\rho^{-a+\delta})}, \qquad \| \VV_{\leqslant} f \|_{C_t\CC^{\kappa}
     (\rho^{b})} \lesssim 2^{(\alpha + \kappa)
     L} \| f \|_{C_t\CC^{- \alpha} (\rho^{b-\alpha-\kappa})}, \]
          where the proportional constant depends on $\alpha,\delta,\kappa,a,b$ but is independent of $f$.
\end{lemma}

\begin{lemma}\label{lemma:sch} ([GH18, Lemma 2.10])
Let $\mu>0$, $\alpha\in\R$ and let $\rho$ be a  polynomial weight. Let $v$ and $w$ solve, respectively,
$$
\LL v=f,\quad v(0)=0,\quad\qquad \LL  w=0,\quad w(0)=w_{0}.
$$
Then  it holds uniformly over $0\leq t\leq T$
\begin{equation}\label{eq:1vw}
\|v(t)\|_{\CC^{2+\alpha}(\rho)}\lesssim \|f\|_{L_T^{\infty}\CC^{\alpha}(\rho)},\qquad
\|w(t)\|_{\CC^{2+\alpha}(\rho)}\lesssim \|w_{0}\|_{\CC^{2+\alpha}(\rho)},
\end{equation}
if $0\leqslant 2+\alpha < 2$ then
$$
\|v\|_{C_T^{(2+\alpha)/{2}}L^{\infty}(\rho)}\lesssim \|f\|_{L_T^{\infty}\CC^{\alpha}(\rho)},\qquad\|v\|_{C_T^{1}L^{\infty}(\rho)}\lesssim \|f\|_{C_T\CC(\rho)},
$$

$$
 \|w\|_{C_T^{(2+\alpha)/2}L^{\infty}(\rho)}\lesssim \|w_{0}\|_{\CC^{2+\alpha}(\rho)}.
$$
\end{lemma}

\subsection{Maximum principle}

We obtain the following maximum principle modified from [GH18, Lemma 2.12].

\begin{lemma}
  \label{lemma:apriori-parabolic}
 Let $\mu\in\R$ and let $\rho$ be a  polynomial weight. Fix $\kappa > 0$ and let $\psi\in L^{\infty}L^{\infty}(\rho^3)$ smooth enough be a classical
  solution to
  \[\LL  \psi + \lambda_3\psi^3 + C_0\varepsilon^{\frac{m-3}{2}} \psi^m +a_1\varepsilon^{\frac{l-3}{2}} \psi^l= \Psi,\qquad \psi(0)=\psi_{0} ,\]
  with $m\geq l\geq 5$ odd, $C_0, a_1\geq0$.
 Then the following a priori estimate holds for small $\delta>0$
  \[ (\lambda_3-\frac{\delta}{2})\| \psi  \|_{L_T^{\infty}L^{\infty} (\rho^m)}^3+a_0\varepsilon^{\frac{m-3}{2}} \|\psi\|_{L_T^{\infty}L^{\infty}(\rho^3)}^m \leq M_\delta+ 2\|\psi_{0}\|_{L^{\infty}(\rho^3)}^m+ 2\|\Psi\|_{L_T^{\infty}L^{\infty}(\rho^{3m})},
  \]
  where the constant $M_\delta$  is independent of $\varepsilon$.
  \end{lemma}

\begin{proof}
Let $\bar\psi=\psi\rho^m, \tilde{\psi}=\psi\rho^3$.
Then we have
$$\aligned
&\rho^{2m}\partial_{t}\bar{\psi}+\rho^{2m}(-\Delta+\mu)\bar{\psi} +\lambda_3\bar{\psi}^{3}+C_0\varepsilon^{\frac{m-3}{2}} \bar{\psi}^m\rho^{-m^2+3m}+a_1\varepsilon^{\frac{l-3}{2}} \bar{\psi}^l\rho^{-ml+3m} \\=& \rho^{3m}\Psi  - \rho^{m}(\Delta \rho^m)\bar{ \psi} - 2\rho^{2m}\nabla\rho^m\nabla\psi.
\endaligned$$
Following the same argument as in the proof of [GH18, Lemma 2.12] we have
\begin{equation} \lambda_3\| \psi  \|_{L_T^{\infty}L^{\infty} (\rho^m)}^3\leq C+ \|\psi_{0}\|_{L^{\infty}(\rho^m)}^3+\| \psi  \|_{L_T^{\infty}L^{\infty} (\rho^m)}+ \|\Psi\|_{L_T^{\infty}L^{\infty}(\rho^{3m})}.
  \end{equation}

On the other hand, we assume that $\tilde\psi$ attains its (global) maximum $M=\tilde\psi(t^{*},x^{*})$ at the point $(t^{*},x^{*})$. If $M\leq0$, then it is necessary to consider  the maximum of $-\tilde\psi$, which we consider below. Let us therefore assume  that $M>0$. If $t_{*}=0$ then
$$
\tilde\psi\leqslant \|\psi_{0}\|_{L^{\infty}(\rho^3)}
$$
Assume that $0<t^{*}\leq T$.
Then
$$\aligned
&\rho^{3(m-1)}\partial_{t}\tilde{\psi}+\rho^{3(m-1)}(-\Delta+\mu)\tilde{\psi} +\lambda_3\tilde{\psi}^{3}\rho^{3(m-3)}+C_0\varepsilon^{\frac{m-3}{2}} \tilde{\psi}^m +a_1\varepsilon^{\frac{l-3}{2}} \tilde{\psi}^l\rho^{3m-3l} \\=& \rho^{3m}\Psi - \rho^{^{3(m-2)}}(\Delta \rho^3)\tilde{ \psi} - 2\rho^{3(m-1)}\nabla\rho^3\nabla\psi,
\endaligned$$
and
$$
\partial_{t}\tilde\psi(t^{*},x^{*})\geq 0,\qquad\nabla\tilde\psi(t^{*},x^{*})=0,\qquad\Delta\tilde\psi(t^{*},x^{*})\leq0
$$
hence $\rho^3\nabla\psi=-\psi\nabla\rho^3$ and $-\rho^{3(m-1)}\Delta\tilde\psi(t^{*},x^{*})\geq0$. Now we have
$$\aligned
C_0\varepsilon^{\frac{m-3}{2}} M^m\leqslant& \left[\rho^{3m}\Psi -\mu\rho^{2m}\bar\psi - \rho^{2m-3}(\Delta \rho^3) \bar{\psi} - 2\rho^{2m-6}|\nabla\rho^3|^2\bar{\psi}\right]_{|_{(t^{*},x^{*})}}
\\\leqslant& \|\Psi\|_{L_T^{\infty}L^{\infty}(\rho^{3m})} +c_{\rho,\mu}\|\bar{\psi}\|_{L_T^\infty L^\infty}.\endaligned
$$
Then we have
$$C_0\varepsilon^{\frac{m-3}{2}} \tilde{\psi}^m\leqslant \|\Psi\|_{L_T^{\infty}L^{\infty}(\rho^{3m})} +c_{\rho,\mu}\|\bar{\psi}\|_{L_T^\infty L^\infty}.$$
For $-\tilde{\psi}$ we have similar results, which implies that
\[ C_0\varepsilon^{\frac{m-3}{2}} \|\tilde{\psi}\|_{L_T^{\infty}L^{\infty}(\rho^3)}^m \leq  \|\Psi\|_{L_T^{\infty}L^{\infty}(\rho^{3m})} +c_{\rho,\mu}\|\bar{\psi}\|_{L_T^\infty L^\infty},
  \]
 Then combining the above estimate and (2.5)
and using Young's inequality the results follow.
\end{proof}

\subsection{Paracontrolled calculus}
\label{ssec:para}

Now we recall the following paraproduct introduced by Bony (see [Bon81]). In general, the product $fg$ of two distributions $f\in \mathcal{C}^\alpha, g\in \mathcal{C}^\beta$ is well defined if and only if $\alpha+\beta>0$. In terms of Littlewood-Paley blocks, the product $fg$ of two distributions $f$ and $g$ can be formally decomposed as
 $$fg=f\prec g+f\circ g+f\succ g,$$
 with $$f\prec g=g\succ f=\sum_{j\geq-1}\sum_{i<j-1}\Delta_if\Delta_jg, \quad f\circ g=\sum_{|i-j|\leq1}\Delta_if\Delta_jg.$$

We also collect the following results on paraproduct on weighted Besov space from [GH18].

\begin{lemma}\label{lem:para} [GH18, Lemma 2.14]
Let $\rho_{1},\rho_{2}$ be  weights  and $\beta\in\R$. Then  it holds
\begin{equation*}
\|f\prec g\|_{\CC^\beta(\rho_{1}\rho_{2})}\lesssim\|f\|_{L^\infty(\rho_{1})}\|g\|_{\CC^{\beta}(\rho_{2})},
\end{equation*}
and if $\alpha<0$ then
\begin{equation*}
\|f\prec g\|_{\CC^{\alpha+\beta}(\rho_{1}\rho_{2})}\lesssim\|f\|_{\CC^{\alpha}(\rho_{1})}\|g\|_{\CC^{\beta}(\rho_{2})}.
\end{equation*}
If  $\alpha+\beta>0$ then it holds
\begin{equation*}
\|f\circ g\|_{\CC^{\alpha+\beta}(\rho_{1}\rho_{2})}\lesssim\|f\|_{\CC^{\alpha}(\rho_{2})}\|g\|_{\CC^{\beta}(\rho_{2})}.
\end{equation*}
\end{lemma}

\begin{lemma}\label{lem:com}[GH18, Lemma 2.16]
Let $\rho_{1}, \rho_{2}, \rho_{3}$ be  weights and let $\alpha\in (0,1)$ and $\beta,\gamma\in \R$ such that $\alpha+\beta+\gamma>0$ and $\beta+\gamma<0$. Then there exist a trilinear bounded operator $\mathrm{com}$ satisfying
$$
\|\mathrm{com}(f,g,h)\|_{\CC^{\alpha+\beta+\gamma}(\rho_{1}\rho_{2}\rho_{3})}\lesssim \|f\|_{\CC^\alpha(\rho_{1})}\|g\|_{\CC^\beta(\rho_{2})}\|h\|_{\CC^\gamma(\rho_{3})}
$$
and for smooth functions $f,g,h$
$$
\mathrm{com}(f,g,h)=(f\prec g)\circ h - f(g\circ h).
$$
\end{lemma}
Moreover, we will make use of the time-mollified paraproducts as introduced in [GIP15, Section 5]. Let $Q:\R\to\R_{+}$ be a smooth function, supported in $[-1,1]$ and $\int_{\R}Q(s)\mathrm{d}s=1$, and for $i\geq -1$ define the operator $Q_{i}:C\CC^{\alpha}(\rho)\to C\CC^{\alpha}(\rho)$ by
$$
Q_{i}f(t)=\int_{\R}2^{2i}Q(2^{2i}(t-s))f(s\vee 0)\mathrm{d} s.
$$
Finally, we define the modified paraproduct of $f,g\in C\CC^{\alpha}(\rho)$ by
$$
f\Prec g := \sum_{i\geq -1}(S_{i-1}Q_{i}f)\Delta_{i} g,
$$
where $S_jf=\sum_{i\leq j-1}\Delta_if.$

\begin{lemma}\label{lem:5.1}[GH18, Lemma 2.18]
Let $\rho_{1},\rho_{2}$ be  polynomial  weights. Let $\alpha\in (0,1),$ $\beta\in \R$, and let $f\in C_T\CC^{\alpha}(\rho_{1})\cap C_T^{\alpha/2}L^{\infty}(\rho_{1})$ and $g\in C_T\CC^{\beta}(\rho_{2})$. Then
$$
\big\|[\LL,f\Prec] g\big\|_{C_T\CC^{\alpha+\beta-2}(\rho_{1}\rho_{2})}\lesssim \big( \|f\|_{C_T^{\alpha/2}L^{\infty}(\rho_{1})}+\|f\|_{C_T\CC^{\alpha}(\rho_{1})} \big)\|g\|_{C_T\CC^{\beta}(\rho_{2})},
$$
and
$$
\|f\prec g-f\Prec g\|_{C_T\CC^{\alpha+\beta}(\rho_{1}\rho_{2})}\lesssim \|f\|_{C_T^{\alpha/2}L^{\infty}(\rho_{1})}\|g\|_{C_T\CC^{\beta}(\rho_{2})}.
$$
\end{lemma}

\section{Paracontrolled structure}

Write (1.2) in the following form
\begin{eqnarray*}
  \LL u_{\varepsilon} & = & - \varepsilon^{- \frac{3}{2}}
  \tilde{F}_{\varepsilon} (\varepsilon^{\frac{1}{2}} u_{\varepsilon} (t, x)) +
  \eta_{\varepsilon}\\
  &  & - \varepsilon^{- 3 / 2} f_{0, \varepsilon} - \varepsilon^{- 1} f_{1,
  \varepsilon} u_{\varepsilon} - \varepsilon^{- 3 / 2} f_{2, \varepsilon} H_2
  \left( \varepsilon^{\frac{1}{2}} u_{\varepsilon}, \sigma_{\varepsilon}^2
  \right) .
\end{eqnarray*}
We write $u_{\varepsilon} = Y_{\varepsilon} + v_{\varepsilon}$, and use a Taylor expansion of
$\tilde{F}_{\varepsilon} (\varepsilon^{1 / 2} Y_{\varepsilon} + \varepsilon^{1
/ 2} v_{\varepsilon})$ around $\varepsilon^{1 / 2} Y_{\varepsilon}$ up to the
third order to have
\begin{equation}
  \begin{array}{lll}
    \LL u_{\varepsilon} & = & \eta_{\varepsilon} - \varepsilon^{- \frac{3}{2}}
    \tilde{F}_{\varepsilon} (\varepsilon^{\frac{1}{2}} Y_{\varepsilon}) -
    \varepsilon^{- 1} \tilde{F}^{(1)}_{\varepsilon} (\varepsilon^{\frac{1}{2}}
    Y_{\varepsilon}) v_{\varepsilon} - \frac{1}{2} \varepsilon^{- \frac{1}{2}}
    \tilde{F}^{(2)}_{\varepsilon} (\varepsilon^{\frac{1}{2}} Y_{\varepsilon})
    v_{\varepsilon}^2 - \frac{1}{6} \tilde{F}^{(3)}_{\varepsilon}
    (\varepsilon^{\frac{1}{2}} Y_{\varepsilon}) v_{\varepsilon}^3\\
    &  & - \varepsilon^{- 3 / 2} f_{0, \varepsilon} - \varepsilon^{- 1} f_{1,
    \varepsilon} (Y_{\varepsilon} + v_{\varepsilon}) - \varepsilon^{- 1 / 2}
    f_{2, \varepsilon} (\llbracket Y_{\varepsilon}^2 \rrbracket + 2
    v_{\varepsilon} Y_{\varepsilon} + v_{\varepsilon}^2) - R_{\varepsilon}
    (v_{\varepsilon}) .
  \end{array} \label{eq:first-exp}
\end{equation}
where $\llbracket Y_{\varepsilon}^2 \rrbracket$ is the Wick product and $R_{\varepsilon} (v_{\varepsilon})$ is the remainder of the Taylor
series and we used that $H_2 (\varepsilon^{1 / 2} Y_{\varepsilon},
\sigma_{\varepsilon}^2) = \varepsilon \llbracket Y_{\varepsilon}^2
\rrbracket$.

Define the following random fields as in [GH17]:

\begin{equation}
  \begin{array}{lllllll}
    \LL Y_{\varepsilon} & \assign &\eta_{\varepsilon} &
    \hspace{4em} &  &  & \\
    \ttwo{\bar{Y}_{\varepsilon}} & \assign & \varepsilon^{- 1 / 2} f_{2,
    \varepsilon}  \llbracket Y_{\varepsilon}^2 \rrbracket &  & \LL
    \tttwo{\bar{Y}_{\varepsilon}} & \assign &
    \ttwo{\bar{Y}_{\varepsilon}},\\
    \tthree{Y_{\varepsilon}} & \assign & \varepsilon^{- \frac{3}{2}}
    \tilde{F}_{\varepsilon} (\varepsilon^{\frac{1}{2}} Y_{\varepsilon}) &  &
    \LL \ttthree{Y_{\varepsilon}} & \assign & \tthree{Y_{\varepsilon}},\\
    \ttwo{Y_{\varepsilon}} & \assign & \frac{1}{3} \varepsilon^{- 1}
    \tilde{F}^{(1)}_{\varepsilon} (\varepsilon^{\frac{1}{2}} Y_{\varepsilon})
    &  & \LL \tttwo{Y_{\varepsilon}} & \assign & \ttwo{Y_{\varepsilon}}\\
    \tone{Y_{\varepsilon}} & \assign & \frac{1}{6} \varepsilon^{- \frac{1}{2}}
    \tilde{F}^{(2)}_{\varepsilon} (\varepsilon^{\frac{1}{2}} Y_{\varepsilon})
    &  & {Y_{\varepsilon}^\emptyset} & \assign & \frac{1}{6}
    \tilde{F}^{(3)}_{\varepsilon} (\varepsilon^{\frac{1}{2}}
    Y_{\varepsilon})\\
    \tttwotwo{\bar{Y}_{\varepsilon}} & \assign &
    \tttwo{\bar{Y}_{\varepsilon}} \circ \ttwo{Y_{\varepsilon}} -
    \tttwotwo{\bar{d}_{\varepsilon}} &  &  \tttwotwo{Y_{\varepsilon}} &
    \assign & \tttwo{Y_{\varepsilon}} \circ \ttwo{Y_{\varepsilon}} -
     \tttwotwo{d_{\varepsilon}},\\
    \ttthreeone{Y_{\varepsilon}} & \assign & \ttthree{Y_{\varepsilon}} \circ
    \tone{Y_{\varepsilon}} - \ttthreeone{d_{\varepsilon}}, &  &
    \ttthreetwo{Y_{\varepsilon}} & \assign & \ttthree{Y_{\varepsilon}}
    \circ \ttwo{Y_{\varepsilon}} - \ttthreetwo{d_{\varepsilon}'}
    Y_{\varepsilon} - \ttthreetwo{d_{\varepsilon}},
  \end{array} \label{e:trees-def}
\end{equation}
with $Y_{\varepsilon}$ stationary solution, while $Y_{\varepsilon},
\ttthree{Y_{\varepsilon}}, \tttwo{Y_{\varepsilon}},
\tttwo{\bar{Y}_{\varepsilon}}$ have $0$ initial condition in $t = 0$.
$\tttwotwo{\bar{d}_{\varepsilon}}$, $\tttwotwo{d_{\varepsilon}}$,
$\ttthreeone{d_{\varepsilon}}$, $\ttthreetwo{d_{\varepsilon}'}$,
$\ttthreetwo{d_{\varepsilon}}$ are  $\varepsilon$-dependent constants
introduced in introduction and
\begin{equation}
  \begin{array}{lll}
    \ttthreetwo{d_{\varepsilon}'} & = & 2 \ttthreeone{d_{\varepsilon}} + 3
    \tttwotwo{d_{\varepsilon}} .
  \end{array} \label{e:ren-constraint}
\end{equation}

The enhanced noise $\mathbb{Y}_\varepsilon$ is constructed
 as
\begin{equation}
  \begin{array}{lll}
    \mathbb{Y}_\varepsilon
    & \assign & (Y_\varepsilon^\emptyset, \tone{Y_\varepsilon},  \ttwo{Y_\varepsilon},
    \ttwo{\bar{Y}_\varepsilon},  \ttthree{Y_\varepsilon},  \ttthreeone{Y_\varepsilon},
     \tttwotwo{Y_\varepsilon},  \tttwotwo{\bar{Y}_\varepsilon},
     \tttwotwo{Y_\varepsilon})
  \end{array} \label{e:def-ylambda}
\end{equation}

For $\kappa>0$ the homogeneities $\alpha_\tau \in \mathbb{R}$ are
given by
\[ \begin{array}{|c|c|c|c|c|c|c|c|c|c|c|}
     \hline
     Y_{\varepsilon}^{\tau} & = & {Y_{\varepsilon}^\emptyset} &
     \tone{Y_{\varepsilon}} & \ttwo{Y_{\varepsilon}} &
     \ttwo{\bar{Y}_{\varepsilon}} & \ttthree{Y_{\varepsilon}} &
     \ttthreeone{Y_{\varepsilon}} & \tttwotwo{Y_{\varepsilon}} &
     \tttwotwo{\bar{Y}_{\varepsilon}} & \ttthreetwo{Y_{\varepsilon}}\\
     \hline
    \alpha^\tau & = & -\kappa & - \frac{1}{2}-\kappa & - 1-\kappa & - 1-\kappa & \frac{1}{2}-\kappa & -\kappa & -\kappa & -\kappa & - \frac{1}{2}-\kappa\\
     \hline
   \end{array} \]

In order to identify interesting limits for equation, we
introduce $\forall \lambda = (\lambda_0, \lambda_1, \lambda_2, \lambda_3) \in
\mathbb{R}^4$ the enhanced noise $\mathbb{X} (\lambda)$ which is constructed
\begin{equation}
  \begin{array}{lll}
    \mathbb{X} (\lambda)
    & \assign & (\lambda_3, \lambda_3 X, \lambda_3 \ttwo{X}, \lambda_2
    \ttwo{X}, \lambda_3 \ttthree{X}, (\lambda_3)^2 \ttthreeone{X},
    (\lambda_3)^2 \tttwotwo{X}, \lambda_3 \lambda_2 \tttwotwo{X},
    (\lambda_3)^2 \tttwotwo{X})
  \end{array} \label{e:def-ylambda}
\end{equation}
where $X$ is the stationary solution to to the linear equation $\LL X =
\xi$ and $\xi$ is the time-space white noise on $\mathbb{R} \times
\mathbb{R}^3$ and other terms can be defined as in [GH18], i.e.
$$ \LL \ttthree{X} = \llbracket X^3 \rrbracket,\quad\ttthree{X}(0)=0,   \quad \LL
   \tttwo{X} = \llbracket X^2 \rrbracket, \quad \tttwo{X}(0)=0,$$
   $$ \LL \ttthree{X_\varepsilon} = \llbracket Y_\varepsilon^3 \rrbracket,\quad\ttthree{Y_\varepsilon}(0)=0,   \quad \LL
   \tttwo{X_\varepsilon} = \llbracket X_\varepsilon^2 \rrbracket, \quad \tttwo{X_\varepsilon}(0)=0,$$
      $$\ttthreeone{X}=\lim_{\varepsilon\to 0}\ttthree{X_{\varepsilon}}\circ Y_{\varepsilon},\quad
      \tttwotwo{X} = \lim_{\varepsilon\to 0}\tttwo{X_{\varepsilon}} \circ \llbracket Y_{\varepsilon}^2 \rrbracket -\frac{b_{\varepsilon}}{3}, \qquad
   \ttthreetwo{X} = \lim_{\varepsilon\to 0}\ttthree{X_{\varepsilon}} \circ \llbracket X_{\varepsilon}^2 \rrbracket -  b_{\varepsilon}X_{\varepsilon}, $$
   where $b_{\varepsilon}(t) = 3 \mathbb{E}[(\tttwo{X_{\varepsilon}} \circ \llbracket X_{\varepsilon}^2 \rrbracket )(t,0)]$ stands for a suitable renormalization constant.

 \begin{theorem}\label{thm:renorm43}
 Let $\rho(x)=\langle x\rangle^{-\nu}$ for some $\nu>0$. For every $\kappa,\sigma>0$ and some $\delta,\gamma>0$. Moreover, if $\tau_{\varepsilon}$ is an component in $\mathbb{Y}_\varepsilon$ with  $\tau$ be the corresponding component in $\mathbb{X}$  then
$\tau_{\varepsilon }\rightarrow \tau$ in $ C_T\CC^{\alpha_{\tau}}(\rho^{\sigma})\cap C_T^{\delta/2}\CC^{\alpha_{\tau}-\gamma}(\rho^{\sigma})$ a.s. as $\varepsilon\to0$.

Moreover, for every $0<\kappa\leq \frac{1}{2}-\frac{\epsilon}{2}, \epsilon>0$ and $G$ as in the introduction we have
$$\mathbb{P}[\|Y_\varepsilon^\emptyset-\lambda_3\|_{C_T\CC^{-\kappa-\epsilon}(\rho^\sigma)}\lesssim \varepsilon^\kappa ]\geq 1-C\varepsilon^{\epsilon/2}.$$
$$\mathbb{P}[\|G^{(l)}(\varepsilon^{1/2}Y_\varepsilon)-\mathbb{E}[G^{(l)}(\varepsilon^{1/2}Y_\varepsilon)]\|_{C_T\CC^{-\kappa-\epsilon}(\rho^\sigma)}\lesssim \varepsilon^\kappa, l=1,...,m_1-1 ]\geq 1-C\varepsilon^{\epsilon/2}.$$
$$\mathbb{P}[\|(\varepsilon^{1/2}Y_\varepsilon)^k-\mathbb{E}[(\varepsilon^{1/2}Y_\varepsilon)^k]\|_{C_T\CC^{-\kappa-\epsilon}(\rho^\sigma)}\lesssim \varepsilon^\kappa, k=1,...,m-1 ]\geq 1-C\varepsilon^{\epsilon/2}.$$
\end{theorem}
\begin{proof}
The convergence of the above renormalization terms  has been given in finite volume case in [FG17].  It can also be extended to infinite volume by similar arguments as in [GH18]. For the second part we have the following estimates by similar calculation as in [FG17]:
$$\mathbb{E}[\|Y_\varepsilon^\emptyset-\lambda_3\|_{C_T\CC^{-\kappa-\epsilon}(\rho^\sigma)}]\lesssim \varepsilon^{\kappa+\frac{\epsilon}{2}}.$$
$$\mathbb{P}[\|G^{(l)}(\varepsilon^{1/2}Y_\varepsilon)-\mathbb{E}[G^{(l)}(\varepsilon^{1/2}Y_\varepsilon)]\|_{C_T\CC^{-\kappa-\epsilon}(\rho^\sigma)}]\lesssim \varepsilon^{\kappa+\frac{\epsilon}{2}}.$$
$$\mathbb{P}[\|(\varepsilon^{1/2}Y_\varepsilon)^k-\mathbb{E}[(\varepsilon^{1/2}Y_\varepsilon)]\|_{C_T\CC^{-\kappa-\epsilon}(\rho^\sigma)}]\lesssim \varepsilon^{\kappa+\frac{\epsilon}{2}}.$$
Then by Chebyshev's inequality the second results follow.
\end{proof}
$\hfill\Box$

\vskip.10in

We  rewrite (3.1) as follows:
\begin{equation}\aligned
\LL u_\varepsilon =&\eta_\varepsilon-\ttwo{\bar{Y}_{\varepsilon}}-\tthree{Y_{\varepsilon}}-3\ttwo{Y_{\varepsilon}}v_\varepsilon-3\tone{Y_{\varepsilon}}v_\varepsilon^2-Y_\varepsilon^\emptyset v_\varepsilon^3\\&-\varepsilon^{-3/2}f_{0,\varepsilon}-\varepsilon^{-1}f_{1,\varepsilon}(Y_\varepsilon+v_\varepsilon)-\varepsilon^{-1/2}f_{2,\varepsilon}
(2Y_\varepsilon v_\varepsilon+v_\varepsilon^2)-R_\varepsilon(v_\varepsilon).
\endaligned\end{equation}
$u_\varepsilon=Y_\varepsilon+v_\varepsilon$,
with
$$ v_\varepsilon =  - \tttwo{\bar{Y}_\varepsilon}-\ttthree{Y_\varepsilon}  + \phi_\varepsilon + \psi_\varepsilon.$$
Then we have
\begin{equation}\aligned
&\LL \phi_\varepsilon + \LL\psi_\varepsilon \\=&-3\ttwo{Y_\varepsilon}(- \tttwo{\bar{Y}_\varepsilon}-\ttthree{Y_\varepsilon}  + \phi_\varepsilon + \psi_\varepsilon)-3\tone{Y_\varepsilon}(- \tttwo{\bar{Y}_\varepsilon}-\ttthree{Y_\varepsilon}  + \phi_\varepsilon + \psi_\varepsilon)^2-Y_\varepsilon^\emptyset (- \tttwo{\bar{Y}_\varepsilon}-\ttthree{Y_\varepsilon}  + \phi_\varepsilon + \psi_\varepsilon)^3\\&-\varepsilon^{-3/2}f_{0,\varepsilon}-\varepsilon^{-1}f_{1,\varepsilon}(Y_\varepsilon- \tttwo{\bar{Y}_\varepsilon}-\ttthree{Y_\varepsilon}  + \phi_\varepsilon + \psi_\varepsilon)-\varepsilon^{-1/2}f_{2,\varepsilon}
[2Y_\varepsilon(- \tttwo{\bar{Y}_\varepsilon}-\ttthree{Y_\varepsilon}  + \phi_\varepsilon + \psi_\varepsilon)\\&+(- \tttwo{\bar{Y}_\varepsilon}-\ttthree{Y_\varepsilon}  + \phi_\varepsilon + \psi_\varepsilon)^2]-R_\varepsilon(- \tttwo{\bar{Y}_\varepsilon}-\ttthree{Y_\varepsilon}  + \phi_\varepsilon + \psi_\varepsilon).
\endaligned\end{equation}

Suppose that $\phi_\varepsilon$ is paracontrolled by $\ttwo{Y}$, namely it holds
\begin{equation}\phi_\varepsilon=\vartheta_\varepsilon-3(- \tttwo{\bar{Y}_\varepsilon}-\ttthree{Y_\varepsilon}  + \phi_\varepsilon + \psi_\varepsilon)\Prec \tttwo{Y_\varepsilon}.\end{equation}

Now we have
\begin{equation}\label{eq2.4}
\begin{aligned}
0&= \LL\vartheta_\varepsilon+\LL\psi_\varepsilon -3\Big( (- \tttwo{\bar{Y}_\varepsilon}-\ttthree{Y_\varepsilon}  + \phi_\varepsilon + \psi_\varepsilon)\Prec\ttwo{Y_\varepsilon}- (- \tttwo{\bar{Y}_\varepsilon}-\ttthree{Y_\varepsilon}  + \phi_\varepsilon + \psi_\varepsilon)\prec\ttwo{Y_\varepsilon}\Big)\\
&\quad+3\ttwo{Y_\varepsilon}\preccurlyeq(- \tttwo{\bar{Y}_\varepsilon}-\ttthree{Y_\varepsilon}+\phi_\varepsilon+\psi_\varepsilon)-3[\LL,(- \tttwo{\bar{Y}_\varepsilon}-\ttthree{Y_\varepsilon}+\phi_\varepsilon+\psi_\varepsilon)\Prec]\tttwo{Y_\varepsilon}\\
&\quad+3\tone{Y_\varepsilon}(- \tttwo{\bar{Y}_\varepsilon}-\ttthree{Y_\varepsilon}  + \phi_\varepsilon + \psi_\varepsilon)^2
+Y_\varepsilon^\emptyset(- \tttwo{\bar{Y}_\varepsilon}-\ttthree{Y_\varepsilon}  + \phi_\varepsilon + \psi_\varepsilon)^3\\&\quad+\varepsilon^{-3/2}f_{0,\varepsilon}+\varepsilon^{-1}f_{1,\varepsilon}(Y_\varepsilon- \tttwo{\bar{Y}_\varepsilon}-\ttthree{Y_\varepsilon}  + \phi_\varepsilon + \psi_\varepsilon)+\varepsilon^{-1/2}f_{2,\varepsilon}
[2Y_\varepsilon(- \tttwo{\bar{Y}_\varepsilon}-\ttthree{Y_\varepsilon}  + \phi_\varepsilon + \psi_\varepsilon)\\&\quad+(- \tttwo{\bar{Y}_\varepsilon}-\ttthree{Y_\varepsilon}  + \phi_\varepsilon + \psi_\varepsilon)^2]+R_\varepsilon(- \tttwo{\bar{Y}_\varepsilon}-\ttthree{Y_\varepsilon}  + \phi_\varepsilon + \psi_\varepsilon).
\end{aligned}
\end{equation}

Compared to the corresponding terms in [GH18] we mainly need to handle the extra terms containing $Y_\varepsilon^\emptyset$ and $R_\varepsilon$.
For other  terms we have the following similar  decomposition as in [GH18] by using localization operator (see Lemma 2.2):
$$
3\ttwo{Y_\varepsilon}\succ(- \tttwo{\bar{Y}_\varepsilon}-\ttthree{Y_\varepsilon}  + \phi_\varepsilon + \psi_\varepsilon)=-\rmm{3\ttwo{Y_\varepsilon}\succ(\tttwo{\bar{Y}_\varepsilon}+\ttthree{Y_\varepsilon})}
+\rmm{3\VV_>\ttwo{Y_\varepsilon}\succ(\phi_\varepsilon+\psi_\varepsilon)}+\rmb{3\VV_{\leq}\ttwo{Y_\varepsilon}\succ(\phi_\varepsilon+\psi_\varepsilon)},
$$
$$\aligned
&3\ttwo{Y_\varepsilon}\preccurlyeq(- \tttwo{\bar{Y}_\varepsilon}-\ttthree{Y_\varepsilon}  + \phi_\varepsilon + \psi_\varepsilon)\\=&-\rmm{3\ttwo{Y_\varepsilon}\prec(\tttwo{\bar{Y}_\varepsilon}+\ttthree{Y_\varepsilon})}
+\rmm{3\VV_>\ttwo{Y_\varepsilon}\prec(\phi_\varepsilon+\psi_\varepsilon)}+\rmb{3\VV_{\leq}\ttwo{Y_\varepsilon}\prec(\phi_\varepsilon+\psi_\varepsilon)}
\\&+3\ttwo{Y_\varepsilon}\circ(-\tttwo{\bar{Y_\varepsilon}}-\ttthree{Y_\varepsilon}+\phi_\varepsilon)+\rmb{3\ttwo{Y_\varepsilon}\circ\psi_\varepsilon}.
\endaligned$$
$$\aligned&
3\ttwo{Y_\varepsilon}\circ(-\tttwo{\bar{Y_\varepsilon}}-\ttthree{Y_\varepsilon}+\phi_\varepsilon)
\\=&\rmm{-3\ttthreetwo{Y_\varepsilon}-3\tttwotwo{\bar{Y}_\varepsilon}}-3\tttwotwo{\bar{d}_\varepsilon}
-3\ttthreetwo{d_\varepsilon}-3\ttthreetwo{d'_\varepsilon}Y_\varepsilon+\rmb{3\ttwo{Y_\varepsilon}\circ\vartheta_\varepsilon}
\\&-\rmb{[9\ttwo{Y_\varepsilon}\circ((- \tttwo{\bar{Y}_\varepsilon}-\ttthree{Y_\varepsilon}  + \phi_\varepsilon + \psi_\varepsilon)\Prec \tttwo{Y_\varepsilon})-9\ttwo{Y_\varepsilon}\circ((- \tttwo{\bar{Y}_\varepsilon}-\ttthree{Y_\varepsilon}  + \phi_\varepsilon + \psi_\varepsilon)\prec \tttwo{Y_\varepsilon})]}
\\&\rmb{-9\mathrm{com}(- \tttwo{\bar{Y}_\varepsilon}-\ttthree{Y_\varepsilon}  + \phi_\varepsilon + \psi_\varepsilon,\tttwo{Y_\varepsilon},\ttwo{Y_\varepsilon})}\\&-\rmg{9(-\tttwo{\bar{Y}_\varepsilon}-\ttthree{Y_\varepsilon}+\phi_\varepsilon
+\psi_\varepsilon)\tttwotwo{{Y}_\varepsilon}}-9\tttwotwo{d_\varepsilon}v_\varepsilon.\endaligned$$

$$\aligned &3\tone{Y_\varepsilon}(- \tttwo{\bar{Y}_\varepsilon}-\ttthree{Y_\varepsilon}  + \phi_\varepsilon + \psi_\varepsilon)^2
\\=&\rmm{3\tone{Y_\varepsilon}\succ(\ttthree{Y_\varepsilon})^2}+\rmm{3\tone{Y_\varepsilon}\prec(\ttthree{Y_\varepsilon})^2}+\rmm{6\ttthree{Y_\varepsilon}
\ttthreeone{Y_\varepsilon}}+\rmb{6\mathrm{com}(\ttthree{Y_\varepsilon},\ttthree{Y_\varepsilon},\tone{Y_\varepsilon})}
+\rmm{3\tone{Y_\varepsilon}\circ R_1(\ttthree{Y_\varepsilon})}
\\&-\rmg{6\tone{Y_\varepsilon}\succ(\ttthree{Y_\varepsilon}(-\tttwo{\bar{Y_\varepsilon}}+\phi_\varepsilon+\psi_\varepsilon))}-
\rmg{6\tone{Y_\varepsilon}\prec(\ttthree{Y_\varepsilon}(-\tttwo{\bar{Y_\varepsilon}}+\phi_\varepsilon+\psi_\varepsilon))}
\\&-\rmb{6\tone{Y_\varepsilon}\circ(\ttthree{Y_\varepsilon}
\preccurlyeq(-\tttwo{\bar{Y_\varepsilon}}+\phi_\varepsilon+\psi_\varepsilon))}-\rmg{6(-\tttwo{\bar{Y_\varepsilon}}+\phi_\varepsilon+\psi_\varepsilon)
\ttthreeone{Y_\varepsilon}}
\\
&-\rmb{6\mathrm{com}(-\tttwo{\bar{Y_\varepsilon}}+\phi_\varepsilon+\psi_\varepsilon,\ttthree{Y_\varepsilon},\tone{Y_\varepsilon})}
+\rmg{3\tone{Y_\varepsilon}(-\tttwo{\bar{Y_\varepsilon}}
+\phi_\varepsilon+\psi_\varepsilon)^2}-6\ttthreeone{d_\varepsilon}v_\varepsilon,\endaligned$$
where $R_1(\ttthree{Y_\varepsilon})=\ttthree{Y_\varepsilon}^2-2\ttthree{Y_\varepsilon}\prec\ttthree{Y_\varepsilon}$.
For the above brown terms we can decompose it similarly as in [GH18] by using localization operator $\VV_>$ and $\VV_\leq$ with $\VV_>+\VV_\leq=1$:
$$\aligned\rmg{-9(-\tttwo{\bar{Y}_\varepsilon}-\ttthree{Y_\varepsilon}+\phi_\varepsilon
+\psi_\varepsilon)\tttwotwo{{Y}_\varepsilon}}=&\rmm{9(\tttwo{\bar{Y}_\varepsilon}+\ttthree{Y_\varepsilon})\tttwotwo{{Y}_\varepsilon}}
-\rmm{9(\phi_\varepsilon+\psi_\varepsilon)\prec\VV_>\tttwotwo{{Y}_\varepsilon}}
\\&-\rmb{9(\phi_\varepsilon+\psi_\varepsilon)\prec\VV_\leq\tttwotwo{{Y}_\varepsilon}}-\rmb{9(\phi_\varepsilon+\psi_\varepsilon)
\succcurlyeq\tttwotwo{{Y}_\varepsilon}}.\endaligned$$
$$\aligned\rmg{6\tone{Y_\varepsilon}\succ(\ttthree{Y_\varepsilon}(-\tttwo{\bar{Y_\varepsilon}}+\phi_\varepsilon+\psi_\varepsilon))}
=&-\rmm{6\tone{Y_\varepsilon}\succ(\ttthree{Y_\varepsilon}\tttwo{\bar{Y_\varepsilon}})}+\rmm{6\VV_>\tone{Y_\varepsilon}\succ(\ttthree{Y_\varepsilon}(\phi_\varepsilon+\psi_\varepsilon))}
\\&+\rmb{6\VV_\leq\tone{Y_\varepsilon}\succ(\ttthree{Y_\varepsilon}(\phi_\varepsilon+\psi_\varepsilon))}.\endaligned$$
$$\aligned\rmg{6\tone{Y_\varepsilon}\prec(\ttthree{Y_\varepsilon}(-\tttwo{\bar{Y_\varepsilon}}+\phi_\varepsilon+\psi_\varepsilon))}
=&-\rmm{6\tone{Y_\varepsilon}\prec(\ttthree{Y_\varepsilon}\tttwo{\bar{Y_\varepsilon}})}+\rmm{6\VV_>\tone{Y_\varepsilon}
\prec(\ttthree{Y_\varepsilon}(\phi_\varepsilon+\psi_\varepsilon))}
\\&+\rmb{6\VV_\leq\tone{Y_\varepsilon}\prec(\ttthree{Y_\varepsilon}(\phi_\varepsilon+\psi_\varepsilon))}.\endaligned$$
$$\aligned\rmg{6(-\tttwo{\bar{Y_\varepsilon}}+\phi_\varepsilon+\psi_\varepsilon)
\ttthreeone{Y_\varepsilon}}=&\rmm{-6\tttwo{\bar{Y_\varepsilon}}
\ttthreeone{Y_\varepsilon}}+\rmm{6(\phi_\varepsilon+\psi_\varepsilon)
\prec\VV_>\ttthreeone{Y_\varepsilon}}\\&+\rmb{6(\phi_\varepsilon+\psi_\varepsilon)
\prec\VV_\leq\ttthreeone{Y_\varepsilon}}+\rmb{6(\phi_\varepsilon+\psi_\varepsilon)
\succcurlyeq\ttthreeone{Y_\varepsilon}}\endaligned$$
$$\aligned\rmg{3\tone{Y_\varepsilon}(-\tttwo{\bar{Y_\varepsilon}}
+\phi_\varepsilon+\psi_\varepsilon)^2}=&\rmm{3\VV_>\tone{Y_\varepsilon}\succ(-\tttwo{\bar{Y_\varepsilon}}
+\phi_\varepsilon+\psi_\varepsilon)^2}+\rmb{3\VV_\leq\tone{Y_\varepsilon}\succ(-\tttwo{\bar{Y_\varepsilon}}
+\phi_\varepsilon+\psi_\varepsilon)^2}\\&+\rmb{3\tone{Y_\varepsilon}\preccurlyeq(-\tttwo{\bar{Y_\varepsilon}}
+\phi_\varepsilon+\psi_\varepsilon)^2}.\endaligned$$
For the term containing $Y^\emptyset_\varepsilon$ we decompose
$$\aligned &Y^\emptyset_\varepsilon(- \tttwo{\bar{Y}_\varepsilon}-\ttthree{Y_\varepsilon}  + \phi_\varepsilon + \psi_\varepsilon)^3\\=&(Y_\varepsilon^\emptyset-\lambda_3)(- \tttwo{\bar{Y}_\varepsilon}-\ttthree{Y_\varepsilon}  + \phi_\varepsilon + \psi_\varepsilon)^3+\lambda_3(- \tttwo{\bar{Y}_\varepsilon}-\ttthree{Y_\varepsilon}  + \phi_\varepsilon + \psi_\varepsilon)^3
\\=&\rmm{\VV_>(Y^\emptyset_\varepsilon-\lambda_3)\succ(- \tttwo{\bar{Y}_\varepsilon}-\ttthree{Y_\varepsilon}  + \phi_\varepsilon + \psi_\varepsilon)^3}+\rmb{(Y_\varepsilon^\emptyset-\lambda_3)\preccurlyeq(- \tttwo{\bar{Y}_\varepsilon}-\ttthree{Y_\varepsilon}  + \phi_\varepsilon + \psi_\varepsilon)^3}\\&+\rmb{\VV_\leq(Y^\emptyset_\varepsilon-\lambda_3)\succ(- \tttwo{\bar{Y}_\varepsilon}-\ttthree{Y_\varepsilon}  + \phi_\varepsilon + \psi_\varepsilon)^3}
\\&+\rmb{\lambda_3(- \tttwo{\bar{Y}_\varepsilon}-\ttthree{Y_\varepsilon}  + \phi_\varepsilon)^3}+\rmb{3\lambda_3(- \tttwo{\bar{Y}_\varepsilon}-\ttthree{Y_\varepsilon}  + \phi_\varepsilon)^2\psi_\varepsilon}+\rmb{3\lambda_3(- \tttwo{\bar{Y}_\varepsilon}-\ttthree{Y_\varepsilon}  + \phi_\varepsilon)\psi_\varepsilon^2}+\lambda_3\psi_\varepsilon^3.\endaligned$$
Combining the above terms containing $d_\varepsilon$ and  the other terms without $R_\varepsilon$ in (3.9),  we have the following decomposition
\begin{equation}
\begin{aligned}
&\rmm{-\lambda_{1,\varepsilon}Y_\varepsilon}-\rmb{\lambda_{0,\varepsilon}}-\rmb{\lambda_{1,\varepsilon}(- \tttwo{\bar{Y}_\varepsilon}-\ttthree{Y_\varepsilon}  + \phi_\varepsilon + \psi_\varepsilon)}+\rmm{2\lambda_{2,\varepsilon}(\ttthreeone{Y_\varepsilon}+\ttthree{Y_\varepsilon}\prec \tone{Y_\varepsilon}+\ttthree{Y_\varepsilon}\succ \tone{Y_\varepsilon})}\\&+\rmm{2\lambda_{2,\varepsilon}
Y_\varepsilon \tttwo{\bar{Y}_\varepsilon}}-\rmm{2\lambda_{2,\varepsilon} \VV_>Y_\varepsilon\succ(\phi_\varepsilon + \psi_\varepsilon)}-\rmb{2\lambda_{2,\varepsilon} \VV_\leq Y_\varepsilon\succ(\phi_\varepsilon + \psi_\varepsilon)}-\rmb{2\lambda_{2,\varepsilon} Y_\varepsilon\preccurlyeq(\phi_\varepsilon+\psi_\varepsilon)}\\&-\rmb{\lambda_{2,\varepsilon}(- \tttwo{\bar{Y}_\varepsilon}-\ttthree{Y_\varepsilon}  + \phi_\varepsilon+ \psi_\varepsilon)^2}.
\end{aligned}
\end{equation}
Now we come to $R_\varepsilon$. By assumption we have that
$$F_\varepsilon^{(3)}(x)=C_{0,\varepsilon}(\Pi_{i=0}^2(m-i))x^{m-3}+G_\varepsilon^{(3)}(x).$$
By this we separate $R_\varepsilon$ as $R_\varepsilon^1+R_\varepsilon^2$.
For $R_\varepsilon^1(- \tttwo{\bar{Y}_\varepsilon}-\ttthree{Y_\varepsilon}  + \phi_\varepsilon + \psi_\varepsilon)$ we have
$$\aligned &R_\varepsilon^1(- \tttwo{\bar{Y}_\varepsilon}-\ttthree{Y_\varepsilon}  + \phi_\varepsilon + \psi_\varepsilon)
\\=&a_{0,\varepsilon}\varepsilon^{\frac{1}{2}} (- \tttwo{\bar{Y}_\varepsilon}-\ttthree{Y_\varepsilon}  + \phi_\varepsilon + \psi_\varepsilon)^4\int_0^1(\varepsilon^{\frac{1}{2}}Y_\varepsilon+\tau\varepsilon^{\frac{1}{2}} (- \tttwo{\bar{Y}_\varepsilon}-\ttthree{Y_\varepsilon}  + \phi_\varepsilon + \psi_\varepsilon))^{m-4}\frac{(1-\tau)^3}{3!}d\tau
\\=&\sum_{l=4}^{m-1}\varepsilon^{\frac{l-3}{2}}b_{l,\varepsilon}(\varepsilon^{\frac{1}{2}}Y_\varepsilon)^{m-l} (- \tttwo{\bar{Y}_\varepsilon}-\ttthree{Y_\varepsilon}  + \phi_\varepsilon + \psi_\varepsilon)^{l}+C_{0,\varepsilon}\varepsilon^{\frac{m-3}{2}}(- \tttwo{\bar{Y}_\varepsilon}-\ttthree{Y_\varepsilon}  + \phi_\varepsilon + \psi_\varepsilon)^{m}
,\endaligned$$
with $a_{0,\varepsilon}=C_{0,\varepsilon}(\Pi_{i=0}^3(m-i))$,  $b_{l,\varepsilon}=C_{0,\varepsilon}\frac{m!}{(m-l)!l!}$.
Furthermore,  we have the following decomposition for the above two terms:
$$\aligned &C_{0,\varepsilon}\varepsilon^{\frac{m-3}{2}}(- \tttwo{\bar{Y}_\varepsilon}-\ttthree{Y_\varepsilon}  + \phi_\varepsilon + \psi_\varepsilon)^{m}=
\rmb{C_{0,\varepsilon}\sum_{k=1}^m \varepsilon^{\frac{m-3}{2}}C_m^k(- \tttwo{\bar{Y}_\varepsilon}-\ttthree{Y_\varepsilon}  + \phi_\varepsilon)^k\psi_\varepsilon^{m-k}} +C_{0,\varepsilon}\varepsilon^{\frac{m-3}{2}}\psi_\varepsilon^m.\endaligned$$

\begin{equation}\aligned &\sum_{l=4}^{m-1}\varepsilon^{\frac{l-3}{2}}b_{l,\varepsilon} (\varepsilon^{\frac{1}{2}}Y_\varepsilon)^{m-l} (- \tttwo{\bar{Y}_\varepsilon}-\ttthree{Y_\varepsilon}  + \phi_\varepsilon + \psi_\varepsilon)^{l}\\=&\sum_{l=4,l \textrm{ odd}}^{m-1}\varepsilon^{\frac{l-3}{2}}b_{l,\varepsilon} \mathbb{E}[(\varepsilon^{\frac{1}{2}}Y_\varepsilon)^{m-l}] (- \tttwo{\bar{Y}_\varepsilon}-\ttthree{Y_\varepsilon}  + \phi_\varepsilon + \psi_\varepsilon)^{l}\\&+\sum_{l=4}^{m-1} b_{l,\varepsilon}
\bigg[\rmb{Y^{\emptyset,m-l}_{1,\varepsilon}\preccurlyeq\varepsilon^{\frac{l-3}{2}} (- \tttwo{\bar{Y}_\varepsilon}-\ttthree{Y_\varepsilon}  + \phi_\varepsilon + \psi_\varepsilon)^{l}}\\&+\rmb{\VV_\leq Y^{\emptyset,m-l}_{1,\varepsilon}\succ\varepsilon^{\frac{l-3}{2}} (- \tttwo{\bar{Y}_\varepsilon}-\ttthree{Y_\varepsilon}  + \phi_\varepsilon + \psi_\varepsilon)^{l}}\\&+\rmm{\VV_>Y^{\emptyset,m-l}_{1,\varepsilon}\succ\varepsilon^{\frac{l-3}{2}} (- \tttwo{\bar{Y}_\varepsilon}-\ttthree{Y_\varepsilon}  + \phi_\varepsilon + \psi_\varepsilon)^{l}}\bigg],\endaligned\end{equation}
with $Y^{\emptyset,m-l}_{1,\varepsilon}= (\varepsilon^{\frac{1}{2}}Y_\varepsilon)^{m-l}-\mathbb{E}[(\varepsilon^{\frac{1}{2}}Y_\varepsilon)^{m-l}]$. Here we used $\mathbb{E}[(\varepsilon^{\frac{1}{2}}Y_\varepsilon)^{m-l}]=0$ for $l$ even.

\begin{equation}\aligned &R_\varepsilon^2(- \tttwo{\bar{Y}_\varepsilon}-\ttthree{Y_\varepsilon}  + \phi_\varepsilon + \psi_\varepsilon)
\\=&(- \tttwo{\bar{Y}_\varepsilon}-\ttthree{Y_\varepsilon}  + \phi_\varepsilon + \psi_\varepsilon)^3\int_0^1[G_\varepsilon^{(3)}(\varepsilon^{\frac{1}{2}}Y_\varepsilon+\tau\varepsilon^{\frac{1}{2}} (- \tttwo{\bar{Y}_\varepsilon}-\ttthree{Y_\varepsilon}  + \phi_\varepsilon + \psi_\varepsilon))-G_\varepsilon^{(3)}(\varepsilon^{\frac{1}{2}}Y_\varepsilon)]\frac{(1-\tau)^2}{2!}d\tau
\\=&(- \tttwo{\bar{Y}_\varepsilon}-\ttthree{Y_\varepsilon}  + \phi_\varepsilon + \psi_\varepsilon)^3\int_0^1\sum_{k=1}^{m_1-4}\bigg[G_\varepsilon^{(k+3)}(\varepsilon^{\frac{1}{2}}Y_\varepsilon)\frac{\tau^{k}}{k!}\varepsilon^{\frac{k}{2}} (- \tttwo{\bar{Y}_\varepsilon}-\ttthree{Y_\varepsilon}  + \phi_\varepsilon + \psi_\varepsilon)^{k}\\&+G_\varepsilon^{(m_1)}(\varepsilon^{\frac{1}{2}}Y_\varepsilon+\theta_\tau\varepsilon^{\frac{1}{2}}v_\varepsilon)
\frac{\tau^{m_1-3}}{(m_1-3)!}\varepsilon^{\frac{m_1-3}{2}} (- \tttwo{\bar{Y}_\varepsilon}-\ttthree{Y_\varepsilon}  + \phi_\varepsilon + \psi_\varepsilon)^{m_1-3}\bigg]\frac{(1-\tau)^2}{2!}d\tau
\\=&\sum_{l=4}^{m_1-1}{\varepsilon^{\frac{l-3}{2}}c_{l,\varepsilon} (- \tttwo{\bar{Y}_\varepsilon}-\ttthree{Y_\varepsilon}  + \phi_\varepsilon + \psi_\varepsilon)^{l}}+\sum_{l=4}^{m_1-1}\varepsilon^{\frac{l-3}{2}}Y^{\emptyset,l}_{2,\varepsilon} (- \tttwo{\bar{Y}_\varepsilon}-\ttthree{Y_\varepsilon}  + \phi_\varepsilon + \psi_\varepsilon)^{l}\\&+\rmb{\int_0^1G_\varepsilon^{(m_1)}(\varepsilon^{\frac{1}{2}}Y_\varepsilon+\theta_\tau\varepsilon^{\frac{1}{2}}v_\varepsilon)
\varepsilon^{\frac{m_1-3}{2}} (- \tttwo{\bar{Y}_\varepsilon}-\ttthree{Y_\varepsilon}  + \phi_\varepsilon + \psi_\varepsilon)^{m_1}\frac{\tau^{m_1-3}}{(m_1-3)!}\frac{(1-\tau)^2}{2!}d\tau},\endaligned\end{equation}
with $\theta_\tau\in(0,1)$ and $Y^{\emptyset,l}_{2,\varepsilon}=\frac{1}{l!}\big[G_\varepsilon^{(l)}(\varepsilon^{\frac{1}{2}}Y_\varepsilon)-\mathbb{E}[G_\varepsilon^{(l)}
(\varepsilon^{\frac{1}{2}}Y_\varepsilon)]\big]$ and $c_{l,\varepsilon}=\frac{1}{l!}\mathbb{E}[G_\varepsilon^{(l)}(\varepsilon^{\frac{1}{2}}Y_\varepsilon)]$.
Combining the first term on the right hand of (3.11) and (3.12) we have
$$\aligned&\sum_{l=4,l \textrm{ even}}^{m_1-1}\rmb{\varepsilon^{\frac{l-3}{2}}c_{l,\varepsilon} (- \tttwo{\bar{Y}_\varepsilon}-\ttthree{Y_\varepsilon}  + \phi_\varepsilon + \psi_\varepsilon)^{l}}+\sum_{l=4,l \textrm{ odd}}^{m-1}\varepsilon^{\frac{l-3}{2}} (a_{l,\varepsilon}\vee0)(- \tttwo{\bar{Y}_\varepsilon}-\ttthree{Y_\varepsilon}  + \phi_\varepsilon + \psi_\varepsilon)^{l}\\&+\sum_{l=4,l \textrm{ odd}}^{m-1}\rmb{\varepsilon^{\frac{l-3}{2}} (a_{l,\varepsilon}\wedge0)(- \tttwo{\bar{Y}_\varepsilon}-\ttthree{Y_\varepsilon}  + \phi_\varepsilon + \psi_\varepsilon)^{l}},\endaligned$$
with $a_{l,\varepsilon}=b_{l,\varepsilon} \mathbb{E}[(\varepsilon^{\frac{1}{2}}Y_\varepsilon)^{m-l}]+c_{l,\varepsilon}$ and $c_{l,\varepsilon}=0$ for $l\geq m_1$. For $l=4,...,m-1,$ odd we have
$$\varepsilon^{\frac{l-3}{2}} (a_{l,\varepsilon}\vee0)(- \tttwo{\bar{Y}_\varepsilon}-\ttthree{Y_\varepsilon}  + \phi_\varepsilon + \psi_\varepsilon)^{l}=\sum_{k=0}^{l-1}\rmb{\varepsilon^{\frac{l-3}{2}}C_l^k (a_{l,\varepsilon}\vee0)(- \tttwo{\bar{Y}_\varepsilon}-\ttthree{Y_\varepsilon}  + \phi_\varepsilon)^{l-k}\psi^k_\varepsilon}+\varepsilon^{\frac{l-3}{2}} (a_{l,\varepsilon}\vee0)\psi_\varepsilon^l.$$

For the second term on the right hand of (3.12) we have
$$\aligned &\varepsilon^{\frac{l-3}{2}}Y^{\emptyset,l}_{2,\varepsilon} (- \tttwo{\bar{Y}_\varepsilon}-\ttthree{Y_\varepsilon}  + \phi_\varepsilon + \psi_\varepsilon)^{l}\\=&
\rmb{Y^{\emptyset,l}_{2,\varepsilon}\preccurlyeq\varepsilon^{\frac{l-3}{2}} (- \tttwo{\bar{Y}_\varepsilon}-\ttthree{Y_\varepsilon}  + \phi_\varepsilon + \psi_\varepsilon)^{l}}+\rmb{\VV_\leq Y^{\emptyset,l}_{2,\varepsilon}\succ\varepsilon^{\frac{l-3}{2}} (- \tttwo{\bar{Y}_\varepsilon}-\ttthree{Y_\varepsilon}  + \phi_\varepsilon + \psi_\varepsilon)^{l}}\\&+\rmm{\VV_>Y^{\emptyset,l}_{2,\varepsilon}\succ\varepsilon^{\frac{l-3}{2}} (- \tttwo{\bar{Y}_\varepsilon}-\ttthree{Y_\varepsilon}  + \phi_\varepsilon + \psi_\varepsilon)^{l}}.\endaligned$$

Now, let $\rmm{\Phi}$ be the sum of all the  magenta terms above and $\rmb{\Psi}$ the sum of all the blue terms. We require that, separately,
\begin{align}\label{eq:two}
 \LL \phi_\varepsilon + \rmm{\Phi_\varepsilon} = 0, \qquad \LL \psi_\varepsilon + \lambda_3\psi_\varepsilon^3 +C_{0,\varepsilon}\varepsilon^{\frac{m-3}{2}}\psi^m_\varepsilon+ \sum_{l=4,l \textrm{ odd}}^{m-1}\varepsilon^{\frac{l-3}{2}} (a_{l,\varepsilon}\vee0)\psi^l_\varepsilon+ \rmb{\Psi_\varepsilon} = 0.
 \end{align}

\section{Uniform estimates}

In the following we fix several small enough positive parameters
$$0<\alpha<\frac{3}{10m},\quad 0<\epsilon<\gamma<\kappa<\alpha,\quad \gamma=\alpha-\kappa,\quad \gamma+\epsilon<4\alpha^2,\quad \gamma_1>4\alpha.$$
$\sigma>0$ is also a small constant, which may change in different estimates below. $0<\delta, \delta_0<1$ are also constants, which may change from line to line.
Set
$$\Omega_\varepsilon^1:=\{\|Y_\varepsilon^\emptyset-\lambda_3\|_{C_T\CC^{-\frac{1}{2}+\frac{3}{2m}-\frac{3\epsilon}{4}}(\rho^\sigma)}\lesssim \varepsilon^{\frac{1}{2}-\frac{3}{2m}+\frac{\epsilon}{2}}, 
\|Y_\varepsilon^\emptyset-\lambda_3\|_{C\CC^{-\frac{1}{2}+\frac{\epsilon}{2}}(\rho^{\sigma})}\lesssim \varepsilon^{\frac{1}{2}-\frac{3\epsilon}{4}}\}.$$
$$\Omega_\varepsilon^2:=\{\|Y_{1,\varepsilon}^{\emptyset,m-l}\|_{C_T\CC^{-\kappa}(\rho^\sigma)}\lesssim \varepsilon^{\frac{\kappa}{2}},\|Y_{1,\varepsilon}^{\emptyset,m-l}\|_{C_T\CC^{-2\alpha-\epsilon}(\rho^\sigma)}\lesssim \varepsilon^{2\alpha+\frac{\epsilon}{2}},  l=4,...,m-1\}.$$
$$\Omega_\varepsilon^3:=\{\|Y_{2,\varepsilon}^{\emptyset,l}\|_{C_T\CC^{-\kappa}(\rho^\sigma)}\lesssim \varepsilon^{\frac{\kappa}{2}}, \|Y_{2,\varepsilon}^{\emptyset,l}\|_{C_T\CC^{-2\alpha-\epsilon}(\rho^\sigma)}\lesssim \varepsilon^{2\alpha+\frac{\epsilon}{2}} , l=4,...,m_1-1\}.$$
$$\aligned\Omega_{\varepsilon,M_0}^4:=&\{\|\tau_{\varepsilon}\|_{C_T\CC^{\alpha_{\tau}}(\rho^{\sigma})}\leq M_0, \| \tttwo{\bar{Y}_\varepsilon}\|_{C_T^{\alpha+\kappa}L^\infty(\rho^{\sigma})}+\|\ttthree{Y_\varepsilon} \|_{C_T^{\alpha+\kappa}L^\infty(\rho^{\sigma})}\leq M_0, \\&\tau_{\varepsilon }\rightarrow \tau \textrm{ in } C_T\CC^{\alpha_{\tau}}(\rho^{\sigma}), \quad
\tttwo{\bar{Y}_\varepsilon}\rightarrow \tttwo{X},\quad \ttthree{Y_\varepsilon}\rightarrow\ttthree{X} \textrm{ in }C_T^{\alpha+\kappa}L^\infty(\rho^{\sigma}), \\&\textrm{ where }\tau_{\varepsilon} \textrm{ is an component in }\mathbb{Y}_\varepsilon,   \tau \textrm{ the corresponding component in } \mathbb{X}\},\endaligned$$
where $M_0>0$.
In the following  we do estimate on $ \Omega_{\varepsilon,M_0}:=\cap_{i=1}^3\Omega_\varepsilon^i\cap\Omega_{\varepsilon,M_0}^4 $.

For fixed $T>0$ and $M$ large enough, define
$$T_{\varepsilon,M}:=\inf\{t>0, \varepsilon^{\frac{m-3}{2}}(\|\phi_\varepsilon\|^m_{C_t L^\infty(\rho^{\frac{3+6\alpha}{2m}})}+\|\psi_\varepsilon\|^m_{C_t L^\infty(\rho^{\frac{3+6\alpha}{2m}})})>M\}\wedge T.$$
We also choose $\varepsilon$ small enough such that $\varepsilon^{\frac{\epsilon}{2}}M\leq 1$. In the following we first do the estimates before $T_{\varepsilon,M}$. Since all the constant in the estimates below are independent of $\varepsilon$ and $M$, we will omit $T_{\varepsilon,M}$ in the estimate below for simplicity. In this case we choose $(L_k)$ in the construction of $\VV_>, \VV_\leq$ similar as in [GH18] and choose the weight in a more delicate way. For the terms containing $Y^{\emptyset}$, the order of $\phi$ and $\psi$ is higher than $3$. In this case we use $\varepsilon^{\frac{m-3}{2}}\|\psi\|^m_{L^\infty L^\infty(\rho^{\frac{3+6\alpha}{2}})}$ to control it.

 \subsection{Bound for $\phi$ in  $C\CC^\alpha(\rho^{\frac{3+6\alpha}{2m}})$}
 For the terms similar as in [GH18], we obtain similar regularity estimates but we choose the weight in a different way. We put this part in Appendix A. In the following we mainly concentrate on the terms containing $Y_\varepsilon^\emptyset$ and on the terms deduced from $R_\varepsilon$.

We use Lemma 2.4 and have the following estimates
$$\aligned\|\rmm{2\lambda_{2,\varepsilon} \VV_>Y_\varepsilon\succ(\phi_\varepsilon + \psi_\varepsilon)}\|_{C\CC^{-2+\alpha}(\rho^{\frac{3+6\alpha}{2m}})}\lesssim&\|{\VV_>Y_\varepsilon\|_{C\CC^{-2+\alpha}(\rho^{\frac{3+6\alpha}{2m}-\frac{1}{2}-\alpha})}\|\phi_\varepsilon + \psi_\varepsilon}\|_{L^\infty L^\infty(\rho^{\frac{1}{2}+\alpha})}\\\lesssim&2^{-(\frac{3}{2}-\alpha-\kappa)\frac{2L}{3}}\| \phi_\varepsilon + \psi_\varepsilon\|_{L^\infty L^\infty(\rho^{\frac{1}{2}+\alpha})},\endaligned$$

$$\aligned&\|\rmm{\VV_>(Y^\emptyset_\varepsilon-\lambda_3)\succ(- \tttwo{\bar{Y}_\varepsilon}-\ttthree{Y_\varepsilon}  + \phi_\varepsilon + \psi_\varepsilon)^3}\|_{C\CC^{-2+\alpha}(\rho^{\frac{3+6\alpha}{2m}})}
\\\lesssim&\|\VV_>(Y^\emptyset_\varepsilon-\lambda_3)\|_{C\CC^{-2+\alpha}(\rho^{-\frac{3}{2}-3\alpha+\frac{3+6\alpha}{2m}})}(1+\| \phi_\varepsilon + \psi_\varepsilon\|_{L^\infty L^\infty(\rho^{\frac{1}{2}+\alpha})}^3)
\\\lesssim&2^{-(2-\alpha-\kappa)\frac{3L}{2}}\|Y^\emptyset_\varepsilon-\lambda_3\|_{C\CC^{-\kappa}(\rho^{\sigma})}(1+\| \phi_\varepsilon + \psi_\varepsilon\|_{L^\infty L^\infty(\rho^{\frac{1}{2}+\alpha})}^3),\endaligned$$

For $l=4,...,m-1$, we have on $\Omega_{\varepsilon,M_0}$

$$\aligned&\|\rmm{\VV_>(Y^{\emptyset,m-l}_{1,\varepsilon})\succ\varepsilon^{\frac{l-3}{2}} (- \tttwo{\bar{Y}_\varepsilon}-\ttthree{Y_\varepsilon}  + \phi_\varepsilon + \psi_\varepsilon)^{l}}\|_{C\CC^{-2+\alpha}(\rho^{\frac{3+6\alpha}{2m}})}
\\\lesssim&\|\VV_>(Y^{\emptyset,m-l}_{1,\varepsilon})\|_{C\CC^{-2+\alpha}(\rho^{-\frac{3}{2}-3\alpha+\frac{3+6\alpha}{2m}})}
\varepsilon^{\frac{l-3}{2}}\| (- \tttwo{\bar{Y}_\varepsilon}-\ttthree{Y_\varepsilon}+\phi_\varepsilon + \psi_\varepsilon)^{l}\|_{L^\infty L^\infty(\rho^{\frac{3+6\alpha}{2}})}\\\lesssim&\|\VV_>(Y^{\emptyset,m-l}_{1,\varepsilon})\|_{C\CC^{-2+\alpha}(\rho^{-\frac{3}{2}-3\alpha+\frac{3+6\alpha}{2m}})}
[\varepsilon^{\frac{m-3}{2}}\| (- \tttwo{\bar{Y}_\varepsilon}-\ttthree{Y_\varepsilon}+\phi_\varepsilon + \psi_\varepsilon)^{m}\|_{L^\infty L^\infty(\rho^{\frac{3+6\alpha}{2}})}]^{\frac{l-3}{m-3}}\\&(1+\|\phi_\varepsilon+\psi_\varepsilon\|_{L^\infty L^\infty(\rho^{\frac{1}{2}+\alpha})}^{\frac{3(m-l)}{m-3}})\\\lesssim&2^{-(2-\alpha-\kappa)\frac{3}{2} L}(1+\|\phi_\varepsilon+\psi_\varepsilon\|_{L^\infty L^\infty(\rho^{\frac{1}{2}+\alpha})}^{\frac{3(m-l)}{m-3}}),\endaligned$$
where
the constant we omit is independent of $M$ and we use Lemma 2.2 to have $$\|\VV_>(Y^{\emptyset,m-l}_{1,\varepsilon})\|_{C\CC^{-2+\alpha}(\rho^{-\frac{3}{2}-3\alpha+\frac{3+6\alpha}{2m}})}\lesssim2^{-(2-\alpha-\kappa)\frac{3}{2} L}
\|Y^{\emptyset,m-l}_{1,\varepsilon}\|_{C\CC^{-\kappa}(\rho^{\sigma})},$$
  $\|Y^{\emptyset,m-l}_{1,\varepsilon}\|_{C\CC^{-\kappa}(\rho^\sigma)}\lesssim \varepsilon^{\frac{\kappa}{2}}$ and $\varepsilon^\frac{\kappa}{2}M\leq 1$ in the last step.

Similarly, for $l=4,...,m_1-4$ we also have on $\Omega_{\varepsilon,M_0}$

$$\aligned&\|\rmm{\VV_>Y^{\emptyset,l}_{2,\varepsilon}\succ\varepsilon^{\frac{l-3}{2}} (- \tttwo{\bar{Y}_\varepsilon}-\ttthree{Y_\varepsilon}  + \phi_\varepsilon + \psi_\varepsilon)^{l}}\|_{C\CC^{-2+\alpha}(\rho^{\frac{3+6\alpha}{2m}})}
\\\lesssim&2^{-(2-\alpha-\kappa)\frac{3}{2} L}(1+\|\phi_\varepsilon+\psi_\varepsilon\|_{L^\infty L^\infty(\rho^{\frac{1}{2}+\alpha})}^{\frac{3(m-l)}{m-3}}).\endaligned$$

Combining all the above estimates and the extra estimates in Appendix A.1 we have

$$\aligned&\|\phi_\varepsilon\|_{L^\infty L^\infty(\rho^{\frac{3+6\alpha}{2m}})}\lesssim\|\phi_\varepsilon\|_{C\CC^\alpha(\rho^{\frac{3+6\alpha}{2m}})}
+\|\phi_\varepsilon\|_{C^{\frac{\alpha}{2}}L^\infty(\rho^{\frac{3+6\alpha}{2m}})}
\lesssim \|\Phi_\varepsilon\|_{C\CC^{-2+\alpha}(\rho^{\frac{3+6\alpha}{2m}})}\\\lesssim&1+2^{-(1-\alpha-\kappa)L}\| \phi_\varepsilon + \psi_\varepsilon\|_{L^\infty L^\infty(\rho^{\frac{1}{2}+\alpha})}+2^{-(\frac{3}{2}-\alpha-\kappa)\frac{4}{3}L}\| \phi_\varepsilon + \psi_\varepsilon\|_{L^\infty L^\infty(\rho^{\frac{1}{2}+\alpha})}^2\\&+2^{-(2-\alpha-\kappa)\frac{3L}{2}}(1+\| \phi_\varepsilon + \psi_\varepsilon\|_{L^\infty L^\infty(\rho^{\frac{1}{2}+\alpha})}^3)\\&+2^{-(2-\alpha-\kappa)\frac{3}{2} L}\sum_{l=4}^{m-1}(1+\|\phi_\varepsilon+\psi_\varepsilon\|_{L^\infty L^\infty(\rho^{\frac{1}{2}+\alpha})}^{\frac{3(m-l)}{m-3}}).\endaligned$$
Choose $L>0$ such that $\|\phi_\varepsilon+\psi_\varepsilon\|_{L^\infty L^\infty(\rho^{\frac{1}{2}+\alpha})}\lesssim 2^{(1-\alpha-\kappa)L}$ we have
\begin{equation}
\|\phi_\varepsilon\|_{L^\infty L^\infty(\rho^{\frac{3+6\alpha}{2m}})}\lesssim\|\phi_\varepsilon\|_{C\CC^\alpha(\rho^{\frac{3+6\alpha}{2m}})}+\|\phi_\varepsilon\|_{C^{\frac{\alpha}{2}}
L^\infty(\rho^{\frac{3+6\alpha}{2m}})}\lesssim1,
\end{equation}
where the omit constant is independent of $M$. So we can choose $L>0$ such that $\|\phi_\varepsilon+\psi_\varepsilon\|_{L^\infty L^\infty(\rho^{\frac{1}{2}+\alpha})}\lesssim 1+\|\psi_\varepsilon\|_{L^\infty L^\infty(\rho^{\frac{1}{2}+\alpha})}\backsimeq2^{(1-\alpha-\kappa)L}$.

\subsection{Bound for $\phi$ in  $C\CC^{\frac{1}{2}+\alpha}(\rho^{\frac{3+6\alpha}{2m}})$}
In the following we estimate $\Phi$ in $C\CC^{-\frac{3}{2}+\alpha}(\rho^{\frac{3+6\alpha}{2m}})$. We also put the estimates for the terms, which are similar as in [GH18], in Appendix A.2.  For other terms we have

$$\aligned\|\rmm{2\lambda_{2,\varepsilon} \VV_>Y_\varepsilon\succ(\phi_\varepsilon + \psi_\varepsilon)}\|_{C\CC^{-\frac{3}{2}+\alpha}(\rho^{\frac{3+6\alpha}{2m}})}\lesssim&2^{-(1-\alpha-\kappa)\frac{2L}{3}}\| \phi_\varepsilon + \psi_\varepsilon\|_{L^\infty L^\infty(\rho^{\frac{1}{2}+\alpha})}\lesssim 1+\|\psi_\varepsilon\|_{L^\infty L^\infty(\rho^{\frac{1}{2}+\alpha})}^\delta,\endaligned$$

$$\aligned&\|\rmm{\VV_>(Y_\varepsilon^\emptyset-\lambda_3)\succ(- \tttwo{\bar{Y}_\varepsilon}-\ttthree{Y_\varepsilon}  + \phi_\varepsilon + \psi_\varepsilon)^3}\|_{C\CC^{-\frac{3}{2}+\alpha}(\rho^{\frac{3+6\alpha}{2m}})}
\\\lesssim&\|\VV_>(Y_\varepsilon^\emptyset-\lambda_3)\|_{C\CC^{-\frac{3}{2}+\alpha}(\rho^{-\frac{3}{2}-3\alpha+\frac{3+6\alpha}{2m}})}(1+\| \phi_\varepsilon + \psi_\varepsilon\|_{L^\infty L^\infty(\rho^{\frac{1}{2}+\alpha})}^3)
\\\lesssim&2^{-(\frac{3}{2}-\alpha-\kappa)\frac{3L}{2}}\|Y_\varepsilon^\emptyset-\lambda_3\|_{C\CC^{-\kappa}(\rho^{\sigma})}(1+\| \phi_\varepsilon + \psi_\varepsilon\|_{L^\infty L^\infty(\rho^{\frac{1}{2}+\alpha})}^3)\lesssim 1+\|\psi_\varepsilon\|_{L^\infty L^\infty(\rho^{\frac{1}{2}+\alpha})}^\delta,\endaligned$$
where we used Lemma 2.4 to deduce $$\|3\VV_>(Y_\varepsilon^\emptyset-\lambda_3)\|_{C\CC^{-\frac{3}{2}+\alpha}(\rho^{-\frac{3}{2}-3\alpha+\frac{3+6\alpha}{2m}})}\lesssim 2^{-(\frac{3}{2}-\alpha-\kappa)\frac{3L}{2}}\|Y_\varepsilon^\emptyset-\lambda_3\|_{C\CC^{-\kappa}(\rho^{\sigma})},$$
and $5\alpha<\frac{3}{2m}$.

For $l=4,...,m-1$, we have on $\Omega_{\varepsilon,M_0}$
$$\aligned&\|\rmm{\VV_>Y_{1,\varepsilon}^{\emptyset,m-l}\succ\varepsilon^{\frac{l-3}{2}} (- \tttwo{\bar{Y}_\varepsilon}-\ttthree{Y_\varepsilon}  + \phi_\varepsilon + \psi_\varepsilon)^{l}}\|_{C\CC^{-\frac{3}{2}+\alpha}(\rho^{\frac{3+6\alpha}{2m}})}
\\\lesssim&\|\VV_>Y_{1,\varepsilon}^{\emptyset,m-l}\|_{C\CC^{-\frac{3}{2}+\alpha}(\rho^{-\frac{3}{2}-3\alpha+\frac{3+6\alpha}{2m}})}
\varepsilon^{\frac{l-3}{2}}\| (- \tttwo{\bar{Y}_\varepsilon}-\ttthree{Y_\varepsilon}+\phi_\varepsilon + \psi_\varepsilon)^{l}\|_{L^\infty L^\infty(\rho^{\frac{3+6\alpha}{2}})}\\\lesssim&\|\VV_>Y_{1,\varepsilon}^{\emptyset,m-l}\|_{C\CC^{-\frac{3}{2}+\alpha}(\rho^{-\frac{3}{2}-3\alpha+\frac{3+6\alpha}{2m}})}
[\varepsilon^{\frac{m-3}{2}}\| (- \tttwo{\bar{Y}_\varepsilon}-\ttthree{Y_\varepsilon}+\phi_\varepsilon + \psi_\varepsilon)^{m}\|_{L^\infty L^\infty(\rho^{\frac{3+6\alpha}{2}})}]^{\frac{l-3}{m-3}}\\&(1+\|\phi_\varepsilon+\psi_\varepsilon\|_{L^\infty L^\infty(\rho^{\frac{1}{2}+\alpha})}^{\frac{3(m-l)}{m-3}})\\\lesssim&2^{-(\frac{3}{2}-\alpha-\kappa)\frac{3}{2} L}(1+\|\phi_\varepsilon+\psi_\varepsilon\|_{L^\infty L^\infty(\rho^{\frac{1}{2}+\alpha})}^{\frac{3(m-l)}{m-3}})\lesssim 1+\|\psi_\varepsilon\|_{L^\infty L^\infty(\rho^{\frac{1}{2}+\alpha})}^\delta,\endaligned$$
where
the constant we omit is independent of $M$ and we use $\|\VV_>(Y^{\emptyset,m-l}_{1,\varepsilon})\|_{C\CC^{-\frac{3}{2}+\alpha}(\rho^{-\frac{3}{2}-3\alpha+\frac{3+6\alpha}{2m}})}
\lesssim2^{-(\frac{3}{2}-\alpha-\kappa)\frac{3}{2} L}
\|Y^{\emptyset,m-l}_{1,\varepsilon}\|_{C\CC^{-\kappa}(\rho^{\sigma})}$,  $\|Y^{\emptyset,m-l}_{1,\varepsilon}\|_{C\CC^{-\kappa}(\rho^\sigma)}\lesssim \varepsilon^{\frac{\kappa}{2}}$ and $\varepsilon^\frac{\kappa}{2}M\leq 1$ in the last step.

Similar estimates also hold for the term containing $Y^{\emptyset,l}_{2,\varepsilon}$.

Combining all the above estimates and the estimates in Appendix A.2 we have

\begin{equation}\aligned&\|\phi_\varepsilon\|_{C\CC^{\frac{1}{2}+\alpha}(\rho^{\frac{3+6\alpha}{2m}})}
+\|\phi_\varepsilon\|_{C^{\frac{1}{4}+\frac{\alpha}{2}}L^\infty(\rho^{\frac{3+6\alpha}{2m}})}
\lesssim1+\|\psi_\varepsilon\|_{L^\infty L^\infty(\rho^{\frac{1}{2}+\alpha})}^\delta,\endaligned\end{equation}
where the constant we omit is independent of $M$.

\subsection{Bound for $\vartheta$ in  $C\CC^{1+\alpha}(\rho^{\frac{3}{2}+\gamma'})$}

In this subsection we use another small parameter $\gamma'$ with $0<\gamma'<\gamma_1$ and close to $\gamma_1$.  Now we do the estimate for $\vartheta$.
$$ \LL\vartheta_\varepsilon=\Theta_\varepsilon,$$
with
$$\aligned\Theta_\varepsilon=&\Phi_\varepsilon+\rmm{3\ttwo{Y_\varepsilon}\succ(\tttwo{\bar{Y}_\varepsilon}
+\ttthree{Y_\varepsilon})}-\rmm{3\VV_>\ttwo{Y_\varepsilon}\succ(\phi_\varepsilon+\psi_\varepsilon)}
-\rmb{3\VV_\leq\ttwo{Y_\varepsilon}\succ(\phi_\varepsilon+\psi_\varepsilon)}\\&-3\Big( (- \tttwo{\bar{Y}_\varepsilon}-\ttthree{Y_\varepsilon}  + \phi_\varepsilon + \psi_\varepsilon)\Prec\ttwo{Y_\varepsilon}- (- \tttwo{\bar{Y}_\varepsilon}-\ttthree{Y_\varepsilon}  + \phi_\varepsilon + \psi_\varepsilon)\prec\ttwo{Y_\varepsilon}\Big)\\
&-3[\LL,(- \tttwo{\bar{Y}_\varepsilon}-\ttthree{Y_\varepsilon}+\phi_\varepsilon+\psi_\varepsilon)\Prec]\tttwo{Y_\varepsilon}.\endaligned$$
Similarly as above we put all the terms similar as in [GH18] in Appendix A.3. For other terms we have

$$\aligned\|\rmm{2\lambda_{2,\varepsilon} \VV_>Y_\varepsilon\succ(\phi_\varepsilon + \psi_\varepsilon)}\|_{C\CC^{-1+\alpha}(\rho^{\frac{3+6\alpha}{2m}})}\lesssim&2^{-(\frac{1}{2}-\alpha-\kappa)\frac{2L}{3}}\| \phi_\varepsilon + \psi_\varepsilon\|_{L^\infty L^\infty(\rho^{\frac{1}{2}+\alpha})}\lesssim 1+\|\psi_\varepsilon\|_{L^\infty L^\infty(\rho^{\frac{1}{2}+\alpha})}^\delta,\endaligned$$

$$\aligned&\|\rmm{\VV_>(Y^\emptyset_\varepsilon-\lambda_3)\succ(- \tttwo{\bar{Y}_\varepsilon}-\ttthree{Y_\varepsilon}  + \phi_\varepsilon + \psi_\varepsilon)^3}\|_{C\CC^{-1+\alpha}(\rho^{\frac{3+6\alpha}{2m}+\frac{1}{2}+\alpha})}
\\\lesssim&\|\VV_>(Y^\emptyset_\varepsilon-\lambda_3)\|_{C\CC^{-1+\alpha}(\rho^{-1-2\alpha+\frac{3+6\alpha}{2m}})}(1+\| \phi_\varepsilon + \psi_\varepsilon\|_{L^\infty L^\infty(\rho^{\frac{1}{2}+\alpha})}^3)
\\\lesssim&2^{-(1-\alpha-\kappa)\frac{3L}{2}}\|Y^\emptyset_\varepsilon-\lambda_3\|_{C\CC^{-\kappa}(\rho^{\sigma})}(1+\| \phi_\varepsilon + \psi_\varepsilon\|_{L^\infty L^\infty(\rho^{\frac{1}{2}+\alpha})}^3)\lesssim 1+\|\psi_\varepsilon\|_{L^\infty L^\infty(\rho^{\frac{1}{2}+\alpha})}^{1+\delta},\endaligned$$

For $l=4,...,m-1$, we have
$$\aligned&\|\rmm{\VV_>Y^{\emptyset,m-l}_{1,\varepsilon}\succ\varepsilon^{\frac{l-3}{2}} (- \tttwo{\bar{Y}_\varepsilon}-\ttthree{Y_\varepsilon}  + \phi_\varepsilon + \psi_\varepsilon)^{l}}\|_{C\CC^{-1+\alpha}(\rho^{\frac{3+6\alpha}{2m}+\frac{1}{2}+\alpha})}
\\\lesssim&\|\VV_>Y^{\emptyset,m-l}_{1,\varepsilon}\|_{C\CC^{-1+\alpha}(\rho^{-1-2\alpha+\frac{3+6\alpha}{2m}})}
\varepsilon^{\frac{l-3}{2}}\| (- \tttwo{\bar{Y}_\varepsilon}-\ttthree{Y_\varepsilon}+\phi_\varepsilon + \psi_\varepsilon)^{l}\|_{L^\infty L^\infty(\rho^{\frac{3+6\alpha}{2}})}\\\lesssim&\|\VV_>Y^{\emptyset,m-l}_{1,\varepsilon}\|_{C\CC^{-1+\alpha}(\rho^{-1-2\alpha+\frac{3+6\alpha}{2m}})}
[\varepsilon^{\frac{m-3}{2}}\| (- \tttwo{\bar{Y}_\varepsilon}-\ttthree{Y_\varepsilon}+\phi_\varepsilon + \psi_\varepsilon)^{m}\|_{L^\infty L^\infty(\rho^{\frac{3+6\alpha}{2}})}]^{\frac{l-3}{m-3}}\\&(1+\|\phi_\varepsilon+\psi_\varepsilon\|_{L^\infty L^\infty(\rho^{\frac{1}{2}+\alpha})}^{\frac{3(m-l)}{m-3}})\\\lesssim&2^{-(1-\alpha-\kappa)\frac{3}{2} L}(1+\|\phi_\varepsilon+\psi_\varepsilon\|_{L^\infty L^\infty(\rho^{\frac{1}{2}+\alpha})}^{\frac{3(m-l)}{m-3}})\lesssim 1+\|\psi_\varepsilon\|_{L^\infty L^\infty(\rho^{\frac{1}{2}+\alpha})}^{1+\delta},\endaligned$$
where
the constant we omit is independent of $M$ and we use $\|\VV_>(Y^{\emptyset,m-l}_{1,\varepsilon})\|_{C\CC^{-1+\alpha}(\rho^{-1-2\alpha+\frac{3+6\alpha}{2m}})}\lesssim2^{-(1-\alpha-\kappa)\frac{3}{2} L}
\|Y^{\emptyset,m-l}_{1,\varepsilon}\|_{C\CC^{-\kappa}(\rho^{\sigma})}$,  $\|Y^{\emptyset,m-l}_{1,\varepsilon}\|_{C\CC^{-\kappa}(\rho^\sigma)}\lesssim \varepsilon^{\frac{\kappa}{2}}$ and $\varepsilon^\frac{\kappa}{2}M\leq 1$ in the last step.
For the term containing $Y^{\emptyset,l}_{2,\varepsilon}$  we have similar estimates. Furthermore,  we have
$$\aligned&\|\rmb{3\VV_\leq\ttwo{Y_\varepsilon}\succ(\phi_\varepsilon+\psi_\varepsilon)}\|_{C\CC^{-1+\alpha}(\rho^{1/2+2\alpha+2\kappa})}
\lesssim\|\VV_\leq\ttwo{Y_\varepsilon}\|_{C\CC^{-1+\alpha}(\rho^{\alpha+2\kappa})}(1+\| \psi_\varepsilon\|_{L^\infty L^\infty(\rho^{\frac{1}{2}+\alpha})})\\\lesssim&2^{(\alpha+\kappa)L}(1+\| \psi_\varepsilon\|_{L^\infty L^\infty(\rho^{\frac{1}{2}+\alpha})})\lesssim 1+\|\psi_\varepsilon\|_{L^\infty L^\infty(\rho^{\frac{1}{2}+\alpha})}^{1+\delta},\endaligned$$

For $\gamma''<\gamma'<\gamma_1$ and $\gamma''$ is close to $\gamma$, by Lemma 2.7 we have

$$\aligned &3\|(- \tttwo{\bar{Y}_\varepsilon}-\ttthree{Y_\varepsilon}  + \phi_\varepsilon + \psi_\varepsilon)\Prec\ttwo{Y_\varepsilon}- (- \tttwo{\bar{Y}_\varepsilon}-\ttthree{Y_\varepsilon}  + \phi_\varepsilon + \psi_\varepsilon)\prec\ttwo{Y_\varepsilon}\|_{C\CC^{-1+\alpha}(\rho^{\frac{3}{2}+\gamma'})}\\\lesssim&\|- \tttwo{\bar{Y}_\varepsilon}-\ttthree{Y_\varepsilon}  + \phi_\varepsilon + \psi_\varepsilon\|_{C^{(\alpha+\kappa)/2}L^\infty(\rho^{\frac{3}{2}+\gamma''})}\|\ttwo{Y_\varepsilon}\|_{C\CC^{-1-\kappa}(\rho^{\sigma})}\\\lesssim&1+\| \phi_\varepsilon + \psi_\varepsilon\|_{C^{(\alpha+\kappa)/2}L^\infty(\rho^{\frac{3}{2}+\gamma''})},
\endaligned$$

$$\aligned &\|[\LL,(- \tttwo{\bar{Y}_\varepsilon}-\ttthree{Y_\varepsilon}+\phi_\varepsilon+\psi_\varepsilon)\Prec]\tttwo{Y_\varepsilon}\|_{C\CC^{-1+\alpha}(\rho^{\frac{3}{2}+\gamma'})}
\\\lesssim&(\|- \tttwo{\bar{Y}_\varepsilon}-\ttthree{Y_\varepsilon}  + \phi_\varepsilon + \psi_\varepsilon\|_{C^{(\alpha+\kappa)/2}L^\infty(\rho^{\frac{3}{2}+\gamma''})}
+\|- \tttwo{\bar{Y}_\varepsilon}-\ttthree{Y_\varepsilon}  + \phi_\varepsilon + \psi_\varepsilon\|_{C\CC^{\alpha+\kappa}(\rho^{\frac{3}{2}+\gamma''})})\|\tttwo{Y_\varepsilon}\|_{C\CC^{1-\kappa}(\rho^{\sigma})}\\\lesssim&1+\| \phi_\varepsilon + \psi_\varepsilon\|_{C^{(\alpha+\kappa)/2}L^\infty(\rho^{\frac{3}{2}+\gamma''})}+\| \phi_\varepsilon + \psi_\varepsilon\|_{C\CC^{\alpha+\kappa}(\rho^{\frac{3}{2}+\gamma''})}.\endaligned$$

By interpolation and Young's inequality we have
\begin{equation}\aligned\| \phi_\varepsilon + \psi_\varepsilon\|_{C^{(\alpha+\kappa)/2}L^\infty(\rho^{\frac{3}{2}+\gamma''})}
\lesssim&\delta_0\| \phi_\varepsilon + \psi_\varepsilon\|_{C^{(\alpha+\frac{1}{2})/2}L^\infty(\rho^{\frac{3}{2}+\gamma_1})}+\| \phi_\varepsilon + \psi_\varepsilon\|_{L^\infty L^\infty(\rho^{\frac{1}{2}+\alpha})}
\\\lesssim&\delta_0\| \psi_\varepsilon\|_{C^{1-\gamma}L^\infty(\rho^{\frac{3}{2}+\gamma_1})}+\|\phi_\varepsilon\|_{C^{(\alpha+\frac{1}{2})/2}
L^\infty(\rho^{\frac{1}{2}+\alpha})}+\| \phi_\varepsilon + \psi_\varepsilon\|_{L^\infty L^\infty(\rho^{\frac{1}{2}+\alpha})}.\endaligned\end{equation}

We have
$$\aligned\| \psi_\varepsilon\|_{C\CC^{\alpha+\kappa}(\rho^{\frac{3}{2}+\gamma''})}
\lesssim&\delta_0\|\psi_\varepsilon\|_{C\CC^{\frac{1}{2}+\alpha}(\rho^{\frac{3}{2}+\gamma'})}+\| \psi_\varepsilon\|_{L^\infty L^\infty(\rho^{\frac{1}{2}+\alpha})}
.\endaligned$$

Combining the estimates as above and the estimates in Section A.3, we have
\begin{equation}\|\vartheta_\varepsilon\|_{C^{\frac{1+\alpha}{2}}L^\infty(\rho^{\frac{3}{2}+\gamma'})}+\|\vartheta_\varepsilon\|_{C\CC^{1+\alpha}(\rho^{\frac{3}{2}+\gamma'})}\lesssim1+\delta_0\| \psi_\varepsilon\|_{C^{1-\gamma}L^\infty(\rho^{\frac{3}{2}+\gamma_1})}+\|\psi_\varepsilon\|_{L^\infty L^\infty(\rho^{\frac{1}{2}+\alpha})}^{1+\delta}+\|\psi_\varepsilon\|_{C\CC^{\frac{1}{2}+\alpha}(\rho^{\frac{3}{2}+\gamma'})},\end{equation}
where the constant we omit is independent of $M$.

\subsection{Bound for $\psi$ in  $C\CC^{2-\gamma}(\rho^{\frac{3}{2}+\gamma_1})$}

In this subsection we do the estimates for $\Psi$ in $C\CC^{-\gamma}(\rho^{\frac{3}{2}+\gamma_1})$, which is a little bit different from that in [GH18], since we have to consider $G_\varepsilon^{(m_1)}(\varepsilon^{\frac{1}{2}}Y_\varepsilon+\theta \varepsilon^{\frac{1}{2}}v_\varepsilon)$, which is not in $\CC^\gamma$ uniformly. In this subsection we will use $\|\psi_\varepsilon\|_{L^\infty L^\infty(\rho^{\frac{1}{2}+\alpha})}\backsimeq2^{(1-\alpha-\kappa)L}$ to control $2^{\delta L}$ from the estimate of the terms containing $\VV_\leq$.  We also put  similar terms in Appendix A.4. For other terms we have

$$\aligned&\|\rmb{2\lambda_{2,\varepsilon} \VV_\leq Y_\varepsilon\succ(\phi_\varepsilon + \psi_\varepsilon)}\|_{C\CC^{-\gamma}(\rho^{\frac{3}{2}+\gamma_1})}
\lesssim\| \VV_\leq Y_\varepsilon\|_{C\CC^{-\gamma}(\rho^{1+\gamma_1-\alpha})}\|\phi_\varepsilon + \psi_\varepsilon\|_{L^\infty L^\infty(\rho^{\frac{1}{2}+\alpha})}\\\lesssim& 2^{(\frac{1}{2}-\gamma+\kappa)\frac{2L}{3}}(1+\| \psi_\varepsilon\|_{L^\infty L^\infty(\rho^{\frac{1}{2}+\alpha})})\lesssim1+\| \psi_\varepsilon\|_{L^\infty L^\infty(\rho^{\frac{1}{2}+\alpha})}^{1+\delta},\endaligned$$

$$\aligned&\|\rmb{2\lambda_{2,\varepsilon} Y_\varepsilon\preccurlyeq(\phi_\varepsilon+\psi_\varepsilon)}\|_{C\CC^{-\gamma}(\rho^{\frac{3}{2}+\gamma_1})}
\lesssim\|Y_\varepsilon\|_{C\CC^{-\frac{1}{2}-\kappa}(\rho^{\sigma})}\|\phi_\varepsilon + \psi_\varepsilon\|_{C\CC^{\frac{1}{2}+\alpha}(\rho^{\frac{3}{2}+\gamma''})}\\\lesssim&1+\|\psi_\varepsilon\|_{L^\infty L^\infty(\rho^{\frac{1}{2}+\alpha})}^\delta+\|\psi_\varepsilon\|_{C\CC^{\frac{1}{2}+\alpha}(\rho^{\frac{3}{2}+\gamma''})},\endaligned$$

$$\aligned&\|\rmb{\lambda_{2,\varepsilon}(- \tttwo{\bar{Y}_\varepsilon}-\ttthree{Y_\varepsilon}  + \phi_\varepsilon + \psi_\varepsilon)^2}\|_{C\CC^{-\gamma}(\rho^{\frac{3}{2}+\gamma_1})}\lesssim1+\| \psi_\varepsilon\|_{L^\infty L^\infty(\rho^{\frac{1}{2}+\alpha})}^2.\endaligned$$

$$\aligned&\|\rmb{2\lambda_{2,\varepsilon} \VV_\leq Y_\varepsilon\succ(\phi_\varepsilon + \psi_\varepsilon)}\|_{C\CC^{-\gamma}(\rho^{\frac{3}{2}+\gamma_1})}
\lesssim\|\VV_\leq Y_\varepsilon\|_{C\CC^{-\gamma}(\rho^{1+\gamma_1-\alpha})}\|\phi_\varepsilon+\psi_\varepsilon\|_{L^\infty L^\infty(\rho^{\frac{1}{2}+\alpha})}
\\\lesssim&2^{(\frac{1}{2}+\alpha)\frac{4L}{3}}\| \phi_\varepsilon + \psi_\varepsilon\|_{L^\infty L^\infty(\rho^{\frac{1}{2}+\alpha})}\lesssim1+\| \psi_\varepsilon\|_{L^\infty L^\infty(\rho^{\frac{1}{2}+\alpha})}^{1+\delta},\endaligned$$

$$\aligned&\|\rmb{2\lambda_{2,\varepsilon} {Y_\varepsilon}\preccurlyeq(\phi_\varepsilon+\psi_\varepsilon)}\|_{C\CC^{-\gamma}(\rho^{\frac{3}{2}+\gamma_1})}
\lesssim\|\phi_\varepsilon+\psi_\varepsilon\|_{C\CC^{\frac{1}{2}+\gamma+\kappa}(\rho^{\frac{3}{2}+\gamma_1-\sigma})}
\\\lesssim&
1+\| \psi_\varepsilon\|_{L^\infty L^\infty(\rho^{\frac{1}{2}+\alpha})}^\delta+\| \psi_\varepsilon\|_{C\CC^{\frac{1}{2}+\gamma+\kappa}(\rho^{1+\gamma_1-\sigma-\alpha})},\endaligned$$

$$\aligned&\|\rmb{(-\tttwo{\bar{Y}_\varepsilon}+\phi_\varepsilon+\psi_\varepsilon)^2}\|_{C\CC^{-\gamma}(\rho^{\frac{3}{2}+\gamma_1})}
\lesssim1+\| \phi_\varepsilon + \psi_\varepsilon\|_{L^\infty L^\infty(\rho^{\frac{1}{2}+\alpha})}^2.\endaligned$$

For the terms containing $Y^\emptyset$ we have

$$\aligned&\|\rmb{(Y^\emptyset_\varepsilon-\lambda_3)\preccurlyeq(- \tttwo{\bar{Y}_\varepsilon}-\ttthree{Y_\varepsilon}  + \phi_\varepsilon + \psi_\varepsilon)^3}\|_{C\CC^{-\gamma}(\rho^{\frac{3}{2}+\gamma_1})}
\\\lesssim&\|Y^\emptyset_\varepsilon-\lambda_3\|_{C\CC^{-\kappa}(\rho^{\sigma})}(1+\| (\phi_\varepsilon + \psi_\varepsilon)^3\|_{C\CC^{\kappa+\gamma}(\rho^{\frac{3}{2}+\gamma_1-\sigma})})
\\\lesssim&\|Y^\emptyset_\varepsilon-\lambda_3\|_{C\CC^{-\kappa}(\rho^{\sigma})}(1+\| \phi_\varepsilon + \psi_\varepsilon\|_{C\CC^{\kappa+\gamma}(\rho^{\frac{1}{2}+\gamma_1-2\alpha-\sigma})})(1+\| \phi_\varepsilon + \psi_\varepsilon\|_{L^\infty L^\infty(\rho^{\frac{1}{2}+\alpha})}^2)\\\lesssim&(1+\|\psi_\varepsilon\|_{C\CC^{\kappa+\gamma}(\rho^{\frac{1}{2}+\gamma_1-2\alpha-\sigma})})(1+\| \psi_\varepsilon\|_{L^\infty L^\infty(\rho^{\frac{1}{2}+\alpha})}^2),\endaligned$$

and
$$\aligned&\|\rmb{\VV_\leq(Y^\emptyset_\varepsilon-\lambda_3)\succ(- \tttwo{\bar{Y}_\varepsilon}-\ttthree{Y_\varepsilon}  + \phi_\varepsilon + \psi_\varepsilon)^3}\|_{C\CC^{-\gamma}(\rho^{\frac{3}{2}+\gamma_1})}
\\\lesssim&\|\VV_\leq(Y^\emptyset_\varepsilon-\lambda_3)\|_{C\CC^{-\gamma}(\rho^{\gamma_1-3\alpha})}(1+\|\phi_\varepsilon + \psi_\varepsilon\|^3_{L^\infty  L^\infty(\rho^{\frac{1}{2}+\alpha})})
\\\lesssim&2^{\alpha \frac{3L}{2}}\|Y^\emptyset_\varepsilon-\lambda_3\|_{C\CC^{-\kappa}(\rho^{\sigma})}(1+\|\psi_\varepsilon\|_{L^\infty L^\infty(\rho^{\frac{1}{2}+\alpha})}^{3})\\\leq&\|Y^\emptyset_\varepsilon-\lambda_3\|_{C\CC^{-\kappa}(\rho^{\sigma})}(1+\|\psi_\varepsilon\|_{L^\infty L^\infty(\rho^{\frac{1}{2}+\alpha})}^{3+\delta}).\endaligned$$

For $\gamma_1>4\alpha$, $k=1,...,m-1$ we have
$$\aligned&\|\rmb{\varepsilon^{\frac{m-3}{2}}C_m^k(- \tttwo{\bar{Y}_\varepsilon}-\ttthree{Y_\varepsilon}  + \phi)^k\psi_\varepsilon^{m-k}}\|_{C\CC^{-\gamma}(\rho^{\frac{3}{2}+\gamma_1})}
\\\lesssim&\varepsilon^{\frac{m-3}{2}}\|(- \tttwo{\bar{Y_\varepsilon}}-\ttthree{Y_\varepsilon}  + \phi_\varepsilon)^k\|_{L^\infty L^\infty(\rho^{\frac{k(3+6\alpha)}{2m}})}\|\psi_\varepsilon\|_{L^\infty L^\infty(\rho^{\frac{3+6\alpha}{2m}})}^{m-k}
\lesssim1,
\endaligned$$
where the constant we omit is independent of $M$ and we used $\varepsilon^{\frac{(m-3)k}{2m}}M\leq1$.

For $k=0,1,...,l$, $l=4,...,m-1$,  we consider the following three cases: for $k\leq \frac{3(m-l)}{m-3}$ we have
$$\aligned&\|\rmb{\varepsilon^{\frac{l-3}{2}}C_l^k(- \tttwo{\bar{Y}_\varepsilon}-\ttthree{Y_\varepsilon}  + \phi_\varepsilon)^k\psi_\varepsilon^{l-k}}\|_{C\CC^{-\gamma}(\rho^{\frac{3}{2}+\gamma_1})}
\\\lesssim&\varepsilon^{\frac{l-3}{2}}\|(- \tttwo{\bar{Y}_\varepsilon}-\ttthree{Y_\varepsilon}  + \phi_\varepsilon)^k\|_{L^\infty L^\infty(\rho^{\frac{k(3+6\alpha)}{2m}})}\|\psi_\varepsilon^{l-k}\|_{L^\infty L^\infty(\rho^{\frac{(3+6\alpha)(m-k)}{2m}})}
\\\lesssim& [\varepsilon^{\frac{m-3}{2}}\|\psi_\varepsilon\|_{L^\infty L^\infty(\rho^{\frac{3+6\alpha}{2m}})}^{m}]^{\frac{l-3}{m-3}}\|\psi_\varepsilon\|_{L^\infty L^\infty(\rho^{\frac{1+2\alpha}{2}})}^{\frac{3(m-l)-km+3k}{m-3}}
\\\lesssim&1+\varepsilon^{\frac{m-3}{2}}\|\psi_\varepsilon\|_{L^\infty L^\infty(\rho^{\frac{3+6\alpha}{2m}})}^{m}+\|\psi_\varepsilon\|_{L^\infty L^\infty(\rho^{\frac{1+2\alpha}{2}})}^{3-k\frac{m-3}{m-l}},\endaligned$$
where we used $-\frac{k(3+6\alpha)}{2m}>-k\frac{1+2\alpha}{2}$ to have the bounds for the weight.
For $l>k> \frac{3(m-l)}{m-3}$ we have
$$\aligned&\|\rmb{\varepsilon^{\frac{l-3}{2}}C_l^k(- \tttwo{\bar{Y}_\varepsilon}-\ttthree{Y_\varepsilon}  + \phi_\varepsilon)^k\psi_\varepsilon^{l-k}}\|_{C\CC^{-\gamma}(\rho^{\frac{3}{2}+\gamma_1})}
\\\lesssim& [\varepsilon^{\frac{m-3}{2}}\|\psi_\varepsilon\|_{L^\infty L^\infty(\rho^{\frac{3+6\alpha}{2m}})}^{m}]^{\frac{l-k}{m}}\varepsilon^{\frac{l-3}{2}-\frac{(m-3)(l-k)}{2m}}
\lesssim1+\varepsilon^{\frac{m-3}{2}}\|\psi_\varepsilon\|_{L^\infty L^\infty(\rho^{\frac{3+6\alpha}{2m}})}^{m},\endaligned$$
where we used $k> \frac{3(m-l)}{m-3}$ to deduce $\frac{l-3}{2}-\frac{(m-3)(l-k)}{2m}>0$.
For $k=l$ the bound holds obviously.

For $\gamma_1>4\alpha$, $\kappa_0=2\alpha$ we have

$$\aligned&\|\rmb{Y^{\emptyset,m-l}_{1,\varepsilon}\preccurlyeq\varepsilon^{\frac{l-3}{2}} (- \tttwo{\bar{Y}_\varepsilon}-\ttthree{Y_\varepsilon}  + \phi_\varepsilon + \psi_\varepsilon)^{l}}\|_{C\CC^{-\gamma}(\rho^{\frac{3}{2}+\gamma_1})}
\\\lesssim&\|Y^{\emptyset,m-l}_{1,\varepsilon}\|_{C\CC^{-\kappa_0-\epsilon}(\rho^{\sigma})}\varepsilon^{\frac{l-3}{2}}(1+\| (\phi_\varepsilon + \psi_\varepsilon)^{l}\|_{C\CC^{\kappa_0+\gamma+\epsilon}(\rho^{\frac{3}{2}+\gamma_1-\sigma})})
\\\lesssim&\varepsilon^{\frac{\epsilon}{2}}\varepsilon^{\frac{l-3+2\kappa_0}{2}}(1+\| (\phi_\varepsilon + \psi_\varepsilon)^{l}\|_{C\CC^{\kappa_0+\gamma+\epsilon}(\rho^{\frac{3}{2}+\gamma_1-\sigma})})
\\\lesssim&1+\varepsilon^{\frac{\epsilon}{2}}(\varepsilon^{\frac{m-3}{2}}\| \psi\|_{L^\infty L^\infty(\rho^{\frac{3+6\alpha}{2m}})}^m)^{\frac{l-3+2\kappa_0}{m-3}}
\|\psi_\varepsilon\|_{C\CC^{2-\gamma}(\rho^{\frac{3}{2}+\gamma_1})}^{\frac{\gamma+\kappa_0+\epsilon}{2-\gamma}}\| \psi_\varepsilon\|_{L^\infty L^\infty(\rho^{\frac{1}{2}+\alpha})}^{\frac{3(m-l)-2m\kappa_0}{m-3}-\frac{\gamma+\kappa_0+\epsilon}{2-\gamma}}
\\\lesssim&1+
\|\psi_\varepsilon\|_{C\CC^{2-\gamma}(\rho^{\frac{3}{2}+\gamma_1})}^{\frac{\gamma+\kappa_0+\epsilon}{2-\gamma}}\| \psi_\varepsilon\|_{L^\infty L^\infty(\rho^{\frac{1}{2}+\alpha})}^{\frac{3(m-l)-2m\kappa_0}{m-3}-\frac{\gamma+\kappa_0+\epsilon}{2-\gamma}}\lesssim1+
\delta_0\|\psi_\varepsilon\|_{C\CC^{2-\gamma}(\rho^{\frac{3}{2}+\gamma_1})}+\| \psi_\varepsilon\|_{L^\infty L^\infty(\rho^{\frac{1}{2}+\alpha})}^{3+\delta},\endaligned$$
where the constant we omit is independent of $M$ and we used $\kappa_0= 2\alpha$ and $\gamma+\epsilon<\alpha$ to have the following bounds for weight
$$\frac{3+6\alpha}{2}\frac{m-l-2\kappa_0}{m-3}-\frac{\kappa_0+\gamma+\epsilon}{2-\gamma}(\frac{3}{2}+\gamma_1)
>(\frac{3(m-l)-2\kappa_0m}{m-3}-\frac{\kappa_0+\gamma+\epsilon}{2-\gamma})(\frac{1}{2}+\alpha).$$

We also have
$$\aligned&\|\rmb{\VV_\leq Y^{\emptyset,m-l}_{1,\varepsilon}\succ\varepsilon^{\frac{l-3}{2}} (- \tttwo{\bar{Y}_\varepsilon}-\ttthree{Y_\varepsilon}  + \phi_\varepsilon + \psi_\varepsilon)^{l}}\|_{C\CC^{-\gamma}(\rho^{\frac{3}{2}+\gamma_1})}
\\\lesssim&\|\VV_\leq Y^{\emptyset,m-l}_{1,\varepsilon}\|_{C\CC^{-\gamma}(\rho^{\gamma+\sigma+\kappa})} \varepsilon^{\frac{l-3}{2}} \|(- \tttwo{\bar{Y}_\varepsilon}-\ttthree{Y_\varepsilon}  + \phi_\varepsilon + \psi_\varepsilon)^{l}\|_{L^\infty L^\infty(\rho^{\frac{3+6\alpha}{2}})}\\\lesssim&2^{\kappa\frac{3}{2} L}\varepsilon^{\frac{\kappa}{2}}[1+(\varepsilon^{\frac{m-3}{2}} \|\psi_\varepsilon\|_{L^\infty L^\infty(\rho^{\frac{3+6\alpha}{2m}})}^m)^{\frac{l-3}{m-3}}\|\psi_\varepsilon\|^{\frac{3(m-l)}{m-3}}_{L^\infty L^\infty(\rho^{\frac{1+2\alpha}{2}})}]\\\lesssim&1+\|\psi_\varepsilon\|^{2+\delta}_{L^\infty L^\infty(\rho^{\frac{1+2\alpha}{2}})},\endaligned$$
where the constant we omit is independent of $M$ and we use $\varepsilon^{\frac{\kappa}{2}}M\leq 1$.  For the terms containing $Y^{\emptyset,l}_{2,\varepsilon}$ we have similar estimate. For $G_\varepsilon^{(m_1)}$ part we have the following estimate

$$\aligned&\|\rmb{\int_0^1G_\varepsilon^{(m_1)}(\varepsilon^{\frac{1}{2}}Y_\varepsilon+\theta\varepsilon^{\frac{1}{2}}v_\varepsilon)
\varepsilon^{\frac{m_1-3}{2}} (- \tttwo{\bar{Y}_\varepsilon}-\ttthree{Y_\varepsilon}  + \phi_\varepsilon + \psi_\varepsilon)^{m_1}\tau^{m_1-3}\frac{(1-\tau)^2}{2!}d\tau}\|_{C\CC^{-\gamma}(\rho^{\frac{3}{2}+\gamma_1})}
\\\lesssim&(\varepsilon^{\frac{m-3}{2}} (1+\|\psi_\varepsilon\|_{L^\infty L^\infty(\rho^{\frac{3+6\alpha}{2m}})}^m)^{\frac{m_1-3}{m-3}}(1+\|\psi_\varepsilon\|^{\frac{3(m-m_1)}{m-3}}_{L^\infty L^\infty(\rho^{\frac{1+2\alpha}{2}})})\\\lesssim&1+\varepsilon^{\frac{m-3}{2}} \|\psi_\varepsilon\|_{L^\infty L^\infty(\rho^{\frac{3+6\alpha}{2m}})}^m+\|\psi_\varepsilon\|_{L^\infty L^\infty (\rho^{\frac{1+2\alpha}{2}})}^3.\endaligned$$

$$\aligned&\|\varepsilon^{\frac{m-3}{2}} \psi_\varepsilon^m\|_{C\CC^{-\gamma}(\rho^{\frac{3}{2}+\gamma_1})}
\lesssim\varepsilon^{\frac{m-3}{2}}\|\psi_\varepsilon\|_{L^\infty L^\infty (\rho^{\frac{3+6\alpha}{2m}})}^m.\endaligned$$
Combining the above estimates and extra estimates in Appendix A.4 we obtain
$$\aligned&\|\psi_\varepsilon\|_{C\CC^{2-\gamma}(\rho^{\frac{3}{2}+\gamma_1})}+\|\psi_\varepsilon\|_{C^{1-\frac{\gamma}{2}}L^\infty(\rho^{\frac{3}{2}+\gamma_1})}
\\\lesssim&1+\|\psi_\varepsilon\|_{C\CC^{1+\alpha}(\rho^{\frac{3}{2}+\gamma'})}+\delta_0\| \psi_\varepsilon\|_{C^{1-\gamma}L^\infty(\rho^{\frac{3}{2}+\gamma_1})}+\|\psi_\varepsilon\|_{L^\infty L^\infty(\rho^{\frac{1}{2}+\alpha})}^{3+\delta}+\| \psi_\varepsilon\|_{C\CC^{\frac{1}{2}+\alpha}(\rho^{\frac{3}{2}+\gamma''})}\\&+(1+\| \psi_\varepsilon\|_{L^\infty L^\infty(\rho^{\frac{1}{2}+\alpha})}^{1+\delta})\|\psi_\varepsilon\|_{C\CC^{\gamma}(\rho^{\frac{1}{2}+2\alpha})}+\delta_0\| \psi_\varepsilon\|_{C\CC^{2-\gamma}(\rho^{\frac{3}{2}+\gamma_1})}\\&+(\| \psi_\varepsilon\|_{L^\infty L^\infty(\rho^{\frac{1}{2}+\alpha})}+1)
\| \psi_\varepsilon\|_{C\CC^{\frac{1}{2}+\gamma+\kappa}(\rho^{1+\gamma_1-\sigma-\alpha})}
+\|\psi_\varepsilon\|_{C\CC^{\alpha}(\rho^{\frac{1}{2}+\gamma_1-2\alpha-\sigma})}\| \psi_\varepsilon\|_{L^\infty L^\infty(\rho^{\frac{1}{2}+\alpha})}^2\\&+\varepsilon^{\frac{m-3}{2}} \|\psi_\varepsilon\|_{L^\infty L^\infty(\rho^{\frac{3+6\alpha}{2m}})}^m .\endaligned$$
Now by Lemma 2.1 we have the following interpolation inequalities:
$$\|\psi_\varepsilon\|_{C\CC^{1+\alpha}(\rho^{\frac{3}{2}+\gamma'})}\lesssim\|\psi_\varepsilon\|_{C\CC^{2-\gamma}
(\rho^{\frac{3}{2}+\gamma_1})}^{\frac{1+\alpha}{2-\gamma}}
\|\psi_\varepsilon\|_{L^\infty L^\infty(\rho^{\frac{1}{2}+\alpha})}^{1-\frac{1+\alpha}{2-\gamma}},$$

$$\| \psi_\varepsilon\|_{C\CC^{\frac{1}{2}+\alpha}(\rho^{1+\gamma_1-\sigma-\alpha})}
\lesssim\|\psi_\varepsilon\|_{C\CC^{2-\gamma}(\rho^{\frac{3}{2}+\gamma_1})}^{\frac{\frac{1}{2}+\alpha}{2-\gamma}}
\|\psi_\varepsilon\|_{L^\infty L^\infty(\rho^{\frac{1}{2}+\alpha})}^{1-\frac{\frac{1}{2}+\alpha}{2-\gamma}},$$

$$\| \psi_\varepsilon\|_{C\CC^{\gamma}(\rho^{\frac{1}{2}+2\alpha})}\lesssim
\|\psi_\varepsilon\|_{C\CC^{2-\gamma}(\rho^{\frac{3}{2}+\gamma_1})}^{\frac{\gamma}{2-\gamma}}
\|\psi_\varepsilon\|_{L^\infty L^\infty(\rho^{\frac{1}{2}+\alpha})}^{1-\frac{\gamma}{2-\gamma}},$$
where we use $\alpha>\frac{\gamma}{2-\gamma}(1+\gamma_1-\alpha)$ to have the bound for weight.
For $\gamma_1>4\alpha$ we have $3\alpha+\frac{\alpha-\alpha^2}{2-\gamma}<\gamma_1(1-\frac{\alpha}{2-\gamma})$, which implies that
$$\aligned\|\psi_\varepsilon\|_{C\CC^{\alpha}(\rho^{\frac{1}{2}+\gamma_1-2\alpha-\sigma})}\lesssim \|\psi_\varepsilon\|_{C\CC^{2-\gamma}(\rho^{\frac{3}{2}+\gamma_1})}^{\frac{\alpha}{2-\gamma}}\|\psi\|_{L^\infty L^\infty(\rho^{\frac{1}{2}+\alpha})}^{1-\frac{\alpha}{2-\gamma}}.\endaligned$$

By the above interpolation inequalities and Young's inequality we have

\begin{equation}\aligned&\|\psi_\varepsilon\|_{C\CC^{2-\gamma}(\rho^{\frac{3}{2}+\gamma_1})}
+\|\psi_\varepsilon\|_{C^{1-\frac{\gamma}{2}}L^\infty(\rho^{\frac{3}{2}+\gamma_1})}
\lesssim1+\|\psi_\varepsilon\|_{L^\infty L^\infty(\rho^{\frac{1}{2}+\alpha})}^{3+\delta}+\varepsilon^{\frac{m-3}{2}} \|\psi_\varepsilon\|_{L^\infty L^\infty(\rho^{\frac{3+6\alpha}{2m}})}^m.\endaligned\end{equation}

\subsection{Bound for $\psi$ in  $L^\infty L^\infty(\rho^{\frac{1}{2}+\alpha})$}

In the following we estimate $\Psi$ in $L^\infty L^\infty(\rho^{\frac{3}{2}+3\alpha})$. Most of the terms are similar as the corresponding terms in Section 4.4. We could use $1+\|\psi_\varepsilon\|_{L^\infty L^\infty(\rho^{\frac{1}{2}+\alpha})}^{2+\delta}+\delta_0\varepsilon^{\frac{m-3}{2}} \|\psi_\varepsilon\|_{L^\infty L^\infty(\rho^{\frac{3+6\alpha}{2m}})}^m$ to control them. We omit the similar terms and only give the different ones.

$$\aligned&\|\rmb{2\lambda_{2,\varepsilon} \VV_\leq Y_\varepsilon\succ(\phi_\varepsilon + \psi_\varepsilon)}\|_{L^\infty L^\infty(\rho^{\frac{3}{2}+3\alpha})}
\lesssim\|2\lambda_{2,\varepsilon} \VV_\leq Y_\varepsilon\|_{C\CC^{\gamma}(\rho^{1+2\alpha})}\|\phi_\varepsilon + \psi_\varepsilon\|_{L^\infty L^\infty(\rho^{\frac{1}{2}+\alpha})}\\\lesssim& 2^{(\frac{1}{2}+\gamma+\kappa)\frac{2L}{3}}(1+\| \psi_\varepsilon\|_{L^\infty L^\infty(\rho^{\frac{1}{2}+\alpha})})\lesssim1+\| \psi_\varepsilon\|_{L^\infty L^\infty(\rho^{\frac{1}{2}+\alpha})}^{1+\delta},\endaligned$$

$$\aligned&\|\rmb{3\VV_\leq\ttwo{Y_\varepsilon}\succ(\phi_\varepsilon+\psi_\varepsilon)}\|_{L^\infty L^\infty(\rho^{\frac{3}{2}+3\alpha})}
\lesssim  \|3\VV_\leq\ttwo{Y_\varepsilon}\|_{C\CC^{\gamma}(\rho^{1+2\alpha})}\|\phi_\varepsilon+\psi_\varepsilon\|_{L^\infty L^\infty(\rho^{\frac{1}{2}+\alpha})}
\\\lesssim & 2^{(1+\gamma+\kappa)L}\|\phi_\varepsilon+\psi_\varepsilon\|_{L^\infty L^\infty(\rho^{\frac{1}{2}+\alpha})}\lesssim1+\|\psi_\varepsilon\|_{L^\infty L^\infty(\rho^{\frac{1}{2}+\alpha})}^{2+\delta},\endaligned$$

$$\aligned&\|\rmb{3\VV_\leq\tone{Y_\varepsilon}\succ(-\tttwo{\bar{Y}_\varepsilon}+\phi_\varepsilon+\psi_\varepsilon)^2}\|_{L^\infty L^\infty(\rho^{\frac{3}{2}+3\alpha})}
\lesssim\|\VV_\leq\tone{Y_\varepsilon}\|_{C\CC^{\gamma/2}(\rho^{\frac{1}{2}+\alpha})}\|-\tttwo{\bar{Y}_\varepsilon}+\phi_\varepsilon+\psi_\varepsilon\|_{L^\infty L^\infty(\rho^{\frac{1}{2}+\alpha})}^2
\\\lesssim&2^{(\frac{1+\alpha+\kappa}{2})\frac{4L}{3}}(1+\| \phi_\varepsilon + \psi_\varepsilon\|_{L^\infty L^\infty(\rho^{\frac{1}{2}+\alpha})}^2)\lesssim1+\| \psi_\varepsilon\|_{L^\infty L^\infty(\rho^{\frac{1}{2}+\alpha})}^{2+\delta},\endaligned$$

$$\aligned&\|\rmb{(Y^\emptyset_\varepsilon-\lambda_3)\preccurlyeq(- \tttwo{\bar{Y}_\varepsilon}-\ttthree{Y_\varepsilon}  + \phi_\varepsilon + \psi_\varepsilon)^3}\|_{L^\infty L^\infty(\rho^{\frac{3}{2}+3\alpha})}
\\\lesssim&\|Y^\emptyset_\varepsilon-\lambda_3\|_{C\CC^{-\frac{1}{2}+\frac{3}{2m}-\frac{3\epsilon}{4}}(\rho^{\sigma})}(1+\| (\phi_\varepsilon + \psi_\varepsilon)^3\|_{C\CC^{\frac{1}{2}-\frac{3}{2m}+\epsilon}(\rho^{\frac{3}{2}+3\alpha-\sigma})})
\\\lesssim&1+\varepsilon^\frac{\epsilon}{2}(\varepsilon^{\frac{m-3}{2}}\|\psi_\varepsilon\|^m_{{L^\infty L^\infty(\rho^{\frac{3+6\alpha}{2m}})}})^{\frac{2}{m}}\|\psi_\varepsilon\|^{\frac{\frac{1}{2}-\frac{3}{2m}
+\epsilon}{2-\gamma}}_{C\CC^{2-\gamma}(\rho^{\frac{3}{2}+\gamma_1})}
\|\psi_\varepsilon\|_{{L^\infty L^\infty(\rho^{\frac{1+2\alpha}{2}})}}^{1-\frac{\frac{1}{2}-\frac{3}{2m}+\epsilon}{2-\gamma}}\\\lesssim&1+\|\psi_\varepsilon\|_{{L^\infty L^\infty(\rho^{\frac{1+2\alpha}{2}})}}^{2+\delta}
.\endaligned$$

$$\aligned&\|\rmb{\VV_\leq(Y_\varepsilon^\emptyset-\lambda_3)\succ(- \tttwo{\bar{Y}_\varepsilon}-\ttthree{Y_\varepsilon}  + \phi_\varepsilon + \psi_\varepsilon)^3}\|_{L^\infty L^\infty(\rho^{\frac{3}{2}+3\alpha})}
\\\lesssim&\|\VV_\leq(Y_\varepsilon^\emptyset-\lambda_3)\|_{C\CC^{\frac{\epsilon}{2}}(\rho^{\frac{1}{2}+\sigma})}(1+\|(\phi + \psi)^3\|_{L^\infty  L^\infty(\rho^{1+3\alpha-\sigma})})
\\\lesssim&2^{\frac{1}{2} \frac{3L}{2}}\varepsilon^\frac{\epsilon}{4}(1+(\varepsilon^{\frac{m-3}{2}}\|\psi_\varepsilon\|^m_{{L^\infty L^\infty(\rho^{\frac{3+6\alpha}{2m}})}})^{\frac{1-2\epsilon}{m-3}})\|\psi_\varepsilon\|_{{L^\infty L^\infty(\rho^{\frac{1+2\alpha}{2}})}}^{3-\frac{m-2m\epsilon}{m-3}}\\\lesssim&1+\|\psi_\varepsilon\|_{{L^\infty L^\infty(\rho^{\frac{1+2\alpha}{2}})}}^{2+\delta},\endaligned$$
where the constant we omit is independent of $M$ and we used Lemma 2.2 to have $$\|\VV_\leq(Y_\varepsilon^\emptyset-\lambda_3)\|_{C\CC^{\frac{\epsilon}{2}}(\rho^{\frac{1}{2}+\sigma})}\lesssim 2^{\frac{1}{2} \frac{3L}{2}}\|Y_\varepsilon^\emptyset-\lambda_3\|_{C\CC^{-\frac{1}{2}+\frac{\epsilon}{2}}(\rho^{\sigma})}
\lesssim\varepsilon^\frac{\epsilon}{4}\varepsilon^{\frac{1}{2}-{\epsilon}},$$
and we used $\epsilon<\frac{\alpha}{1+2\alpha}$ to have the following bound for the weight
$$\frac{3+6\alpha}{2}\frac{1-2\epsilon}{m-3}+\frac{1+2\alpha}{2}(3-\frac{m-2m\epsilon}{m-3})<1+3\alpha.$$

For $k=1,...,m$, we have similar estimates for
$$\|\rmb{\varepsilon^{\frac{m-3}{2}}C_m^k(- \tttwo{\bar{Y}_\varepsilon}-\ttthree{Y_\varepsilon}  + \phi_\varepsilon)^k\psi_\varepsilon^{m-k}}\|_{L^\infty L^\infty(\rho^{\frac{3}{2}+3\alpha})}$$ as in Section 4.4 and for $k=1,...,l$, we have similar estimate as in  Section 4.4 for $$\|\rmb{\varepsilon^{\frac{l-3}{2}}C_l^k(- \tttwo{\bar{Y}_\varepsilon}-\ttthree{Y_\varepsilon}  + \phi_\varepsilon)^k\psi_\varepsilon^{l-k}}\|_{L^\infty L^\infty(\rho^{\frac{3}{2}+3\alpha})}.$$
For the term containing $Y^{\emptyset,m-l}$ we have similar estimates as before. The main change comes from the weight. We give the more complicated one

$$\aligned&\|\rmb{\VV_\leq(Y^{\emptyset,m-l}_{1,\varepsilon})\succ\varepsilon^{\frac{l-3}{2}} (- \tttwo{\bar{Y}_\varepsilon}-\ttthree{Y_\varepsilon}  + \phi_\varepsilon + \psi_\varepsilon)^{l-k}}\|_{L^\infty L^\infty(\rho^{\frac{3}{2}+3\alpha})}
\\\lesssim&\|\VV_\leq Y^{\emptyset,m-l}_{1,\varepsilon}\|_{C\CC^{\gamma}(\rho^{\gamma+\epsilon+\sigma+\kappa_0})} \varepsilon^{\frac{l-3}{2}} \|(- \tttwo{\bar{Y}_\varepsilon}-\ttthree{Y_\varepsilon}  + \phi_\varepsilon + \psi_\varepsilon)^{l}\|_{L^\infty L^\infty(\rho^{\frac{3+6\alpha}{2}-\gamma-\epsilon-\sigma-\kappa_0})}\\\lesssim&2^{(\kappa_0+\gamma+\epsilon)\frac{3}{2} L}\varepsilon^{\frac{\sigma}{2}}[1+(\varepsilon^{\frac{m-3}{2}} \|\psi_\varepsilon\|_{L^\infty L^\infty(\rho^{\frac{3+6\alpha}{2m}})}^m)^{\frac{l-3+2\kappa_0}{m-3}}\|\psi_\varepsilon\|^{\frac{3(m-l)-2m\kappa_0}{m-3}}_{L^\infty L^\infty(\rho^{\frac{1+2\alpha}{2}})}]\\\lesssim&1+\|\psi_\varepsilon\|^{2+\delta}_{L^\infty L^\infty(\rho^{\frac{1+2\alpha}{2}})},\endaligned$$
where we used $\kappa_0=2\alpha$ to have
$$\frac{3+6\alpha}{2}-\gamma-\epsilon-\kappa_0>\frac{3+6\alpha}{2}\frac{l-3+2\kappa_0}{m-3}+
(\frac{1}{2}+\alpha)\frac{3(m-l)-2m\kappa_0}{m-3},$$
which gives the bound for the weight.
Similar estimates also hold for the terms containing $Y^{\emptyset,l}_{2,\varepsilon}$.

For the term $\psi^l$ we combine it with the following term and use (1.7) to have
\begin{equation}\aligned&2\sum_{l=4, l \textrm{ even}}^{m_1-1}\|\rmb{c_{l,\varepsilon}\varepsilon^{\frac{l-3}{2}}\psi_\varepsilon^l}\|_{L^\infty L^\infty(\rho^{\frac{3}{2}+3\alpha})}+2\sum_{l=4, l \textrm{ odd}}^{m_1-1}
\|\rmb{(a_{l,\varepsilon}\wedge0)\varepsilon^{\frac{l-3}{2}}\psi_\varepsilon^l}\|_{L^\infty L^\infty(\rho^{\frac{3}{2}+3\alpha})}\\&+2\|\rmb{\int_0^1G_\varepsilon^{(m_1)}(\varepsilon^{\frac{1}{2}}Y_\varepsilon+\theta\varepsilon^{\frac{1}{2}}v_\varepsilon)
\varepsilon^{\frac{m_1-3}{2}} (- \tttwo{\bar{Y}_\varepsilon}-\ttthree{Y_\varepsilon}  + \phi_\varepsilon + \psi_\varepsilon)^{m_1}\tau^{m_1-3}\frac{(1-\tau)^2}{2!}d\tau}\|_{L^\infty L^\infty(\rho^{\frac{3}{2}+3\alpha})}
\\\leq&\sum_{l=4, l \textrm{ even}}^{m_1-1}|c_{l,\varepsilon}|(\varepsilon^{\frac{m-3}{2}}\|\psi_\varepsilon\|^m_{L^\infty L^\infty(\rho^{\frac{3+6\alpha}{2m}})})^{\frac{l-3}{m-3}}\|\psi_\varepsilon\|^{\frac{3(m-l)}{m-3}}_{L^\infty L^\infty(\rho^{\frac{1+2\alpha}{2}})}\\&+\sum_{l=4, l \textrm{ odd}}^{m_1-1}
|a_{l,\varepsilon}\wedge0|(\varepsilon^{\frac{m-3}{2}}\|\psi_\varepsilon\|^m_{L^\infty L^\infty(\rho^{\frac{3+6\alpha}{2m}})})^{\frac{l-3}{m-3}}\|\psi_\varepsilon\|^{\frac{3(m-l)}{m-3}}_{L^\infty L^\infty(\rho^{\frac{1+2\alpha}{2}})}\\&+C_1\frac{1}{m_1!}(\varepsilon^{\frac{m-3}{2}} \|\psi_\varepsilon\|_{L^\infty L^\infty(\rho^{\frac{3+6\alpha}{2m}})}^m)^{\frac{m_1-3}{m-3}}\|\psi_\varepsilon\|^{\frac{3(m-m_1)}{m-3}}_{L^\infty L^\infty(\rho^{\frac{1+2\alpha}{2}})}\\\leq&(C_{0,\varepsilon}-\delta)\varepsilon^{\frac{m-3}{2}} \|\psi_\varepsilon\|_{L^\infty L^\infty(\rho^{\frac{3+6\alpha}{2m}})}^m+(\lambda_3-\delta)\|\psi_\varepsilon\|_{L^\infty L^\infty (\rho^{\frac{1+2\alpha}{2}})}^3.\endaligned\end{equation}

We have

$$\aligned&\|\psi_\varepsilon\|_{L^\infty L^\infty(\rho^{\frac{1}{2}+\alpha})}^3+\varepsilon^{\frac{m-3}{2}}\|\psi_\varepsilon\|_{L^\infty L^\infty(\rho^{\frac{3+6\alpha}{2m}})}^m
\lesssim1+\|\psi_\varepsilon\|_{L^\infty L^\infty(\rho^{\frac{1}{2}+\alpha})}^{2+\delta}+\delta_0\varepsilon^{\frac{m-3}{2}}\|\psi_\varepsilon\|_{L^\infty L^\infty(\rho^{\frac{3+6\alpha}{2m}})}^m,\endaligned$$
which implies that
$$\aligned&\|\psi_\varepsilon\|_{L^\infty L^\infty(\rho^{\frac{1}{2}+\alpha})}^3+\varepsilon^{\frac{m-3}{2}}\|\psi_\varepsilon\|_{L^\infty L^\infty(\rho^{\frac{3+6\alpha}{2m}})}^m
\lesssim1,\endaligned$$
where the constant we omit is independent of $M$.

\vskip.10in

Now we can prove our main results.
\vskip.10in
\indent\emph{Proof of Theorem 1.1 } By Theorem 3.1 we know that when $\varepsilon\rightarrow0, M_0\rightarrow\infty$, $\mathbb{P}(\Omega_{\varepsilon,M_0})\rightarrow1$. For any $\delta>0$ we  could find $\varepsilon_0$ small enough and $M_0$ large enough such that
for $\varepsilon\leq  \varepsilon_0$ $\mathbb{P}(\Omega_{\varepsilon,M_0})\geq1-\delta$.
On such $\Omega_{\varepsilon,M_0}$ we have for $t\leq T_\varepsilon^M$  the following uniform bounds,
 $$\|\phi_\varepsilon\|_{C_{T_\varepsilon^M}\CC^{\frac{1}{2}+\alpha}(\rho^{\frac{3+6\alpha}{2m}})}
+\|\phi_\varepsilon\|_{C_{T_\varepsilon^M}^{\frac{1}{4}+\frac{\alpha}{2}}L^\infty(\rho^{\frac{3+6\alpha}{2m}})}
\lesssim1,$$
$$\|\vartheta_\varepsilon\|_{C_{T_\varepsilon^M}^{\frac{1+\alpha}{2}}L^\infty(\rho^{\frac{3}{2}+\gamma'})}
+\|\vartheta_\varepsilon\|_{C_{T_\varepsilon^M}\CC^{1+\alpha}(\rho^{\frac{3}{2}+\gamma'})}\lesssim1,$$
 $$\|\psi_\varepsilon\|_{C_{T_\varepsilon^M}\CC^{2-\gamma}(\rho^{\frac{3}{2}+\gamma_1})}
+\|\psi_\varepsilon\|_{C_{T_\varepsilon^M}^{1-\frac{\gamma}{2}}L^\infty(\rho^{\frac{3}{2}+\gamma_1})}
\lesssim1,$$
 $$\aligned&\|\psi_\varepsilon\|_{L_{T_\varepsilon^M}^\infty L^\infty(\rho^{\frac{1}{2}+\alpha})}^3+\varepsilon^{\frac{m-3}{2}}\|\psi_\varepsilon\|_{L_{T_\varepsilon^M}^\infty L^\infty(\rho^{\frac{3+6\alpha}{2m}})}^m
\lesssim1,\endaligned$$
 where all the constants we omit are independent of $M$. Then we could choose $M$ large enough and $\varepsilon$ small satisfying $\varepsilon^{\frac{\epsilon}{2}}M\leq 1$ such that the uniform bounds are smaller than $M$, which implies that $ T_\varepsilon^M=T$ and we have the above  uniform bounds hold with $T_\varepsilon^M$ replaced by $T$ on $\Omega_{\varepsilon,M_0}$.

 Based on the uniform bounds  we obtain compactness of the sequence of approximate solutions $(\phi_{\varepsilon},\vartheta_{\varepsilon},\psi_{\varepsilon})$ in a slightly worse space (see the proof of [GH18, Theorem 6.1]). By Theorem 3.1 we could easily  pass to the limit for the terms similar as in [GH18] in the approximation equation. The terms containing $Y^\emptyset_\varepsilon$  go to zero,  since $Y^\emptyset_\varepsilon-\lambda_3$ approximates to zero in a suitable space on $\Omega_{\varepsilon,M_0}$. For the terms like $\varepsilon^{\frac{m-3}{2}}\psi_\varepsilon^m$ coming from $R_\varepsilon$,  we could use the uniform bounds to obtain that it also goes to zero in $L^\infty_T L^\infty(\rho^{(\frac{1}{2}+\alpha)m})$.

  Finally, we obtain that the limit solutions $(\phi,\vartheta,\theta)$ belong to the spaces where the uniform bounds hold by similar argument as in the proof of  [GH18, Theorem 6.1] and satisfy the same equation as that in [GH18]. As mentioned in Remark 1.2 the limit equation in our case contains more terms. But by using the same technique as in [GH18], the solutions are the unique solutions to the dynamical $\Phi^4_3$ model. As a result,  we have proved the solution $u_\varepsilon$ converges to the solution $u(\lambda)$ to the dynamical $\Phi^4_3$ model in probability.
\vskip.10in

\no\textbf{Acknowledgments}
\vskip.10in
 We would  like to thank Professor Massimiliano Gubinelli for proposing this work to us.

 \appendix
  \renewcommand{\appendixname}{Appendix~\Alph{section}}

  \section{Extra estimates}
\subsection{Extra estimates for $\phi$ in  $C\CC^\alpha(\rho^{\frac{3+6\alpha}{2m}})$}
For other terms in $\Phi$ we have the following estimates:
$$\aligned&\|\rmm{3\VV_>\ttwo{Y_\varepsilon}\succ(\phi_\varepsilon+\psi_\varepsilon)}\|_{C\CC^{-2+\alpha}(\rho^{\frac{3+6\alpha}{2m}})}
+\|\rmm{3\VV_>\ttwo{Y_\varepsilon}\prec(\phi_\varepsilon+\psi_\varepsilon)}\|_{C\CC^{-2+\alpha}(\rho^{\frac{3+6\alpha}{2m}})}
\\\lesssim & \|3\VV_>\ttwo{Y_\varepsilon}\|_{C\CC^{-2+\alpha}(\rho^{-\frac{1}{2}-\alpha+\frac{3+6\alpha}{2m}})}\|\phi_\varepsilon+\psi_\varepsilon\|_{L^\infty L^\infty(\rho^{\frac{1}{2}+\alpha})}
\\\lesssim & 2^{-(1-\alpha-\kappa)L}\|\phi_\varepsilon+\psi_\varepsilon\|_{L^\infty L^\infty(\rho^{\frac{1}{2}+\alpha})},\endaligned$$

$$\aligned&\|\rmm{9(\phi_\varepsilon+\psi_\varepsilon)\prec \VV_>\tttwotwo{Y_\varepsilon}}\|_{C\CC^{-2+\alpha}(\rho^{\frac{3+6\alpha}{2m}})}
\\\lesssim&\|\VV_>\tttwotwo{Y_\varepsilon}\|_{C\CC^{-2+\alpha}(\rho^{-\frac{1}{2}-\alpha+\frac{3+6\alpha}{2m}})}\|\phi_\varepsilon+\psi_\varepsilon\|_{L^\infty L^\infty(\rho^{\frac{1}{2}+\alpha})}
\\\lesssim&2^{-(2-\alpha-\kappa)\frac{L}{2}}\|\tttwotwo{Y_\varepsilon}\|_{C\CC^{-\kappa}(\rho^{\sigma})}\| \phi_\varepsilon + \psi_\varepsilon\|_{L^\infty L^\infty(\rho^{\frac{1}{2}+\alpha})},\endaligned$$

$$\aligned&\|\rmm{\VV_>\tone{Y_\varepsilon}\succ(\ttthree{Y_\varepsilon}(\phi_\varepsilon+\psi_\varepsilon))}
\|_{C\CC^{-2+\alpha}(\rho^{\frac{3+6\alpha}{2m}})}
+\|\rmm{\VV_>\tone{Y_\varepsilon}\prec(\ttthree{Y_\varepsilon}(\phi_\varepsilon+\psi_\varepsilon))}\|_{C\CC^{-2+\alpha}(\rho^{\frac{3+6\alpha}{2m}})}
\\\lesssim&\|\VV_>\tone{Y_\varepsilon}\|_{C\CC^{-2+\alpha}(\rho^{-\frac{1}{2}-\alpha+\frac{3+6\alpha}{2m}-\sigma})}
\|\phi_\varepsilon+\psi_\varepsilon\|_{L^\infty L^\infty(\rho^{\frac{1}{2}+\alpha})}
\lesssim2^{-(\frac{3}{2}-\alpha-\kappa)\frac{2L}{3}}\| \phi_\varepsilon + \psi_\varepsilon\|_{L^\infty L^\infty(\rho^{\frac{1}{2}+\alpha})},\endaligned$$

$$\aligned\|\rmm{6(\phi_\varepsilon+\psi_\varepsilon)\prec\VV_>\ttthreeone{Y_\varepsilon}}\|_{C\CC^{-2+\alpha}(\rho^{\frac{3+6\alpha}{2m}})}
\lesssim&\|\VV_>\ttthreeone{Y_\varepsilon}\|_{C\CC^{-2+\alpha}(\rho^{-\frac{1}{2}-\alpha+\frac{3+6\alpha}{2m}})}\|\phi_\varepsilon+\psi_\varepsilon\|_{L^\infty L^\infty(\rho^{\frac{1}{2}+\alpha})}
\\\lesssim&2^{-(2-\alpha-\kappa)\frac{L}{2}}\| \phi_\varepsilon + \psi_\varepsilon\|_{L^\infty L^\infty(\rho^{\frac{1}{2}+\alpha})},\endaligned$$

$$\aligned\|\rmm{3\VV_>\tone{Y_\varepsilon}\succ(-\tttwo{\bar{Y}_\varepsilon}+\phi_\varepsilon+\psi_\varepsilon)^2}\|_{C\CC^{-2+\alpha}
(\rho^{\frac{3+6\alpha}{2m}})}
\lesssim&\|\VV_>\tone{Y_\varepsilon}\|_{C\CC^{-2+\alpha}(\rho^{-1-2\alpha+\frac{3+6\alpha}{2m}})}\|\tttwo{\bar{Y}_\varepsilon}+\phi_\varepsilon
+\psi_\varepsilon\|_{L^\infty L^\infty(\rho^{\frac{1}{2}+\alpha})}^2
\\\lesssim&2^{-(\frac{3}{2}-\alpha-\kappa)\frac{4L}{3}}(1+\| \phi_\varepsilon + \psi_\varepsilon\|_{L^\infty L^\infty(\rho^{\frac{1}{2}+\alpha})}^2).\endaligned$$

\subsection{Extra estimates for $\phi$ in  $C\CC^{\frac{1}{2}+\alpha}(\rho^{\frac{3+6\alpha}{2m}})$}

In this subsection we consider the extra terms which we omit in Section 3.2.
$$\aligned&\|\rmm{3\VV_>\ttwo{Y_\varepsilon}\succ(\phi_\varepsilon+\psi_\varepsilon)}\|_{C\CC^{-\frac{3}{2}+\alpha}(\rho^{\frac{3+6\alpha}{2m}})}
+\|\rmm{3\VV_>\ttwo{Y_\varepsilon}\prec(\phi_\varepsilon+\psi_\varepsilon)}\|_{C\CC^{-\frac{3}{2}+\alpha}(\rho^{\frac{3+6\alpha}{2m}})}
\\\lesssim & \|3\VV_>\ttwo{Y_\varepsilon}\|_{C\CC^{-\frac{3}{2}+\alpha}(\rho^{-\frac{1}{2}-\alpha+\frac{3+6\alpha}{2m}})}\|\phi_\varepsilon+\psi_\varepsilon\|_{L^\infty L^\infty(\rho^{\frac{1}{2}+\alpha})}
\\\lesssim & 2^{-(\frac{1}{2}-\alpha-\kappa)L}\|\phi_\varepsilon+\psi_\varepsilon\|_{L^\infty L^\infty(\rho^{\frac{1}{2}+\alpha})}\lesssim 1+\|\psi_\varepsilon\|_{L^\infty L^\infty(\rho^{\frac{1}{2}+\alpha})}^\delta,\endaligned$$

$$\aligned&\|\rmm{9(\phi_\varepsilon+\psi_\varepsilon)\prec \VV_>\tttwotwo{Y_\varepsilon}}\|_{C\CC^{-\frac{3}{2}+\alpha}(\rho^{\frac{3+6\alpha}{2m}})}
\\\lesssim&\|\VV_>\tttwotwo{Y_\varepsilon}\|_{C\CC^{-\frac{3}{2}+\alpha}(\rho^{-\frac{1}{2}-\alpha+\frac{3+6\alpha}{2m}})}
\|\phi_\varepsilon+\psi_\varepsilon\|_{L^\infty L^\infty(\rho^{\frac{1}{2}+\alpha})}
\\\lesssim&2^{-(\frac{3}{2}-\alpha-\kappa)\frac{L}{2}}\|\tttwotwo{Y_\varepsilon}\|_{C\CC^{-\kappa}(\rho^{\sigma})}\| \phi_\varepsilon + \psi_\varepsilon\|_{L^\infty L^\infty(\rho^{\frac{1}{2}+\alpha})}\lesssim 1+\|\psi_\varepsilon\|_{L^\infty L^\infty(\rho^{\frac{1}{2}+\alpha})}^\delta,\endaligned$$

$$\aligned&\|\rmm{\VV_>\tone{Y_\varepsilon}\succ(\ttthree{Y_\varepsilon}
(\phi_\varepsilon+\psi_\varepsilon))}\|_{C\CC^{-\frac{3}{2}+\alpha}(\rho^{\frac{3+6\alpha}{2m}})}
+\|\rmm{\VV_>\tone{Y_\varepsilon}\prec(\ttthree{Y_\varepsilon}(\phi_\varepsilon+\psi_\varepsilon))}\|_{C\CC^{-\frac{3}{2}+\alpha}(\rho^{\frac{3+6\alpha}{2m}})}
\\\lesssim&\|\VV_>\tone{Y_\varepsilon}\|_{C\CC^{-\frac{3}{2}+\alpha}(\rho^{-\frac{1}{2}-\alpha+\frac{3+6\alpha}{2m}-\sigma})}\|\phi_\varepsilon+\psi_\varepsilon\|_{L^\infty L^\infty(\rho^{\frac{1}{2}+\alpha})}
\lesssim2^{-(1-\alpha-\kappa)\frac{2L}{3}}\| \phi_\varepsilon + \psi_\varepsilon\|_{L^\infty L^\infty(\rho^{\frac{1}{2}+\alpha})}\\\lesssim& 1+\|\psi_\varepsilon\|_{L^\infty L^\infty(\rho^{\frac{1}{2}+\alpha})}^\delta,\endaligned$$

$$\aligned&\|\rmm{6(\phi_\varepsilon+\psi_\varepsilon)\prec\VV_>\ttthreeone{Y_\varepsilon}}\|_{C\CC^{-\frac{3}{2}+\alpha}(\rho^{\frac{3+6\alpha}{2m}})}
\lesssim\|\VV_>\ttthreeone{Y_\varepsilon}\|_{C\CC^{-\frac{3}{2}+\alpha}(\rho^{-\frac{1}{2}-\alpha+\frac{3+6\alpha}{2m}})}
\|\phi_\varepsilon+\psi_\varepsilon\|_{L^\infty L^\infty(\rho^{\frac{1}{2}+\alpha})}
\\\lesssim&2^{-(\frac{3}{2}-\alpha-\kappa)\frac{L}{2}}\| \phi_\varepsilon + \psi_\varepsilon\|_{L^\infty L^\infty(\rho^{\frac{1}{2}+\alpha})}\lesssim 1+\|\psi_\varepsilon\|_{L^\infty L^\infty(\rho^{\frac{1}{2}+\alpha})}^\delta,\endaligned$$

$$\aligned&\|\rmm{3\VV_>\tone{Y_\varepsilon}\succ(-\tttwo{\bar{Y}_\varepsilon}+\phi_\varepsilon+\psi_\varepsilon)^2}
\|_{C\CC^{-\frac{3}{2}+\alpha}(\rho^{\frac{3+6\alpha}{2m}})}
\lesssim\|\VV_>\tone{Y_\varepsilon}\|_{C\CC^{-\frac{3}{2}+\alpha}(\rho^{-1-2\alpha+\frac{3+6\alpha}{2m}})}\|\tttwo{\bar{Y}_\varepsilon}
+\phi_\varepsilon+\psi_\varepsilon\|_{L^\infty L^\infty(\rho^{\frac{1}{2}+\alpha})}^2
\\\lesssim&2^{-(1-\alpha-\kappa)\frac{4L}{3}}(1+\| \phi_\varepsilon + \psi_\varepsilon\|_{L^\infty L^\infty(\rho^{\frac{1}{2}+\alpha})}^2)\lesssim 1+\|\psi_\varepsilon\|_{L^\infty L^\infty(\rho^{\frac{1}{2}+\alpha})}^\delta.\endaligned$$

\subsection{Extra estimates for $\vartheta$}

In this subsection we consider the extra terms which we omit in Section 3.3.
$$\aligned&\|\rmm{3\VV_>\ttwo{Y_\varepsilon}\prec(\phi_\varepsilon+\psi_\varepsilon)}\|_{C\CC^{-1+\alpha}(\rho^{\frac{3}{2}+\gamma'})}
\\\lesssim & \|3\VV_>\ttwo{Y_\varepsilon}\|_{C\CC^{-\frac{3}{2}}}\|\phi_\varepsilon+\psi_\varepsilon\|_{C\CC^{\frac{1}{2}+\alpha}(\rho^{\frac{3}{2}+\gamma'})}
\\\lesssim & 2^{-(\frac{1}{2}-\kappa)L}\|\phi_\varepsilon+\psi_\varepsilon\|_{C\CC^{\frac{1}{2}+\alpha}(\rho^{\frac{3}{2}+\gamma'})}\lesssim 1+\|\psi_\varepsilon\|_{L^\infty L^\infty(\rho^{\frac{1}{2}+\alpha})}^\delta+\|\psi_\varepsilon\|_{C\CC^{\frac{1}{2}+\alpha}(\rho^{\frac{3}{2}+\gamma'})},\endaligned$$

$$\aligned&\|\rmm{9(\phi_\varepsilon+\psi_\varepsilon)\prec \VV_>\tttwotwo{Y_\varepsilon}}\|_{C\CC^{-1+\alpha}(\rho^{\frac{3+6\alpha}{2m}})}
\\\lesssim&\|\VV_>\tttwotwo{Y_\varepsilon}\|_{C\CC^{-1+\alpha}(\rho^{-\frac{1}{2}-\alpha+\frac{3+6\alpha}{2m}})}\|\phi_\varepsilon+\psi_\varepsilon\|_{L^\infty L^\infty(\rho^{\frac{1}{2}+\alpha})}
\\\lesssim&2^{-(1-\alpha-\kappa)\frac{L}{2}}\|\tttwotwo{Y_\varepsilon}\|_{C\CC^{-\kappa}(\rho^{\sigma})}\| \phi_\varepsilon + \psi_\varepsilon\|_{L^\infty L^\infty(\rho^{\frac{1}{2}+\alpha})}\lesssim 1+\|\psi_\varepsilon\|_{L^\infty L^\infty(\rho^{\frac{1}{2}+\alpha})}^\delta,\endaligned$$

$$\aligned&\|\rmm{\VV_>\tone{Y_\varepsilon}\succ(\ttthree{Y_\varepsilon}(\phi_\varepsilon+\psi_\varepsilon))}\|_{C\CC^{-1+\alpha}(\rho^{\frac{3+6\alpha}{2m}})}
+\|\rmm{\VV_>\tone{Y_\varepsilon}\prec(\ttthree{Y_\varepsilon}(\phi_\varepsilon+\psi_\varepsilon))}\|_{C\CC^{-1+\alpha}(\rho^{\frac{3+6\alpha}{2m}})}
\\\lesssim&\|\VV_>\tone{Y_\varepsilon}\|_{C\CC^{-1+\alpha}(\rho^{-\frac{1}{2}-\alpha+\frac{3+6\alpha}{2m}})}\|\phi_\varepsilon+\psi_\varepsilon\|_{L^\infty L^\infty(\rho^{\frac{1}{2}+\alpha})}
\lesssim2^{-(\frac{1}{2}-\alpha-\kappa)\frac{2L}{3}}\| \phi_\varepsilon + \psi_\varepsilon\|_{L^\infty L^\infty(\rho^{\frac{1}{2}+\alpha})}\\\lesssim& 1+\|\psi_\varepsilon\|_{L^\infty L^\infty(\rho^{\frac{1}{2}+\alpha})}^\delta,\endaligned$$

$$\aligned&\|\rmm{6(\phi_\varepsilon+\psi_\varepsilon)\prec\VV_>\ttthreeone{Y_\varepsilon}}\|_{C\CC^{-1+\alpha}(\rho^{\frac{3+6\alpha}{2m}})}
\lesssim\|\VV_>\ttthreeone{Y_\varepsilon}\|_{C\CC^{-1+\alpha}(\rho^{-\frac{1}{2}-\alpha+\frac{3+6\alpha}{2m}})}\|\phi_\varepsilon+\psi_\varepsilon\|_{L^\infty L^\infty(\rho^{\frac{1}{2}+\alpha})}
\\\lesssim&2^{-(1-\alpha-\kappa)\frac{L}{2}}\| \phi_\varepsilon + \psi_\varepsilon\|_{L^\infty L^\infty(\rho^{\frac{1}{2}+\alpha})}\lesssim 1+\|\psi_\varepsilon\|_{L^\infty L^\infty(\rho^{\frac{1}{2}+\alpha})}^\delta,\endaligned$$

$$\aligned&\|\rmm{3\VV_>\tone{Y_\varepsilon}\succ(-\tttwo{\bar{Y}_\varepsilon}+\phi_\varepsilon+\psi_\varepsilon)^2}\|_{C\CC^{-1+\alpha}
(\rho^{\frac{3+6\alpha}{2m}+\frac{1}{2}+\alpha})}
\\\lesssim&\|\VV_>\tone{Y_\varepsilon}\|_{C\CC^{-1+\alpha}(\rho^{-\frac{1}{2}-\alpha+\frac{3+6\alpha}{2m}})}\|-\tttwo{\bar{Y}_\varepsilon}
+\phi_\varepsilon+\psi_\varepsilon\|_{L^\infty L^\infty(\rho^{\frac{1}{2}+\alpha})}^2
\\\lesssim&2^{-(\frac{1}{2}-\alpha-\kappa)\frac{4L}{3}}(1+\| \phi_\varepsilon + \psi_\varepsilon\|_{L^\infty L^\infty(\rho^{\frac{1}{2}+\alpha})}^2)\lesssim 1+\|\psi_\varepsilon\|_{L^\infty L^\infty(\rho^{\frac{1}{2}+\alpha})}^{1+\delta}.\endaligned$$

\subsection{Extra bounds for $\psi$ in $C\CC^{2-\gamma}(\rho^{\frac{3}{2}+\gamma_1})$}
By (4.4) we have

$$\aligned&\|\rmb{3\ttwo{Y_\varepsilon}\circ\psi_\varepsilon}\|_{C\CC^{-\gamma}(\rho^{\frac{3}{2}+\gamma_1})}
+\|\rmb{3\ttwo{Y_\varepsilon}\circ\vartheta_\varepsilon}\|_{C\CC^{-\gamma}(\rho^{\frac{3}{2}+\gamma_1})}
\lesssim\|\psi_\varepsilon\|_{C\CC^{1+\alpha}(\rho^{\frac{3}{2}+\gamma'})}+\|\vartheta_\varepsilon\|_{C\CC^{1+\alpha}(\rho^{\frac{3}{2}+\gamma'})}
\\\lesssim&1+\|\psi_\varepsilon\|_{C\CC^{1+\alpha}(\rho^{\frac{3}{2}+\gamma'})}+\delta_0\| \psi_\varepsilon\|_{C^{1-\gamma}L^\infty(\rho^{\frac{3}{2}+\gamma_1})}+\|\psi_\varepsilon\|_{L^\infty L^\infty(\rho^{\frac{1}{2}+\alpha})}^{1+\delta},\endaligned$$

$$\aligned&\|\rmb{[9\ttwo{Y_\varepsilon}\circ((- \tttwo{\bar{Y}_\varepsilon}-\ttthree{Y_\varepsilon}  + \phi_\varepsilon + \psi_\varepsilon)\Prec \tttwo{Y_\varepsilon})-9\ttwo{Y_\varepsilon}\circ((- \tttwo{\bar{Y}_\varepsilon}-\ttthree{Y_\varepsilon}  + \phi_\varepsilon + \psi_\varepsilon)\prec \tttwo{Y_\varepsilon})]}\|_{C\CC^{-\gamma}(\rho^{\frac{3}{2}+\gamma_1})}\\\lesssim&
\|\ttwo{Y_\varepsilon}\|_{C\CC^{-1-\kappa}(\rho^{\sigma})}\|\tttwo{Y_\varepsilon}\|_{C\CC^{1-\kappa}(\rho^{\sigma})}\|- \tttwo{\bar{Y}_\varepsilon}-\ttthree{Y_\varepsilon}  + \phi_\varepsilon + \psi_\varepsilon\|_{C^{\frac{\alpha+\kappa}{2}}L^\infty(\rho^{\frac{3}{2}+\gamma''})}\lesssim
1+\| \psi_\varepsilon\|_{C^{\frac{\alpha+\kappa}{2}}L^\infty(\rho^{\frac{3}{2}+\gamma''})}\\\lesssim&1+\delta_0\| \psi_\varepsilon\|_{C^{1-\gamma}L^\infty(\rho^{\frac{3}{2}+\gamma_1})}+\|\psi_\varepsilon\|_{L^\infty L^\infty(\rho^{\frac{1}{2}+\alpha})},\endaligned$$

$$\aligned&\|\rmb{-9\mathrm{com}(- \tttwo{\bar{Y}_\varepsilon}-\ttthree{Y_\varepsilon}  + \phi_\varepsilon + \psi_\varepsilon,\tttwo{Y_\varepsilon},\ttwo{Y_\varepsilon})}\|_{C\CC^{-\gamma}(\rho^{\frac{3}{2}+\gamma_1})}
\\\lesssim&1+\| \phi_\varepsilon + \psi_\varepsilon\|_{C\CC^{\frac{1}{2}}(\rho^{\frac{3}{2}+\gamma''})}\lesssim1+\|\psi_\varepsilon\|_{L^\infty L^\infty(\rho^{\frac{1}{2}+\alpha})}^{\delta}+\| \psi_\varepsilon\|_{C\CC^{\frac{1}{2}}(\rho^{\frac{3}{2}+\gamma''})},\endaligned$$

$$\aligned&\|\rmb{6\tone{Y_\varepsilon}\circ(\ttthree{Y_\varepsilon}\preccurlyeq(\phi_\varepsilon+\psi_\varepsilon))}\|_{C\CC^{-\gamma}
(\rho^{\frac{3}{2}+\gamma_1})}\lesssim\| \phi_\varepsilon + \psi_\varepsilon\|_{C\CC^{\frac{1}{2}+\alpha}(\rho^{\frac{3}{2}+\gamma''})}\\\lesssim&1+\|\psi_\varepsilon\|_{L^\infty L^\infty(\rho^{\frac{1}{2}+\alpha})}^{\delta}+\| \psi_\varepsilon\|_{C\CC^{\frac{1}{2}+\alpha}(\rho^{\frac{3}{2}+\gamma''})},\endaligned$$

$$\aligned&\|\rmb{6\mathrm{com}(\phi_\varepsilon+\psi_\varepsilon,\ttthree{Y_\varepsilon},\tone{Y_\varepsilon})}\|_{C\CC^{-\gamma}
(\rho^{\frac{3}{2}+\gamma_1})}\lesssim\| \phi_\varepsilon + \psi_\varepsilon\|_{C\CC^{\frac{1}{2}}(\rho^{\frac{3}{2}+\gamma''})}\lesssim1+\|\psi_\varepsilon\|_{L^\infty L^\infty(\rho^{\frac{1}{2}+\alpha})}^{\delta}+\| \psi_\varepsilon\|_{C\CC^{\frac{1}{2}}(\rho^{\frac{3}{2}+\gamma''})},\endaligned$$
and

$$\aligned&\|\rmb{(- \tttwo{\bar{Y}_\varepsilon}-\ttthree{Y_\varepsilon}  + \phi_\varepsilon)^3}+\rmb{3(- \tttwo{\bar{Y}_\varepsilon}-\ttthree{Y_\varepsilon}  + \phi_\varepsilon)^2\psi}+\rmb{3(- \tttwo{\bar{Y}_\varepsilon}-\ttthree{Y_\varepsilon}  + \phi_\varepsilon)\psi_\varepsilon^2}\|_{C\CC^{-\gamma}(\rho^{\frac{3}{2}+\gamma_1})}\\\lesssim&1+\| \psi_\varepsilon\|_{L^\infty L^\infty(\rho^{\frac{1}{2}+\alpha})}^2.\endaligned$$

Since $\gamma_1>4\alpha$, we have
$$\aligned&\|\rmb{3\VV_\leq\ttwo{Y_\varepsilon}\succ(\phi_\varepsilon+\psi_\varepsilon)}\|_{C\CC^{-\gamma}(\rho^{\frac{3}{2}+\gamma_1})}
\lesssim  \|3\VV_\leq\ttwo{Y_\varepsilon}\|_{C\CC^{-\gamma}(\rho^{1+\gamma_1-\alpha})}\|\phi_\varepsilon+\psi_\varepsilon\|_{L^\infty L^\infty(\rho^{\frac{1}{2}+\alpha})}
\\\lesssim & 2^{(1-\gamma+\kappa)L}\|\phi_\varepsilon+\psi_\varepsilon\|_{L^\infty L^\infty(\rho^{\frac{1}{2}+\alpha})}\lesssim1+\|\psi_\varepsilon\|_{L^\infty L^\infty(\rho^{\frac{1}{2}+\alpha})}^{2+\delta},\endaligned$$

$$\aligned&\|\rmb{3\VV_\leq\ttwo{Y_\varepsilon}\prec(\phi_\varepsilon+\psi_\varepsilon)}\|_{C\CC^{-\gamma}(\rho^{\frac{3}{2}+\gamma_1})}\lesssim  \|3\VV_\leq\ttwo{Y_\varepsilon}\|_{C\CC^{\gamma}(\rho^{1+\alpha+\sigma})}\|\phi_\varepsilon+\psi_\varepsilon\|_{C\CC^{\gamma}(\rho^{\frac{1}{2}+2\alpha})}
\\\lesssim & 2^{(1+\gamma+\kappa)L}\|\phi_\varepsilon+\psi_\varepsilon\|_{C\CC^{\gamma}(\rho^{\frac{1}{2}+2\alpha})}\lesssim 1+\|\psi_\varepsilon\|_{L^\infty L^\infty(\rho^{\frac{1}{2}+\alpha})}^{1+\delta}+\|\psi_\varepsilon\|_{L^\infty L^\infty(\rho^{\frac{1}{2}+\alpha})}^{1+\delta}\|\psi_\varepsilon\|_{C\CC^{\gamma}(\rho^{\frac{1}{2}+2\alpha})},\endaligned$$

$$\aligned&\|\rmb{9(\phi_\varepsilon+\psi_\varepsilon)\prec \VV_\leq\tttwotwo{Y_\varepsilon}}\|_{C\CC^{-\gamma}(\rho^{\frac{3}{2}+\gamma_1})}
\lesssim\|\VV_\leq\tttwotwo{Y_\varepsilon}\|_{C\CC^{-\gamma}(\rho^{1+\gamma_1-\alpha})}\|\phi_\varepsilon+\psi_\varepsilon\|_{L^\infty L^\infty(\rho^{\frac{1}{2}+\alpha})}
\\\lesssim&2^{\kappa\frac{L}{2}}\| \phi_\varepsilon + \psi_\varepsilon\|_{L^\infty L^\infty(\rho^{\frac{1}{2}+\alpha})}\lesssim1+\| \psi_\varepsilon\|_{L^\infty L^\infty(\rho^{\frac{1}{2}+\alpha})}^{1+\delta},\endaligned$$

$$\aligned&\|\rmb{9(\phi_\varepsilon+\psi_\varepsilon)\succcurlyeq \tttwotwo{Y_\varepsilon}}\|_{C\CC^{-\gamma}(\rho^{\frac{3}{2}+\gamma_1})}
\lesssim\|\phi_\varepsilon+\psi_\varepsilon\|_{C\CC^{\gamma+\kappa}(\rho^{\frac{3}{2}+\gamma})}
\lesssim1+\|\psi_\varepsilon\|_{C\CC^{\gamma+\kappa}(\rho^{\frac{3}{2}+\gamma})},\endaligned$$

$$\aligned&\|\rmb{\VV_\leq\tone{Y_\varepsilon}\succ(\ttthree{Y_\varepsilon}(\phi_\varepsilon+\psi_\varepsilon))}
\|_{C\CC^{-\gamma}(\rho^{\frac{3}{2}+\gamma_1})}
\lesssim\|\VV_\leq\tone{Y_\varepsilon}\|_{C\CC^{-\gamma}(\rho^{1+\gamma_1-\alpha-\sigma})}\|\phi_\varepsilon+\psi_\varepsilon\|_{L^\infty L^\infty(\rho^{\frac{1}{2}+\alpha})}
\\\lesssim&2^{(\frac{1}{2}-\gamma+\kappa)\frac{2L}{3}}\| \phi_\varepsilon + \psi_\varepsilon\|_{L^\infty L^\infty(\rho^{\frac{1}{2}+\alpha})}\lesssim1+\|  \psi_\varepsilon\|_{L^\infty L^\infty(\rho^{\frac{1}{2}+\alpha})}^{1+\delta},\endaligned$$

$$\aligned&\|\rmb{\VV_\leq\tone{Y_\varepsilon}\prec(\ttthree{Y_\varepsilon}(\phi_\varepsilon+\psi_\varepsilon))}\|_{C\CC^{-\gamma}
(\rho^{\frac{3}{2}+\gamma_1})}
\lesssim\|\VV_\leq\tone{Y_\varepsilon}\|_{C\CC^{-\frac{1}{2}+\gamma+\kappa}(\rho^{\alpha+2\kappa})}
\|\phi_\varepsilon+\psi_\varepsilon\|_{C\CC^{\frac{1}{2}+\gamma+\kappa}(\rho^{\frac{3}{2}+\gamma''})}
\\\lesssim&1+\|  \psi_\varepsilon\|_{L^\infty L^\infty(\rho^{\frac{1}{2}+\alpha})}^{\delta}+\|  \psi_\varepsilon\|_{L^\infty L^\infty(\rho^{\frac{1}{2}+\alpha})}^{\delta}\|\psi_\varepsilon\|_{C\CC^{\frac{1}{2}+\gamma+\kappa}(\rho^{\frac{3}{2}+\gamma''})},\endaligned$$

$$\aligned\|\rmb{6(\phi_\varepsilon+\psi_\varepsilon)\prec\VV_\leq\ttthreeone{Y_\varepsilon}}\|_{C\CC^{-\gamma}(\rho^{\frac{3}{2}+\gamma_1})}
\lesssim&\|\VV_\leq\ttthreeone{Y_\varepsilon}\|_{C\CC^{-\gamma}(\rho)}\|\phi_\varepsilon+\psi_\varepsilon\|_{L^\infty L^\infty(\rho^{\frac{1}{2}+\alpha})}
\\\lesssim&2^{\alpha\frac{L}{2}}\| \phi_\varepsilon + \psi_\varepsilon\|_{L^\infty L^\infty(\rho^{\frac{1}{2}+\alpha})}\lesssim1+\| \psi_\varepsilon\|_{L^\infty L^\infty(\rho^{\frac{1}{2}+\alpha})}^{1+\delta},\endaligned$$

$$\aligned\|\rmb{6(\phi_\varepsilon+\psi_\varepsilon)\succcurlyeq\ttthreeone{Y_\varepsilon}}\|_{C\CC^{-\gamma}(\rho^{\frac{3}{2}+\gamma_1})}
\lesssim&\|\phi_\varepsilon+\psi_\varepsilon\|_{C\CC^{\gamma+\kappa}(\rho^{\frac{3}{2}+\gamma''})}
\\\lesssim&1+\|\psi_\varepsilon\|_{C\CC^{\gamma+\kappa}(\rho^{\frac{3}{2}+\gamma''})},\endaligned$$

$$\aligned&\|\rmb{3\VV_\leq\tone{Y_\varepsilon}\succ(-\tttwo{\bar{Y}_\varepsilon}+\phi_\varepsilon+\psi_\varepsilon)^2}\|_{C\CC^{-\gamma}
(\rho^{\frac{3}{2}+\gamma_1})}
\lesssim\|\VV_\leq\tone{Y_\varepsilon}\|_{C\CC^{-\gamma}(\rho^{\frac{1}{2}+\gamma_1-2\alpha})}
\|-\tttwo{\bar{Y}_\varepsilon}+\phi_\varepsilon+\psi_\varepsilon\|_{L^\infty L^\infty(\rho^{\frac{1}{2}+\alpha})}^2
\\\lesssim&2^{(\frac{1}{2}+\alpha)\frac{4L}{3}}(1+\| \phi_\varepsilon + \psi_\varepsilon\|_{L^\infty L^\infty(\rho^{\frac{1}{2}+\alpha})}^2)\lesssim1+\| \psi_\varepsilon\|_{L^\infty L^\infty(\rho^{\frac{1}{2}+\alpha})}^{2+\delta},\endaligned$$

$$\aligned&\|\rmb{3\tone{Y_\varepsilon}\preccurlyeq(-\tttwo{\bar{Y}_\varepsilon}+\phi_\varepsilon+\psi_\varepsilon)^2}
\|_{C\CC^{-\gamma}(\rho^{\frac{3}{2}+\gamma_1})}
\lesssim\|(-\tttwo{\bar{Y}_\varepsilon}+\phi_\varepsilon+\psi_\varepsilon)^2\|_{C\CC^{\frac{1}{2}+\gamma+\kappa}(\rho^{\frac{3}{2}+\gamma_1-\sigma})}
\\\lesssim&(1+\| \phi_\varepsilon + \psi_\varepsilon\|_{L^\infty L^\infty(\rho^{\frac{1}{2}+\alpha})})(1+\| \phi + \psi_\varepsilon\|_{C\CC^{\frac{1}{2}+\gamma+\kappa}(\rho^{1+\gamma_1-\sigma-\alpha})})\\\lesssim&(1+\| \psi_\varepsilon\|_{L^\infty L^\infty(\rho^{\frac{1}{2}+\alpha})})
(1+\| \psi_\varepsilon\|_{L^\infty L^\infty(\rho^{\frac{1}{2}+\alpha})}^\delta+\| \psi_\varepsilon\|_{C\CC^{\frac{1}{2}+\gamma+\kappa}(\rho^{1+\gamma_1-\sigma-\alpha})}).\endaligned$$

      \section{Global well-posedness for smooth noise case}

\begin{proposition}\label{prop:42d}
Let $T>0, C_0>0$,  $\eta\in C^{\infty}([0,T]\times\mathbb{T}^d)$ and $\varphi_{0}\in C^{\infty}(\mathbb{T}^d)$. There exists a unique classical solution $\varphi\in C^{\infty}([0,T]\times \mathbb{T}^{d})$  to
\begin{equation}
\LL \varphi  + C_0\varphi^m =-G(\varphi) +\eta,\qquad\varphi(0)=\varphi_{0}.
\end{equation}
\end{proposition}

\begin{proof} The existence and uniqueness of weak solutions can be obtained by monotonicity argument.
We then prove a priori estimate in $L^p$. We test by $\varphi^{2p-1}$ to obtain
$$\aligned
&\frac{1}{2p}\partial_{t}\int_{\mathbb{T}^{d}} |\varphi|^{2p}dx+(2p-1)\int_{\mathbb{T}^{d}} |\varphi|^{2p-2}|\nabla\varphi|^2\,dx+C_0\int_{\mathbb{T}^{d}} |\varphi|^{2p+m-1}dx
\\\leq&\int_{\mathbb{T}^{d}} |G(\varphi)|\varphi^{2p-1}dx+\int_{\mathbb{T}^{d}} |\eta||\varphi|^{2p-1}dx.\endaligned$$
By our assumption on $G(\varphi)$ we know that for $\kappa>0$
$|G(\varphi)|\leq C+(C_0-\kappa)|\varphi|^{m}$, which implies the right hand side of the above inequality can be controlled by
 $$\aligned
(C_0-\kappa)\int_{\mathbb{T}^{d}} |\varphi|^{2p+m-1}dx+C\int_{\mathbb{T}^{d}} |\varphi|^{2p-1}dx.\endaligned$$
Then we have
$$\aligned
&\frac{1}{2p}\partial_{t}\int_{\mathbb{T}^{d}} |\varphi|^{2p}dx+(2p-1)\int_{\mathbb{T}^{d}} |\varphi|^{2p-2}|\nabla\varphi|^2\,dx+C_0\int_{\mathbb{T}^{d}} |\varphi|^{2p+m-1}dx\leq C_{T,p}.\endaligned$$
By using Sobolev embedding we know that $\varphi, \varphi^m,  G(\varphi)\in L_T^\infty \CC^{-\alpha}(\mathbb{T}^{d})$ for $\alpha>0$ and small enough. Then using Lemma 2.3 on the torus case, we know that $\varphi\in L_T^\infty \CC^{2-\alpha}(\mathbb{T}^d)\cap C_T^{(2-\alpha)/2}L^\infty(\mathbb{T}^d)$.
Then by simple calculation we obtain that  $ \varphi^m,  G(\varphi)\in L_T^\infty \CC^{\alpha}(\mathbb{T}^{d})$, which implies that $\varphi\in L_T^\infty \CC_T^{2+\alpha}(\mathbb{T}^d)\cap C^{(2+\alpha)/2}L^\infty(\mathbb{T}^d)$ is a classical solution to (A.1).

\end{proof}

\begin{proposition}
Let $T>0, C_0>0$ and let $\rho$ be a polynomial weight,   $\eta\in C_T\CC^\gamma(\rho^{\frac{3}{2}+\gamma_1})\cap L_T^\infty L^\infty(\rho^{\frac{3+6\alpha}{2}}) $ and $\varphi_{0}\in \CC^{2+\gamma}(\rho^{\frac{3}{2}+\gamma_1})\cap L^\infty(\rho^{\frac{3+6\alpha}{2m}})$. There exists a unique classical solution $\varphi\in C_T\CC^{2+\gamma}(\rho^{\frac{3}{2}+\gamma_1})\cap C_T^1L^\infty(\rho^{\frac{3}{2}+\gamma_1})\cap L_T^\infty L^\infty (\rho^{\frac{3+6\alpha}{2m}})$  to
\begin{equation}
\LL \varphi  + C_0\varphi^m =-G(\varphi) +\eta,\qquad\varphi(0)=\varphi_{0}.
\end{equation}
\end{proposition}
\begin{proof}
Consider the periodization $\eta_M$ on $\mathbb{T}^d_M$. By Proposition A.1 there exists a classical solution $\varphi_M$ to (A.2) with $\eta$ replaced by $\eta_M$. In the following we obtain the uniform estimates for $\varphi_M$. Since the estimate is independent of $M$, we omit $M$ for simplicity. By similar argument as in the proof of Lemma 2.4 we have
$$\aligned C_0\|\varphi\|_{L_T^{\infty}L^{\infty}(\rho^{\frac{3+6\alpha}{2m}})}^m \leq C_0\|\varphi_{0}\|_{L^{\infty}(\rho^{\frac{3+6\alpha}{2m}})}^m+ \|G(\varphi)\|_{L_T^{\infty}L^{\infty}(\rho^{\frac{3+6\alpha}{2}})}+\|\eta\|_{L_T^{\infty}L^{\infty}(\rho^{\frac{3+6\alpha}{2}})}.\endaligned$$
By Assumption 1 in introduction we know that
$$\|G(\varphi)\|_{L_T^{\infty}L^{\infty}(\rho^{\frac{3+6\alpha}{2}})}\leq C_\delta+(\frac{C_1}{m_1!}+\delta)\|\varphi\|_{L_T^{\infty}L^{\infty}(\rho^{\frac{3+6\alpha}{2m}})}^{m_1}\leq C_\delta+(C_0-\delta)\|\varphi\|_{L_T^{\infty}L^{\infty}(\rho^{\frac{3+6\alpha}{2m}})}^{m},$$
where for $m_1<m$ we can use Young's inequality and for $m_1=m$ we use (1.7) in the last inequality. Now we have
$$\aligned \|\varphi\|_{L_T^{\infty}L^{\infty}(\rho^{\frac{3+6\alpha}{2m}})}^m\lesssim\|\varphi_{0}\|_{L^{\infty}(\rho^{\frac{3+6\alpha}{2m}})}^m
+\|\eta\|_{L_T^{\infty}L^{\infty}(\rho^{\frac{3+6\alpha}{2}})}.\endaligned$$
Moreover, we have
$$\aligned&\|G(\varphi)\|_{C_T\CC^{\gamma}(\rho^{\frac{3}{2}+\gamma_1})}+\| \varphi^m\|_{C\CC^{\gamma}(\rho^{\frac{3}{2}+\gamma_1})}\lesssim (1+\|\varphi\|_{L^\infty_T L^\infty(\rho^{\frac{3+6\alpha}{2m}})}^{m-1})\|\varphi\|_{C_T\CC^\gamma(\rho^{\frac{3+6\alpha}{2m}+\gamma_1-3\alpha})}\\\lesssim& (1+\|\varphi\|_{L_T^\infty L^\infty(\rho^{\frac{3+6\alpha}{2m}})}^{m-\gamma})\|\varphi\|_{C_T\CC^{2+\gamma}(\rho^{\frac{3}{2}+\gamma_1})}^\gamma
\\\lesssim& \|\varphi_{0}\|_{L^{\infty}(\rho^{\frac{3+6\alpha}{2m}})}^{m+1}
+\|\eta\|_{L_T^{\infty}L^{\infty}(\rho^{\frac{3+6\alpha}{2}})}^2+1+\delta_0\|\varphi\|_{C_T\CC^{2+\gamma}(\rho^{\frac{3}{2}+\gamma_1})},\endaligned$$
for $0<\delta_0<1$. Now we have the following uniform estimate by Lemma 2.3
$$\aligned&\| \varphi\|_{C_T\CC^{2+\gamma}(\rho^{\frac{3}{2}+\gamma_1})}+\| \varphi\|_{C^{1}L^\infty(\rho^{\frac{3}{2}+\gamma_1})}
\\\lesssim& \|\varphi_{0}\|_{L^{\infty}(\rho^{\frac{3+6\alpha}{2m}})}^{m+1}+\|\varphi_{0}\|_{\CC^{2+\gamma}(\rho^{\frac{3}{2}+\gamma_1})}
+\|\eta\|_{L_T^{\infty}L^{\infty}(\rho^{\frac{3+6\alpha}{2}})}^2+\|\eta\|_{C_T\CC^{\gamma}(\rho^{\frac{3}{2}+\gamma_1})}+1.\endaligned$$
Since the constant we omit in the above estimate is independent of $M$, we can obtain compactness of the approximation sequence $\varphi_M$ in a slightly worse space, which allows to pass to the limit in the approximation equation. We can also obtain the solution belong to the spaces where the uniform bounds hold. For the uniqueness in the above weighted space we can also choose time dependent weight $\pi(t,x)=\exp(-t\rho^{-2b}(x))$ for $\rho=\langle x\rangle^{-1}$ and $b\in (0,1/2)$ as in [GH18].  Take two different solutions $\varphi_1$ and $\varphi_2$ starting from the same initial data $\varphi_0$ and satisfying the above bounds. Set $u:=\varphi_1-\varphi_2$.
Then we have
$$\LL u  + C_0(\varphi_1^m-\varphi_1^m) =-G(\varphi_1)+G(\varphi_2).$$
Now we take inner product with $\pi^2 u$ in $L^2$ and use $\partial_t \pi=-\pi \rho^{-2b}$ to have
$$\frac{1}{2}\partial_t\|u\|_{L^2(\pi)}^2+\|\nabla u\|_{L^2(\pi)}^2 +\|\rho^{-b}u\|_{L^2(\pi)}^2  \leq \langle-G(\varphi_1)+G(\varphi_2),u\pi^2\rangle+C_\delta \|u\|_{L^2(\pi)}^2+\delta\|\nabla u\|_{L^2(\pi)}^2,$$
where we use $|\frac{\nabla \pi}{\pi}|\lesssim1$.
$$\aligned|\langle-G(\varphi_1)+G(\varphi_2),u\pi^2\rangle|\lesssim & (1+\|\varphi_1\|_{L^{\infty}(\rho^{\frac{b}{m_1-1}})}^{m_1-1}+\|\varphi_2\|_{L^{\infty}(\rho^{\frac{b}{m_1-1}})}^{m_1-1})\|\rho^{-\frac{b}{2}}u\|_{L^2(\pi)}^2
\\\lesssim & \delta\|\rho^{-b}u\|_{L^2(\pi)}^2+C_\delta\|u\|_{L^2(\pi)}^2.\endaligned$$
Now the uniqueness follows by Gronwall's Lemma.
\end{proof}

\end{document}